\documentclass[reqno, 11pt]{amsart}
\usepackage[mathscr]{eucal}
\usepackage[all]{xy}
\usepackage{mathrsfs}
\usepackage{xypic}
\usepackage{amsfonts,mathtools}
\usepackage{amsmath}
\usepackage{amsthm}
\usepackage{amssymb}
\usepackage{latexsym}
\usepackage{tabularx}
\usepackage{graphicx}
\usepackage{tikz}
\usepackage{tikz-cd}
\usepackage[colorlinks=true, citecolor=blue, urlcolor=blue, linkcolor=blue]{hyperref}
\usepackage[active]{srcltx}
\usetikzlibrary{positioning}

\usepackage[text={168mm,240mm},centering]{geometry}
\input diagxy
\input xy
\newtheorem{thm}{Theorem}[section]
\newtheorem{cor}[thm]{Corollary}
\newtheorem{lem}[thm]{Lemma}
\newtheorem{prop}[thm]{Proposition}

\theoremstyle{definition}
\newtheorem{defn}[thm]{Definition}

\newtheorem{rem}[thm]{Remark}

\newtheorem{prob}{Problem}[section]
\newtheorem{ex}[thm]{Example}

\newcommand{\PV}{\mathbb{P}}

\newcommand{\coker}{\mathrm{coker}}

\allowdisplaybreaks

\numberwithin{equation}{section} %----------------------------------------------------------------------------------------------------------------------
\begin{document}

\title[]{Bott--Chern hypercohomology and bimeromorphic invariants}

\author[S. Yang]{Song Yang}
\address{Center for Applied Mathematics, Tianjin University, Tianjin 300072, P. R.  China}%
\email{syangmath@tju.edu.cn}%

\author[X. Yang]{Xiangdong Yang}
\address{Department of mathematics, Lanzhou University, Lanzhou 730000, P.R. China}
\email{yangxd@lzu.edu.cn}

\date{\today}
%----------------------------------------------------------------------------------------------------------------------

\begin{abstract}
The aim of this article is to study the geometry of Bott--Chern hypercohomology from the bimeromorphic point of view.
We construct some new bimeromorphic invariants involving the cohomology for the sheaf of germs of pluriharmonic functions, the truncated holomorphic de Rham cohomology, and the de Rham cohomology.
To define these invariants, using a sheaf-theoretic approach, we establish a blow-up formula together with a canonical morphism for the Bott--Chern hypercohomology.
In particular, we compute the invariants of some compact complex threefolds, such as Iwasawa manifolds and quintic threefolds.
\end{abstract}



% ----------------------------------------------------------------------------------------------------------------------
\maketitle
%----------------------------------------------------------------------------------------------------------------------
\setcounter{tocdepth}{1}
\tableofcontents

%====================================================================

\section{Introduction}
\subsection{Background}
In complex geometry, one of fundamental problems is to classify compact complex manifolds up to bimeromorphic equivalence and to find nice models in every equivalence class (cf. \cite{CP99}).
In general, it is impossible to describe the bimeromorphic equivalence classes completely for the lacking of bimeromorphic invariants.
For a meaningful classification, the first thing is to determine bimeromorphic invariants of compact complex manifolds as many as possible, or to understand which geometric property of a compact complex manifold admits the bimeromorphic invariance.
Among these, the Kodaira dimension plays an important role in the bimeromorphic classification of compact complex manifolds.
Especially, in the sense of minimal model program, compact complex surfaces can be divided into ten classes by the Kodaira dimension called the Enriques--Kodaira classification (cf. \cite[\S\,VI]{BHP04}).
According to the celebrated weak factorization theorem by Abramovich--Karu--Matsuki--W{\l}odarczyk \cite{AKMW02}, every bimeromorphic map between compact complex manifolds can be factored into a finite sequence of blow-ups and blow-downs with smooth centers.
For this reason, to show that an invariant of compact complex manifolds is a bimeromorphic invariant, it is sufficient to verify the stability of such an invariant under blow-ups along smooth centers.

Let $X$ be a compact complex manifold.
Then we have a natural double complex, the Dolbeault complex $(\mathcal{A}^{\bullet, \bullet}(X),\partial,\bar{\partial})$.
Arising from this double complex, we can construct several cohomological invariants of $X$, such as Dolbeault cohomology \cite{Dol53}, $\partial$-cohomology, Bott--Chern cohomology \cite{BC65}, and Aeppli cohomology \cite{Aep65}.
Recall that the $(p,q)$-th Hodge (resp. Bott--Chern) number $h^{p,q}_{\bar{\partial}}(X)$ (resp. $h^{p,q}_{BC}(X)$) is defined to be the complex dimension of the $(p,q)$-th Dolbeault (resp. Bott--Chern) cohomology group.
In general, Hodge and Bott--Chern numbers are not stable under blow-ups for the cohomological contributions of blowing up centers.
However, the $(p,0)$ and $(0,p)$-th Hodge and Bott--Chern numbers are invariants under blow-ups (cf. \cite{Hir64,YY20}).
Most recently, Stelzig \cite{Ste21c} showed that for all compact complex manifolds a $\mathbb{Z}$-linear combination or congruence of Hodge and Chern numbers is a bimermorphic invariant, if and only if it is a linear combinations or congruences of the $(p,0)$ or $(0,p)$-Hodge numbers.

Assume that the $\partial\bar{\partial}$-lemma holds on $X$.
By a result in \cite{DGMS}, the Bott--Chern cohomology canonically coincides with the Dolbeault cohomology.
Each compact K\"{a}hler manifold satisfies the $\partial\bar{\partial}$-lemma,
and there exist many interesting classes of compact non-K\"{a}hlerian complex manifolds satisfying the $\partial\bar{\partial}$-lemma, for examples, Moishezon manifolds and compact complex manifolds in the Fujiki of class $\mathscr{C}$, see \cite{DGMS}.
The Bott--Chern cohomology is an important holomorphic invariant in non-K\"{a}hler complex geometry.
Using Bott--Chern cohomology, Angella--Tomassini \cite{AT13} proved a cohomological characterization theorem of the $\partial\bar{\partial}$-lemma.
Bismut \cite{Bis13} presented a Riemann--Roch--Grothendieck theorem taking values in Bott--Chern cohomology which generalizes the classical Riemann--Roch--Grothendieck theorem to complex Hermitian geometry.
More recently, for an arbitrary coherent sheaf on a compact complex manifold, Wu \cite{Wu20} construct the Chern classes valued in rational Bott--Chern cohomology and established a Riemann--Roch--Grothendieck formula.

From the viewpoint of bimeromorphic geometry, it is natural to ask whether the $\partial\bar{\partial}$-lemma is stable under bimeromorphic maps or not.
In \cite{YY20}, using a sheaf-theoretic approach, we proved a blow-up formulae for Bott--Chern cohomology and showed that the Non-K\"{a}hlerness degrees defined by Angella--Tomassini \cite{AT13} are bimeromorphic invariants for compact complex threefolds, and therefore the $\partial\bar{\partial}$-lemma is stable under bimeromorphic maps.
Recently, extensive works have been done on the topics of blow-up formulae and the $\partial\bar{\partial}$-lemma.
We refer the readers to \cite{ASTT20}, \cite{AT17}, \cite{CY19}, \cite{Ste21}, \cite{Men1}, \cite{Men2}, \cite{Men3}, \cite{Men4}, \cite{Men5}, \cite{Wu20}, \cite{Zou20}, \cite{Tar20}, \cite{FLY12}, \cite{Fri19}, \cite{RWZ19},
\cite{RYYY} etc., and the references therein for some resent results.

\subsection{Motivation and results}
Fix a bi-degree $(p,q)$.
Schweitzer \cite{Sch07} introduced a new hypercohomology description of the Bott--Chern cohomology.
To be more specific, he defined a sheaf complex, denoted by $\mathscr{B}^{\bullet}_{X}(p,q)$, and we call it the \emph{Bott--Chern complex}.
The construction of Bott--Chern complex is very similar to the Deligne complex.
It contains the sheaf of locally constant functions, the truncated holomorphic de Rham complex, and the truncated anti-holomorphic de Rham complex of $X$.
Moreover, the $(p+q)$-th hypercohomology of $\mathscr{B}^{\bullet}_{X}(p,q)$ is isomorphic to the $(p,q)$-th Bott--Chern cohomology, see \cite[Propositions 4.2, 4.3]{Sch07}.
Contrary to the Dolbeault resolution of the (anti) holomorphic de Rham complex, we would like emphasize that the Bott--Chern complex does not admit a good resolution of fine sheaves.
Notice that both $(p,0)$ and $(0,p)$-th Bott--Chern numbers admit the bimeromorphic invariance.
In particular, for threefolds, the Non-K\"{a}hlerness degrees are bimeromorphic invariants which are defined in terms of Bott--Chern cohomology groups.
By definition, the Bott--Chern hypercohomology includes more information of $X$, both topological and holomorphic.
So a natural problem that arises now is:
\begin{prob}
For a compact complex manifold $X$, is it possible to constuct some new bimeromorphic invariants of $X$ in terms of Bott--Chern hypercohomology groups?
\end{prob}

In this article, we set ourselves a more modest goal.
As a continuation of our previous work \cite{YY20},
via comparing the Bott--Chern hypercohomology under a blow-up,
we define some new bimeromorphic invariants of an arbitrary compact complex manifold.
More precisely, we obtain the following.

\begin{thm}[=Theorem \ref{birat-inv-thm1} \& Theorem \ref{real-inv}]\label{thm1}
Suppose $X$ is a compact complex manifold.
Let $\mathcal{H}_{X}$ be the sheaf of germs of pluriharmonic functions on $X$,  and $\Omega_{X}^{\bullet \geq 1}$ the truncated holomorphic de Rham complex.
Then we have:
\begin{enumerate}
  \item [(i)]For any integer $k\geq1$, both the kernel and cokernel of the morphism
  $$
  \mathfrak{C}^{k}:H^{k-1}(X,\mathcal{H}_{X}) \longrightarrow H^{k-1}(X,\Omega_{X}^{\bullet \geq 1})
  $$
  are bimeromorphic invariants, where $\mathfrak{C}^{k}$ is the morphism induced by \eqref{map-d-com}.
  In particular, the integer
  $$
  \spadesuit^{k}(X)=
\dim_{\mathbb{C}}\, H^{k-1}(X,\Omega_{X}^{\bullet \geq 1})-\dim_{\mathbb{C}}\,H^{k-1}(X,\mathcal{H}_{X})
$$
is a bimeromorphic invariant of $X$.
\item [(ii)] The integer
$$
\clubsuit^{k}(X)=
\dim_{\mathbb{C}} H_{dR}^{k}(X; \mathbb{C})-\dim_{\mathbb{C}}\,H^{k-1}(X,\mathcal{H}_{X})
$$
is a bimeromorphic invariant of $X$, where $H_{dR}^{k}(X; \mathbb{C})$ is the $k$-th de Rham cohomology of $X$.
\end{enumerate}
\end{thm}

To accomplish this goal, we make a careful comparison between the Bott--Chern hypercohomology and other cohomologies.
Particularly, we study the higher direct images of the Bott--Chern complex under the projection of the projective bundle.
This enables us to derive a blow-up formula with an explicit morphism for Bott--Chern hypercohomology (Theorem \ref{main-thm1}) via a sheaf-theoretic approach originally used in \cite{YY20}.
It should be pointed out that each term in the Bott--Chern complex $\mathscr{B}^{\bullet}_{X}(p,q)$ has a canonical fine resolution, i.e. the de Rham resolution of $\mathbb{C}_{X}$ or the Dolbeault resolution of the sheaf of germs of (anti) holomorphic forms. However, these resolutions can not make up a double complex of sheaves and hence can not give rise to a fine resolution of $\mathscr{B}^{\bullet}_{X}(p,q)$.

For compact complex threefolds, as mentioned above, it was shown in \cite{YY20} that the Non-K\"{a}hlerness degrees are bimermorphic invariants.
It is noteworthy that, for a general bi-degree $(p,q)$ with $0<p,q<3$, both the Hodge number $h^{p,q}_{\bar{\partial}}$ and the Bott--Chern number $h^{p,q}_{BC}$ are not stable under bimermorphic maps.
In the preprint version \cite{RYY17}, together with Sheng Rao, we show that the integer
$$
h^{p,q}_{BC}-h^{p,q}_{\bar{\partial}}
$$
is a bimermorphic invariant of compact complex threefolds.
Observe that there exists a natural map from the Bott--Chern cohomology to the Dolbeault cohomology defined by the identity map.
It is of importance to point out that this map is neither injective nor surjective if the $\partial\bar{\partial}$-lemma fails.
As a by-product, we refine \cite[Corollary 1.5]{RYY17} and get the following result.

\begin{thm}[=Theorem \ref{bc-d-inv-1}]\label{bc-d-inv-0}
Let $X$ be a compact complex threefold.
Consider the natural morphism
$$
\mathfrak{I}^{p,q}:
H^{p,q}_{BC}(X) \longrightarrow H^{p,q}_{\bar{\partial}}(X),
$$
for any $0\leq p,q\leq 3$.
Then both the kernel $\ker\,\mathfrak{I}^{p,q}$ and the cokernel $\mathrm{coker}\,\mathfrak{I}^{p,q}$
are bimeromorphic invariants of $X$.
\end{thm}

It is worthy to note that the kernel (resp. cokernel) of $\mathfrak{I}^{p,q}$ may be non-trivial in general, even if $h^{p,q}_{BC}(X)-h^{p,q}_{\bar{\partial}}(X)=0$ holds.
To verify the non-triviality, we explicitly present the generators of the kernel $\ker\,\mathfrak{I}^{p,q}$ and the cokernel $\mathrm{coker}\,\mathfrak{I}^{p,q}$ for some compact complex nilmanifolds.

An outline of this article is organized as follows.
We devote Section \ref{bcc} to review the definition of Bott--Chern hypercohomology and establish a Poincar\'{e} type duality of Bott--Chern hypercohomology.
In Section \ref{geo-blw}, we study the behavior of Bott--Chern hypercohomology under proper modifications and mainly focus on the blow-up transformations.
In Section \ref{bimero-inv}, we construct some new bimeromorphic invariants of compact complex manifolds in terms of Bott--Chern hypercohomology.
In Section \ref{example}, we compute the bimeromorphic invariants defined in this article for some compact complex threefolds.
In Appendix \ref{appA}, we briefly review the definition and basic properties of relative Dolbeault sheaves with respect to a closed complex submanifold.
In Appendix \ref{appB}, we investigate the higher direct images of the Bott--Chern complexes under the projective bundle morphism.

%--------------------------------------------------------------------------------------------------------------------------------

\subsection*{Acknowledgement}
We would like to express our great gratitude to Professors An-Min Li, Guosong Zhao, Xiaojun Chen, and Bohui Chen for their constant encouragements and supports,
and sincerely thank the School of Mathematics of Sichuan University and Tianyuan Mathemtical Center in Southwest China for hosting our research visit during the winter of 2021 when we were working on this work.
We are indebted to Sheng Rao for informing us of the paper \cite{Ste22}, and Lingxu Meng for many helpful suggestions.
In particular, we are grateful to Jonas Stelzig for sending us the new version of \cite{Ste21c} and suggesting us the further Problem \ref{prob4.2}.
This work is partially supported by the National Nature Science Foundation of China (Grant No. 12171351, 11701414, and 11701051), and the Natural Science Foundation of Tianjin (Grant No. 20JCQNJC02000).

%--------------------------------------------------------------------------------------------------------------------------------

\subsection*{Conventions}
Let $f: Y\rightarrow X$ be a holomorphic map of complex manifolds.
The following symbols will have a fixed meaning throughout this paper.
\begin{itemize}
\item $\mathcal{A}_{X}^{k}$ the sheaf of complex differential $k$-forms;
\item $\mathcal{A}_{X}^{s,t}$ the sheaf of differential $(s,t)$-forms on $X$,
$\mathcal{A}^{s,t}(X):=\Gamma(X, \mathcal{A}_{X}^{s,t})$ the space of differential $(s,t)$-forms,
and $\mathcal{D}_{X}^{s, t}$ the sheaf of $(s,t)$-current on $X$;
\item $\Omega_{X}^{s}$ (resp. $\bar{\Omega}_{X}^{s}$) the sheaf of (resp. anti-) holomorphic $s$-forms on $X$, and $\mathcal{O}_{X}:=\Omega_{X}^{0}$ the structure sheaf of $X$;
\item $\Omega_{X}^{\bullet<p}$ the truncated holomorphic de Rham complex ended at $\Omega^{p-1}_{X}$;
\item $\Omega_{X}^{\bullet\geq p}$ the truncated holomorphic de Rham complex started at $\Omega^{p}_{X}$;
\item $\Omega_{X}^{s\leq\bullet\leq t}$ the truncated holomorphic de Rham complex from $\Omega^{s}_{X}$ to $\Omega^{t}_{X}$;
\item $\mathbb{G}_{X}$ the constant sheaf with value in the group $\mathbb{G}$ on $X$ ($\mathbb{G}=\mathbb{Z}, \mathbb{R}, \mathbb{C}$), i.e, the sheaf of locally constant sheaves with value in $\mathbb{G}$;
\item $f^{\star}$ the pullback of differential forms;
\item $f_{\ast}$  the direct image, and $Rf_{\ast}$ the derived direct image;
\item $f_{!}$      the proper direct image (if $f$ is proper, then $f_{\ast}=f_{!}$);
\item $f^{-1}$    the topological inverse image.
\end{itemize}

%==================================================================================

\section{Bott--Chern hypercohomology}\label{bcc}
Throughout of this paper, we assume that $X$ is a compact complex manifold of complex dimension $n$.
Recall that the sheaf $\Omega_{X}^{s}$ has a fine resolution
$$
\xymatrix@C=0.5cm{
0\ar[r]^{} &\Omega_{X}^{s} \ar[r]^{} & \mathcal{A}_{X}^{s,0} \ar[r]^{\bar{\partial}}
\ar[r]^{} & \mathcal{A}_{X}^{s,1} \ar[r]^{\bar{\partial}}& \cdots  \ar[r]^{\bar{\partial}} & \mathcal{A}_{X}^{s,n}  \ar[r]^{}& 0},
$$
which is called {\it Dolbeault resolution}.
Using this resolution, one obtains the Dolbeault theorem
$$
H_{\bar{\partial}}^{p,q}(X)
=H^{q}(\Gamma(X, \mathcal{A}_{X}^{p,\bullet}),\bar{\partial})
\cong\mathbb{H}^{q}(X, \mathcal{A}_{X}^{p,\bullet})
\cong H^{q}(X, \Omega_{X}^{p}),
$$
where $\mathbb{H}^{q}(X, \mathcal{A}_{X}^{p,\bullet})$ is the hypercohomology of the sheaf complex.
For all $0\leq p,q\leq n$, the {\it $(p,q)$-th Bott--Chern cohomology} of $X$ is defined to be the quotient space
$$
H_{BC}^{p,q}(X):=\frac{\mathrm{ker}\, \partial\cap \mathrm{ker}\, \bar{\partial}\cap\mathcal{A}^{p,q}(X)}
{\mathrm{im}\, \partial\bar{\partial}\,\cap\mathcal{A}^{p,q}(X)}
$$
and the {\it $(p,q)$-th Aeppli cohomology} of $X$ is defined to be the quotient space
$$
H_{A}^{p,q}(X):=\frac{\mathrm{ker}\, \partial\bar{\partial}\,\cap\mathcal{A}^{p,q}(X)}
{(\mathrm{im}\, \partial+\mathrm{im}\,\bar{\partial})\,\cap\mathcal{A}^{p,q}(X)}.
$$

In \cite[\S\,4.b]{Sch07}, Schweitzer showed that there exist three hypercohomology descriptions of the Bott--Chern cohomology.
In what follows, we fix a \textbf{bi-degree} $(p,q)$ with $p,q\geq 0$ and setup the following notations:
$$
\mathscr{L}_{X}^{l}=\bigoplus_{s+t=l,\atop s<p,t<q}\mathcal{A}^{s,t}_{X}\,\,\,\mathrm{when}\,\,\,l\leq p+q-2
$$
and
$$
\mathscr{L}_{X}^{l-1}=\bigoplus_{s+t=l,\atop s\geq p,t\geq q}\mathcal{A}^{s,t}_{X}\,\,\,\textrm{when}\,\,\,l\geq p+q.
$$
Define the operators:
\begin{itemize}
  \item
  $
  \delta_{l}=\mathrm{pr}\circ d:\Gamma(X,\mathscr{L}_{X}^{l})\longrightarrow
  \Gamma(X,\mathscr{L}_{X}^{l+1}),
  $
  for each $l\leq p+q-3$, where $d=\partial+\bar{\partial}$ is the de Rham differential operator and $\mathrm{pr}:\Gamma(X,\mathcal{A}_{X}^{l+1})\rightarrow\Gamma(X,\mathscr{L}_{X}^{l+1})$
  is the projection.
  \item $\delta_{k-2}=\partial\bar{\partial}:
  \Gamma(X,\mathscr{L}_{X}^{k-2})\longrightarrow\Gamma(X,\mathscr{L}_{X}^{k-1})$, for $k:=p+q$;
  \item $\delta_{l}=d:\Gamma(X,\mathscr{L}_{X}^{l})\longrightarrow\Gamma(X,\mathscr{L}_{X}^{l+1})$
  for any $l\geq k-1$.
\end{itemize}

Naturally, there is sheaf complex $\mathscr{L}_{X}^{\bullet}(p,q)$ of \textbf{fine} sheaves:
\begin{equation*}\label{BC-complex1}
\xymatrix@C=0.5cm{
\cdots \ar[rr]^{\delta_{k-4}} && \mathscr{L}^{k-3}_{X} \ar[rr]^{\delta_{k-3}} && \mathscr{L}^{k-2}_{X}
  \ar[rr]^{\delta_{k-2}}&&\mathscr{L}^{k-1}_{X} \ar[rr]^{\delta_{k-1}} &&  \mathscr{L}^{k}_{X} \ar[rr]^{\delta_{k}}
  && \cdots.}
\end{equation*}
By definition, the $(p,q)$-th Bott--Chern cohomology can be reinterpreted as hypercohomology
\begin{equation}\label{hyper-bc}
H^{p,q}_{BC}(X)
\cong H^{p+q-1}(\Gamma(X, \mathscr{L}_{X}^{\bullet}(p,q)))
\cong \mathbb{H}^{p+q-1}(X,\mathscr{L}_{X}^{\bullet}(p,q)),
\end{equation}
which is essentially given by smooth differential forms.
Specially, if $p=q=1$ we have
$$
\mathscr{L}_{X}^{\bullet}(1,1):
\xymatrix@C=0.5cm{
  0 \ar[r] & \mathcal{A}^{0,0}_{X} \ar[r]^{\partial\bar{\partial}}
  & \mathcal{A}^{1,1}_{X} \ar[r]^{d\quad\,\,\,}
  &  \mathcal{A}^{1,2}_{X}\oplus\mathcal{A}^{2,1}_{X}\ar[r]^{d\quad\,\,}
  & \mathcal{A}^{1,3}_{X}\oplus\mathcal{A}^{2,2}_{X}\oplus\mathcal{A}^{3,1}_{X} \ar[r]^{\quad\quad \quad d} & \cdots.}
$$
Let $\mathcal{H}_{X}$ be the kernel sheaf of the morphism $\partial\bar{\partial}:\mathcal{A}^{0,0}_{X}\rightarrow\mathcal{A}^{1,1}_{X}$.
In fact, $\mathcal{H}_{X}$ is nothing else than the sheaf of germs of $\mathbb{C}$-valued pluriharmonic functions on $X$.
Moreover, we have
\begin{prop}[{\cite[Theorema 2.1]{Big69} or \cite[\S\,2.2, Remark]{Wel08}}]\label{pluri-har}
The sheaf $\mathcal{H}_{X}$ admits a  fine resolution
$$
\xymatrix@C=0.5cm{
  0 \ar[r] & \mathcal{H}_{X}\ar[r]^{}&\mathcal{A}^{0,0}_{X} \ar[r]^{\partial\bar{\partial}}
  & \mathcal{A}^{1,1}_{X} \ar[r]^{d\quad\,\,\,}
  &  \mathcal{A}^{1,2}_{X}\oplus\mathcal{A}^{2,1}_{X}\ar[r]^{d\quad\,\,}
  & \mathcal{A}^{1,3}_{X}\oplus\mathcal{A}^{2,2}_{X}\oplus\mathcal{A}^{3,1}_{X} \ar[r]^{\quad\quad \quad d} & \cdots,}
$$
and therefore we have
\begin{equation*}\label{b-1-1}
\mathbb{H}^{l}(X,\mathscr{L}_{X}^{\bullet}(1,1))\cong
H^{l}(X,\mathcal{H}_{X}),
\end{equation*}
for any integer $l\geq0$.
\end{prop}
A real version of Proposition \ref{pluri-har} was also considered by Harvey--Lawson \cite[Proposition 1]{HL83}.
Moreover, from \eqref{hyper-bc} we have
$$
H^{1,1}_{BC}(X)
\cong
\mathbb{H}^{1}(X,\mathscr{L}_{X}^{\bullet}(1,1))
\cong
H^{1}(X,\mathcal{H}_{X}).
$$

The second sheaf complex associated to Bott--Chern cohomology contains the sheaves of germs of (anti-)holomorphic forms.
We denote $\mathscr{S}_{X}^{\bullet}(p,q)$:
\begin{equation*}
\xymatrix@C=0.5cm{
0\ar[r]^{} & \mathcal{O}_{X}+\bar{\mathcal{O}}_{X} \ar[r]^{} &\Omega_{X}^{1}\oplus \bar{\Omega}_{X}^{1}
  \ar[r]^{}& \cdots  \ar[r]^{}& \Omega_{X}^{p-1}\oplus \bar{\Omega}_{X}^{p-1} \ar[r]^{} & \bar{\Omega}_{X}^{p} \ar[r]^{} & \cdots   \ar[r]^{} & \bar{\Omega}_{X}^{q-1} \ar[r]^{} & 0.}
\end{equation*}
Here the sheaf $\mathcal{O}_{X}+\bar{\mathcal{O}}_{X}$ is isomorphic to $\mathcal{H}_{X}$ under the natural inclusion $\mathcal{O}_{X}+\bar{\mathcal{O}}_{X}\hookrightarrow\mathcal{H}_{X}$.

The third sheaf complex associated to Bott--Chern cohomology contains the sheaves of germs of (anti-)holomorphic forms which augmented over $\mathbb{C}$.

\begin{defn}\label{BCS-complex}
The $(p,q)$-type \emph{Bott--Chern complex} of $X$ is defined to be $\mathscr{B}_{X}^{\bullet}(p,q)$:
\begin{equation}\label{BC-complex2}
\small{
\xymatrix@C=0.4cm{
0\ar[r]^{} & \mathbb{C}_{X}\ar[r]^{(+,-)\;\;\;\;\;\;\;} & \mathcal{O}_{X}\oplus \bar{\mathcal{O}}_{X} \ar[r]^{} &\Omega_{X}^{1}\oplus \bar{\Omega}_{X}^{1}
  \ar[r]^{}& \cdots  \ar[r]^{}& \Omega_{X}^{p-1}\oplus \bar{\Omega}_{X}^{p-1} \ar[r]^{} & \bar{\Omega}_{X}^{p} \ar[r]^{} & \cdots   \ar[r]^{} & \bar{\Omega}_{X}^{q-1} \ar[r]^{} & 0.}}
\end{equation}
The $k$-th hypercohomology of \eqref{BC-complex2}, denoted by
$$
H^{k}_{BC}(X,\mathbb{C}(p,q)):=\mathbb{H}^{k}(X,\mathscr{B}_{X}^{\bullet}(p,q))
$$
is called the $k$-th \emph{Bott--Chern hypercohomology} of $X$ with respect to the bi-degree $(p,q)$.
\end{defn}

The Bott--Chern hypercohomology groups are finite dimensional complex vector spaces if $X$ is compact.
This comes from the fact that the Bott--Chern complex is quasi-isomorphic to a sheaf complex which has finite dimensional hypercohomology (cf. \cite[Theorem 12.4]{Dem12}).
The complex $\mathscr{L}_{X}^{\bullet}(p,q)$ has the virtue that each component is a fine sheaf.
So we can compute its hypercohomology groups via the corresponding complex of global sections.
In particular, we have the following result, see \cite[Lemma 12.1]{Dem12} and \cite[Propositions 4.2 and 4.3]{Sch07}.

\begin{lem}\label{equal-cohom-BCcomplex}
There exist canonical quasi-isomorphisms of sheaf complexes
$$
\mathscr{B}^{\bullet}_{X}(p,q)
\stackrel{\sim}\longrightarrow
\mathscr{S}^{\bullet}_{X}(p,q)[-1]
\stackrel{\sim}\longrightarrow
\mathscr{L}^{\bullet}_{X}(p,q)[-1],
$$
which induces an isomorphism of hypercohomology
$$
H^{k}_{BC}(X,\mathbb{C}(p,q))
\cong \mathbb{H}^{k-1}(X,\mathscr{L}_{X}^{\bullet}(p,q))
\cong H^{k-1}(\Gamma(X, \mathscr{L}_{X}^{\bullet}(p,q)))
$$
for any $k\in \mathbb{Z}$.
In particular, $H^{k}_{BC}(X,\mathbb{C}(p,q))$ is trivial except for $1\leq k\leq 2n$.
\end{lem}

Given a Hermitian metric $g$ on $X$.
From a result by Schweitzer \cite[Section 2.c]{Sch07}, for each bi-degree $(r,s)$,
the Hodge-$\ast$-operator
$\ast:\mathcal{A}^{r,s}(X)\rightarrow \mathcal{A}^{n-r,n-s}(X)$
induces an isomorphism
\begin{equation}\label{BC-A-duality}
\ast:H^{r,s}_{BC}(X)\stackrel{\simeq}\longrightarrow H^{n-r,n-s}_{A}(X).
\end{equation}
In general, consider the $l$-th hypercohomology group $\mathbb{H}^{l}(X,\mathscr{L}_{X}^{\bullet}(p,q))$.
For the simplicity, we set $p^{\prime}=n-p+1$ and $q^{\prime}=n-q+1$,
and denote by $\Gamma_{X}^{\bullet}(p,q):=\Gamma(X, \mathscr{L}_{X}^{\bullet}(p,q))$.
Then we get
$$
\mathbb{H}^{l}(X,\mathscr{L}_{X}^{\bullet}(p,q))\cong H^{l}(\Gamma_{X}^{\bullet}(p,q))
$$
and
$$
\mathbb{H}^{2n-l-1}(X,\mathscr{L}^{\bullet}_{X}(p^{\prime},q^{\prime}))\cong H^{2n-l-1}(\Gamma_{X}^{\bullet}(p^{\prime},q^{\prime})).
$$
The Hodge-$\ast$-operator determines a morphism from
$\Gamma_{X}^{\bullet}(p,q)$ to $\Gamma_{X}^{\bullet}(p^{\prime},q^{\prime})$, see the figure below.

$$
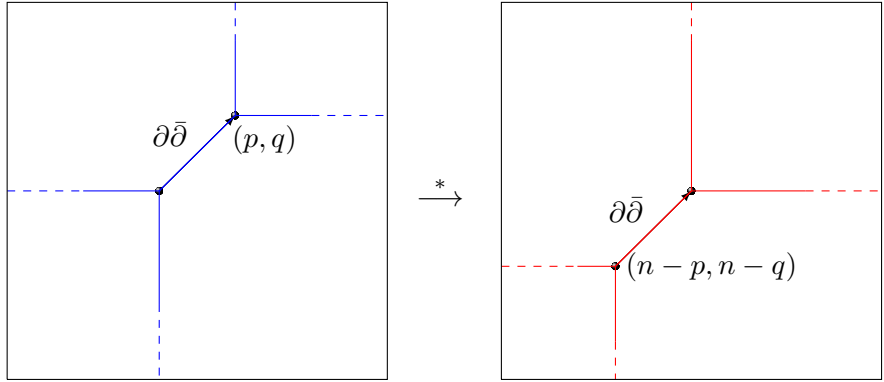
\begin{tikzpicture}[scale=0.5]
\draw[-latex]  (4,5)--(6,7) node[below] {$\qquad (p,q)$};
  \shade[ball color=black] (6,7) circle (0.1);
    \shade[ball color=black] (4,5) circle (0.1);
\draw  (5,6.5) node[left] {$\partial\bar{\partial}$};
\draw[blue] (2,5) --(4,5)--(6,7)--(6,9);
\draw[blue] (4,2) --(4,5)--(6,7)--(8,7);
\draw[dashed,blue] (4,2) --(4,0);
\draw[dashed,blue] (0,5) --(2,5);
\draw[dashed,blue] (6,10) --(6,9);
\draw[dashed,blue] (8,7) --(10,7);
\draw (0,0) --(10,0)--(10,10)--(0,10)--(0,0);
%%%%%%%%%%%%%%%%%%%%%%%%%%%%%%%%%%%%%%%%%%%%%%%%%%%%%
\draw[-latex]  (16,3)--(18,5);
\draw  (17,4.5) node[left] {$\partial\bar{\partial}$};
  \shade[ball color=black] (16,3) circle (0.1);
    \shade[ball color=black] (18,5) circle (0.1);
\draw  (16,3)node[right] {$(n-p,n-q)$};
\draw[red] (15,3) --(16,3)--(18,5)--(21,5);
\draw[red] (16,1) --(16,3)--(18,5)--(18,9);
\draw[dashed,red] (16,1) --(16,0);
\draw[dashed,red] (13,3) --(15,3);
\draw[dashed,red] (18,10) --(18,9);
\draw[dashed,red] (21,5) --(23,5);
\draw (13,0) --(23,0)--(23,10)--(13,10)--(13,0);
\draw  (10.5,5) node[right] {$\stackrel{*}{\longrightarrow}$};
\end{tikzpicture}
$$

The following result is a slight generalization of the duality between the Bott--Chern and Aeppli cohomologies.
Compare also \cite[Corollary A.2]{Ste22} for a similar result using the structure theory of double complexes.

\begin{prop}\label{SD-bchper}
Let $X$ be a compact complex manifold of complex dimension $n$.
Given a Hermitian metric $g$ on $X$, the Hodge-$\ast$-operator induces an isomorphism
\begin{equation}\label{sd-iso}
\ast:\mathbb{H}^{l}(X,\mathscr{L}_{X}^{\bullet}(p,q))
\stackrel{\simeq}\longrightarrow
\mathbb{H}^{2n-l-1}(X,\mathscr{L}^{\bullet}_{X}(n-p+1,n-q+1)),
\end{equation}
for any integer $l\geq0$.
In particular, there is an isomorphism
$$
H^{k}_{BC}(X,\mathbb{C}(p,q))
\cong
H^{2n+1-k}_{BC}(X,\mathbb{C}(n-p+1,n-q+1))
$$
for any integer $k\geq1$.
\end{prop}

\begin{proof}
According to the duality \eqref{BC-A-duality},
when $l=p+q-1$, we get
$$
\ast:H^{p+q-1}(\Gamma_{X}^{\bullet}(p,q))=H^{p,q}_{BC}(X)
\stackrel{\simeq}\longrightarrow
H^{n-p,n-q}_{A}(X)=H^{2n-p-q}(\Gamma_{X}^{\bullet}(p^{\prime},q^{\prime}));
$$
$$
\ast:H^{p+q-2}(\Gamma_{X}^{\bullet}(p,q))=H^{p-1,q-1}_{A}(X)\stackrel{\simeq}\longrightarrow
H^{p^{\prime},q^{\prime}}_{BC}(X)=H^{2n-p-q+1}(\Gamma_{X}^{\bullet}(p^{\prime},q^{\prime})).
$$
So we only need to verify the assertion in the case of $l\geq p+q$ or $l\leq p+q-3$.
Without loss of generality, we assume that $l\geq p+q$.
Then $\mathbb{H}^{l}(X,\mathscr{L}_{X}^{\bullet}(p,q))$ is equal to the cohomology of the complex
$$
\xymatrix@C=0.4cm{
  \bigoplus\limits_{s+t=l\atop{s\geq p,t\geq q}}\mathcal{A}^{s,t}(X)
   \ar[rr]^{\delta_{l-1}} &&
   \bigoplus\limits_{s+t=l+1\atop{s\geq p,t\geq q}}\mathcal{A}^{s,t}(X) \ar[rr]^{\delta_{l}} &&
   \bigoplus\limits_{s+t=l+2\atop{s\geq p,t\geq q}}\mathcal{A}^{s,t}(X). }
$$
For any $0\leq s,t\leq n$, the Hermitian metric $g$ determines a canonical Hermitian structure on the complex vector bundle
$\bigwedge^{s}(T^{1,0}X)^{\ast}\bigoplus_{\mathbb{C}}
\bigwedge^{t}(T^{0,1}X)^{\ast}$.
As a result, the differential operators $\delta_{l-1}$ and $\delta_{l}$ above admit the unique formal adjoint operators $\delta^{\ast}_{l-1}$ and
$\delta^{\ast}_{l}$ (cf. \cite[Proposition 2.8]{Wel08}).
Let $\alpha$ and $\beta$ be two arbitrary forms expressed as
$$
\alpha=\sum^{l-q+1}_{j=p}\alpha^{j,l-j+1}\in\bigoplus\limits_{s+t=l+1\atop{s\geq p,t\geq q}}\mathcal{A}^{s,t}(X),\,\,\,\,\quad\quad
\beta=\sum^{l-q+2}_{j=p}\beta^{j,l-j+2}\in\bigoplus\limits_{s+t=l+2\atop{s\geq p,t\geq q}}\mathcal{A}^{s,t}(X).
$$
By Stokes' theorem and the basic properties of the Hodge-$\ast$-operator, we have
\begin{eqnarray*}
   \int_{X}d\alpha\wedge\ast\beta
  &=&\int_{X}d(\alpha\wedge\ast\beta)-(-1)^{l+1}\int_{X}\alpha\wedge d\ast\beta
   = (-1)^{l}\int_{X}\alpha\wedge d\ast\beta\\
   &=&(-1)^{2n+2n(l+1)+1}\int_{X}\alpha\wedge \ast(\ast d\ast\beta)\\
   &=&-\int_{X}\alpha\wedge \ast(d^{\ast}\beta).
\end{eqnarray*}
Observe that $\partial^{\ast}\beta^{p,l-p+2}$ and $\bar{\partial}^{\ast}\beta^{l-q+2,q}$ are of types $(p-1,l-p+2)$ and $(l-q+2,q-1)$ respectively.
For the degree reason, we obtain
$\alpha\wedge\ast(\partial^{\ast}\beta^{p,l-p+2})=0$ and
$\alpha\wedge\ast(\bar{\partial}^{\ast}\beta^{l-q+2,q})=0$.
From definition, we have
$$
(\delta_{l}\alpha,\beta)=\int_{X}d\alpha\wedge\ast\beta=-\int_{X}\alpha\wedge \ast(d^{\ast}\beta)
=\int_{X}\alpha\wedge \ast(-\Pi\circ d^{\ast}\beta),
$$
where
$$
\Pi:\bigoplus\limits_{s+t=l+1\atop{s\geq p-1,t\geq q-1}}\mathcal{A}^{s,t}(X)
\longrightarrow
\bigoplus\limits_{s+t=l+1\atop{s\geq p,t\geq q}}\mathcal{A}^{s,t}(X)
$$
is the projection.
As a result, we get $\delta^{\ast}_{l}=-\Pi\circ d^{\ast}$, namely, there exists a commutative diagram
$$
\xymatrix@=1.5cm{
  \bigoplus\limits_{s+t=l+2\atop{s\geq p,t\geq q}}\mathcal{A}^{s,t}(X)
   \ar[dr]_{\delta^{\ast}_{l}} \ar[r]^{-d^{\ast}}
   & \bigoplus\limits_{s+t=l+1\atop{s\geq p-1,t\geq q-1}}\mathcal{A}^{s,t}(X)
   \ar[d]^{\Pi}  \\
   & \bigoplus\limits_{s+t=l+1\atop{s\geq p,t\geq q}}\mathcal{A}^{s,t}(X),            }
$$

Similar to $\delta^{\ast}_{l}$, we obtain the expression of the adjoint operator $\delta^{\ast}_{l-1}$.
Put $\Delta_{l}=\delta^{\ast}_{l}\delta_{l}+\delta_{l-1}\delta^{\ast}_{l-1}$.
Following the steps in the proof of \cite[Lemma 5.18]{Voi02}, without any essential changes, we can show that $\Delta_{l}$ is a self-adjoint elliptic differential operator acting on $\Gamma_{X}^{l}(p,q)$.
According to the Hodge theorem of self-adjoint elliptic operators on compact oriented smooth manifolds \cite[Theorem 4.12]{Wel08}, we have an isomorphism
\begin{equation*}
H^{l}(\Gamma_{X}^{\bullet}(p,q))\cong\ker\,
\bigl[\Delta_{l}:\Gamma^{l}_{X}(p,q)\rightarrow
\Gamma^{l}_{X}(p,q)\bigr]:
=\mathcal{H}^{l}_{p,q}.
\end{equation*}
Similarly, we have another self-adjoint elliptic differential operator
$$
\Delta_{2n-l-1}:\Gamma^{2n-l-1}_{X}(p^{\prime},q^{\prime})\rightarrow
\Gamma^{2n-l-1}_{X}(p^{\prime},q^{\prime})
$$
such that
\begin{equation*}
H^{2n-l-1}(\Gamma_{X}^{\bullet}(p^{\prime},q^{\prime}))\cong\ker\,\Delta_{2n-l-1}:
=\mathcal{H}^{2n-l-1}_{p^{\prime},q^{\prime}}.
\end{equation*}
A direct checking shows that the map
$$
\ast:\mathcal{H}^{l}_{p,q}\longrightarrow\mathcal{H}^{2n-l-1}_{p^{\prime},q^{\prime}}
$$
is an isomorphism and therefore we are led to the conclusion that \eqref{sd-iso} is isomorphic.
\end{proof}

Consider the bilinear pairing
\begin{equation}\label{pairing}
  \langle-,-\rangle:
  \mathbb{H}^{l}(X,\mathscr{L}_{X}^{\bullet}(p,q))\times
  \mathbb{H}^{2n-l-1}(X,\mathscr{L}^{\bullet}_{X}(n-p+1,n-q+1))\longrightarrow \mathbb{C}
\end{equation}
defined by setting
$$
\langle[\alpha],[\beta]\rangle= \int_{X}\alpha\wedge\beta.
$$
On account of the degree reason, this pairing is independent of the choices of the representatives and hence is well-defined.
As a corollary of Proposition \ref{SD-bchper}, we obtain
\begin{cor}
The bilinear pairing \eqref{pairing} is a non-degenerate duality.
\end{cor}
\begin{proof}
From the proof of Proposition \ref{SD-bchper}, we know that $\alpha\in\mathcal{H}^{l}_{p,q}$ if and only if $\ast\alpha\in\mathcal{H}^{2n-l-1}_{p^{\prime},q^{\prime}}$.
Then the assertion comes from the fact that the integral
$\langle\alpha,\ast\alpha\rangle=\int_{X}\alpha\wedge\ast\alpha=\|\alpha\|^{2}$
does not vanish unless $\alpha=0$.
\end{proof}

\begin{rem}\label{birat-inv-thm0}
Let $X$ be a compact complex manifold of dimension $n\geq 2$.
From the Proposition \ref{SD-bchper}, we get
$$
\mathrm{dim}_{\mathbb{C}}\,H^{2n-1}_{BC}(X, \mathbb{C}(n-1,n-1))=
\mathrm{dim}_{\mathbb{C}}\,H^{2}_{BC}(X, \mathbb{C}(2,2))
=\mathrm{dim}_{\mathbb{C}}\,H^{1}_{dR}(X,\mathbb{C}),
$$
which is a bimeromorphic invariant of $X$.
\end{rem}

Next we make a comparison for the Bott--Chern complex and truncated holomorphic de Rham complex.
Consider the sheaf complex
\begin{equation*}\label{c-hol-tru}
\mathbb{C}_{X}(p):\xymatrix@C=0.3cm{
0\ar[rr]^{} && \mathbb{C}_{X}\ar[rr]^{} && \mathcal{O}_{X}\ar[rr]^{\partial} &&\Omega_{X}^{1}
  \ar[rr]^{\partial}&& \cdots  \ar[rr]^{\partial}&& \Omega_{X}^{p-1} \ar[rr]^{} && 0.}
\end{equation*}
Note that the holomorphic de Rham complex $(\Omega^{\bullet}_{X},\partial)$ is a resolution of $\mathbb{C}_{X}$ via the obvious inclusion $\imath:\mathbb{C}_{X}\rightarrow\mathcal{O}_{X}$ (cf. \cite[Lemma 8.13]{Voi02}).
It follows that $\mathbb{C}_{X}(p)$ is canonically quasi-isomorphic (after a degree shifting) to the truncated holomorphic de Rham complex $(\Omega^{\bullet\geq p}_{X}[-p],\partial)$ and therefore for any $k\in\mathbb{Z}$ we have the isomorphism
\begin{equation}\label{c-dol-0}
\mathbb{H}^{k}(X,\mathbb{C}_{X}(p))\cong\mathbb{H}^{k}(X,\Omega^{\bullet\geq p}_{X}[-p])=
\mathbb{H}^{k-p}(X,\Omega^{\bullet\geq p}_{X}).
\end{equation}
If $X$ is a K\"{a}hler manifold, by a result of Schweitzer \cite[Lemme 7.2]{Sch07} the equalities
\begin{equation}\label{c-dol-1}
\mathbb{H}^{l}(X,\Omega^{\bullet< p}_{X})=\bigoplus_{s<p\atop{s+t=l}}H^{s,t}_{\bar{\partial}}(X)
\,\,\,\mathrm{and}\,\,\,
\mathbb{H}^{l}(X,\Omega^{\bullet\geq p}_{X})=\bigoplus_{s\geq p\atop{s+t=l}}H^{s,t}_{\bar{\partial}}(X)
\end{equation}
hold for any $l\geq0$.

Observe that there exists a natural morphism of sheaf complexes:
\begin{equation*}\label{varpi}
\vcenter{
\xymatrix@C=0.5cm{
 \mathscr{B}^{\bullet}_{X}(p,p): \ar[d]_{\mathcal{P}} &  0 \ar[r]^{} & \mathbb{C}_{X}
  \ar[d]_{\mathrm{id}}\ar[r]^{} &  \Omega^{\bullet <p}_{X} \oplus
 \bar{\Omega}^{\bullet <p}_{X}\ar[d]_{\mathrm{pr}} \ar[r]^{} & \bar{\Omega}^{p\leq\bullet\leq q-1}_{X} \ar[d]_{} \ar[r]^{} & 0\\
\mathbb{C}_{X}(p): &  0 \ar[r]^{} & \mathbb{C}_{X} \ar[r]^{} & \Omega^{\bullet<p}_{X} \ar[r] &0
\ar[r] &0, }}
\end{equation*}
where $\rho$ is the projection.
As a direct consequence, we get a short exact sequence of sheaf complexes
\begin{equation}\label{s-e-bcs-d}
\xymatrix@C=0.5cm{
  0 \ar[r] & \bar{\Omega}^{\bullet <q}_{X}[-1] \ar[rr]^{\imath} && \mathscr{B}^{\bullet}_{X}(p,q) \ar[rr]^{\mathcal{P}} && \mathbb{C}_{X}(p) \ar[r] & 0. }
\end{equation}

\begin{rem}
Similarly, one can define the integral Bott--Chern complex which  is analogous to \eqref{BC-complex2} with $\mathbb{C}_{X}$ replaced by $\mathbb{Z}_{X}$ (cf. \cite{Sch07}).
For an integral Bott--Chern complex there exists a natural splitting as the direct sum of a truncated anti-holomorphic de Rham complex and an integral Deligne complex (cf. \cite[Proposition 7.3]{Sch07}).
Differing from the integral case, although the short exact sequence \eqref{s-e-bcs-d} is very similar to the one in the proof of \cite[Proposition 7.3]{Sch07}, it is not a splitting.
\end{rem}

%============================================================================

\section{Geometry of blow-ups}\label{geo-blw}

In this section we study the behavior of the Bott--Chern hypercohomology under blow-up transformations of compact complex manifolds.

\subsection{Projective bundle formulae}\label{projective-b-f}

Suppose that $\mathcal{V}$ is a holomorphic vector bundle of rank $c$ over a compact complex manifold $Z$.
Let $\rho:\PV=\PV(\mathcal{V})\longrightarrow Z$ be the projective bundle and
$h:=c_{1}(\mathcal{O}_{\PV}(1))\in H^{2}(\PV, \mathbb{Z})$ the first Chern class of the relative tautological line bundle $\mathcal{O}_{\PV}(1)$.
Then we have

\begin{lem}\label{(1,1)-projbundle-formula}
There is a canonical quasi-isomorphism of sheaf complexes
$$
\varphi=\sum_{i=0}^{c-1} h^{i}\wedge \rho^{\star}:
\bigoplus_{i=0}^{c-1} \mathscr{L}_{Z}^{\bullet}(1-i,1-i)[-2i]
\stackrel{\sim}\longrightarrow
R\rho_{\ast}\mathscr{L}_{\PV}^{\bullet}(1,1).
$$
on $Z$, where $\mathscr{L}_{Z}^{\bullet}(1-i,1-i)=\mathcal{A}_{Z}^{\bullet}[1]$ for $i\geq 1$.
\end{lem}

\begin{proof}
First of all, we note that there exists a fine resolutions
$$
\mathcal{H}_{Z}\oplus \bigoplus_{i=1}^{c-1}\mathbb{C}_{Z}[-2i+1]
\stackrel{\sim}\longrightarrow
\bigoplus_{i=0}^{c-1} \mathscr{L}_{Z}^{\bullet}(1-i,1-i)
$$
and hence we get the cohomology sheaf
\begin{equation}\label{coh-sheaf}
 \mathscr{H}^{j}\biggl(\bigoplus_{i=0}^{c-1}\mathscr{L}_{Z}^{\bullet}(1-i,1-i)[-2i]\biggr)=
\bigoplus_{i=0}^{c-1} \mathscr{H}^{j-2i}(\mathscr{L}_{Z}^{\bullet}(1-i,1-i))
\cong
\begin{cases}
\mathcal{H}_{Z},      & j=0; \\
 \mathbb{C}_{Z},     & j \text{ is odd}; \\
 0,     & \text{otherwise}.
\end{cases}
\end{equation}
To conclude the proof,
it suffices to show that the induced morphism of cohomology sheaves
\begin{equation}\label{ind-cohsh}
\mathscr{H}^{j}(\varphi):
\bigoplus_{i=0}^{c-1} \mathscr{H}^{j-2i}(\mathscr{L}_{Z}^{\bullet}(1-i,1-i))
\longrightarrow
 \mathscr{H}^{j}(R\rho_{\ast}\mathscr{L}_{\PV}^{\bullet}(1,1))
\end{equation}
is isomorphic for any integer $j\geq0$.
Since $\mathscr{L}_{\PV}^{\bullet}(1,1)$ is a fine resolution of $\mathcal{H}_{\PV}$, by definition, the $j$-th direct image of $\mathcal{H}_{\PV}$ is
$$
R^{j}\rho_{\ast}\mathcal{H}_{\PV}
\cong
\mathscr{H}^{j}(R\rho_{\ast}\mathscr{L}_{\PV}^{\bullet}(1,1))
=
\mathscr{H}^{j}(\rho_{\ast}\mathscr{L}_{\PV}^{\bullet}(1,1)).
$$
Our next goal is to compute the higher direct image of $\mathcal{H}_{\PV}$.
For this purpose, we consider the short exact sequence
\begin{equation*}
\xymatrix@C=0.5cm{
0
\ar[r]^{}&
\mathbb{C}_{\PV}
\ar[r]^{} &
\mathcal{O}_{\PV}\oplus  \bar{\mathcal{O}}_{\PV}
\ar[r]^{} &
\mathcal{H}_{\PV}
\ar[r]^{} &
0.}
\end{equation*}
Then there is a long exact sequence of higher direct images
\begin{equation}\label{exactseq-hdimg}
\begin{tikzcd}
 0\rar &  \rho_{\ast}\mathbb{C}_{\PV}
  \rar &   \rho_{\ast}(\mathcal{O}_{\PV}\oplus\bar{\mathcal{O}}_{\PV})
  \rar &  \rho_{\ast}\mathcal{H}_{\PV} \ar[out=-10, in=165]{dll} \\
  & R^{1}\rho_{\ast}\mathbb{C}_{\PV}
 \rar &  R^{1}\rho_{\ast}(\mathcal{O}_{\PV}\oplus\bar{\mathcal{O}}_{\PV})
 \rar & R^{1}\rho_{\ast}\mathcal{H}_{\PV}   \ar[out=-10, in=165]{dll} \\
  & R^{2}\rho_{\ast}\mathbb{C}_{\PV}
  \rar &  R^{2}\rho_{\ast}(\mathcal{O}_{\PV}\oplus\bar{\mathcal{O}}_{\PV})   \rar & \cdots.
\end{tikzcd}
\end{equation}
Since $R^{1}\rho_{\ast}\mathbb{C}_{\PV}=0$,
from \eqref{exactseq-hdimg}, we obtain a short exact sequence
$$
\xymatrix@C=0.5cm{
 0 \ar[r]^{} & \rho_{\ast}\mathbb{C}_{\PV}   \ar[r]^{} & \rho_{\ast}(\mathcal{O}_{\PV}\oplus\bar{\mathcal{O}}_{\PV})  \ar[r]^{} &  \rho_{\ast}\mathcal{H}_{\PV} \ar[r]^{} & 0.}
$$
Consider the following commutative diagram
\begin{equation}\label{dir-im-Har}
\vcenter{
\xymatrix@C=0.5cm{
 0 \ar[r]^{} &  \mathbb{C}_{Z}   \ar[d]_{\cong} \ar[r]^{} &  \mathcal{O}_{Z}\oplus\bar{\mathcal{O}}_{Z}   \ar[d]_{\cong} \ar[r]^{} &   \mathcal{H}_{Z}  \ar[d]_{} \ar[r]^{} & 0 \\
0 \ar[r]^{} & \rho_{\ast}\mathbb{C}_{\PV}   \ar[r]^{} & \rho_{\ast}(\mathcal{O}_{\PV}\oplus\bar{\mathcal{O}}_{\PV})  \ar[r]^{} &  \rho_{\ast}\mathcal{H}_{\PV} \ar[r]^{} & 0.}}
\end{equation}
Because both the first and second verticals in \eqref{dir-im-Har} are isomorphic,
so is the third one and hence $\mathcal{H}_{Z}\cong \rho_{\ast}\mathcal{H}_{\PV}$.
Moreover,
since $R^{j}\rho_{\ast}(\mathcal{O}_{\PV}\oplus\bar{\mathcal{O}}_{\PV})=0$ for any $j\geq 1$, we have
$$
R^{j}\rho_{\ast}\mathcal{H}_{\PV}
\cong R^{j+1}\rho_{\ast}\mathbb{C}_{\PV}\,\,\,\mathrm{for}\,\,\, j\geq 1.
$$

Finally, we show that \eqref{ind-cohsh} is isomorphic.
Choose an open subset small enough (e.g., a small open polydisc) $W\subset Z$.
According to the K\"{u}nneth formula for de Rham cohomology,
we obtain
$$
R^{j}\rho_{\ast}\mathcal{H}_{\PV}(W)
\cong R^{j+1}\rho_{\ast}\mathbb{C}_{\PV}(W)
=H_{dR}^{j+1}(W\times \mathbb{P}^{c-1}; \mathbb{C})
\cong
H_{dR}^{0}(W; \mathbb{C})\otimes H_{dR}^{j+1}(\mathbb{P}^{c-1}; \mathbb{C}).
$$
Note that $H_{dR}^{j+1}(\mathbb{P}^{c-1}; \mathbb{C})=0$ if $j$ is even;
and $H_{dR}^{j+1}(\mathbb{P}^{c-1}; \mathbb{C})$ is generated by the restriction of $h^{k}$ for $j=2k-1$.
Therefore, we derive
$$
R^{j}\rho_{\ast}\mathcal{H}_{\PV}
\cong
\begin{cases}
\mathcal{H}_{Z},     & j=0; \\
\mathbb{C}_{Z},     & j \, \text{is odd}; \\
0,      & \text{otherwise}.
\end{cases}
$$
Combining \eqref{coh-sheaf} finishes the proof.
\end{proof}

Based on Lemma \ref{(1,1)-projbundle-formula},
we have the following general result.

\begin{prop}\label{(p,q)-projbundle-formula}
There is a canonical quasi-isomorphism of sheaf complexes
$$
\sum_{i=0}^{c-1} h^{i}\wedge \rho^{\star}:
\bigoplus_{i=0}^{c-1} \mathscr{L}_{Z}^{\bullet}(p-i,q-i)[-2i]
\stackrel{\sim}\longrightarrow
R\rho_{\ast}\mathscr{L}_{\PV}^{\bullet}(p,q)
$$
on $Z$.
\end{prop}

We present the proof of Proposition \ref{(p,q)-projbundle-formula} in Appendix \ref{appB},
which uses the same arguments as in Lemma \ref{(1,1)-projbundle-formula}.
In particular, a direct consequence of Proposition \ref{(p,q)-projbundle-formula} is the projective bundle formula of Bott--Chern hypercohomology.

\begin{cor}\label{cohom-BCcomplex-projbundle}
There is a canonical isomorphism of Bott--Chern hypercohomology
$$
\sum_{i=0}^{c-1} h^{i}\wedge\rho^{\star}(-):
\bigoplus_{i=0}^{c-1}H_{BC}^{k-2i}(Z, \mathbb{C}(p-i,q-i))
\stackrel{\simeq}\longrightarrow
H_{BC}^{k}(\PV, \mathbb{C}(p,q)).
$$
for any $k\in \mathbb{N}$.
In particular, if $k=p+q$, we get
$$
\sum_{i=0}^{c-1} h^{i}\wedge\rho^{\star}(-):
\bigoplus_{i=0}^{c-1}H_{BC}^{p-i,q-i}(Z)
\stackrel{\simeq}\longrightarrow
H_{BC}^{p,q}(\PV).
$$
\end{cor}

\begin{rem}
In the proof of \cite[Theorem 1.2]{YY20}, we did not obtain an explicit presentation for the Bott--Chern cohomology of a projective bundle.
We conjectured an explicit Bott--Chern projective bundle formula in \cite{YY20} and \cite{RYY19}.
This formula was later confirmed by Stelzig using a structure theory of double complexes developed in \cite{Ste21,Ste21b}.
\end{rem}

%================================================================

\subsection{Blow-up formulae}\label{blf}
Let us first discuss the behavior of Bott--Chern hypercohomology of compact complex manifolds under modifications.
Recall that a {\it proper modification} is a proper holomorphic map $f:Y\rightarrow X$ of complex spaces of the same dimension satisfying:
\begin{itemize}
\item there exists an analytic subset $S\subset X$ of codimension $\geq2$ such that the restriction
$$
f:Y\setminus f^{-1}(S)\longrightarrow X\setminus S
$$
is biholomorphic.
\end{itemize}

\begin{prop}\label{cohom-BCcomplex-injective}
Let $f: Y\longrightarrow X$ be a modification of compact complex manifolds.
Then the pullback of differential forms induces an injective map
\begin{equation}\label{inj-X-Y}
f^{\star}:H_{BC}^{k}(X,\mathbb{C}(p,q))
\longrightarrow H_{BC}^{k}(Y,\mathbb{C}(p,q))
\end{equation}
for any $k\in \mathbb{Z}$.
\end{prop}

\begin{proof}
Following the steps in the construction of the sheaf complex $\mathscr{L}_{X}^{\bullet}(p,q)$,
we can define a new sheaf complex $\mathscr{C}_{X}^{\bullet}(p,q)$ which consists of the sheaves of currents $\mathcal{D}_{X}^{s,t}$.
It is noteworthy that there exists a natural inclusion
$\tau:\mathscr{L}_{X}^{\bullet}(p,q)
\hookrightarrow\mathscr{C}_{X}^{\bullet}(p,q)$ which is a quasi-isomorphism of sheaf complexes, see \cite[Lemma VI.12.1]{Dem12}.
This implies that $\tau$ yields an isomorphism of hypercohomology groups
$$
\tau:\mathbb{H}^{k-1}(X, \mathscr{L}_{X}^{\bullet}(p,q))
\stackrel{\simeq}\longrightarrow
\mathbb{H}^{k-1}(X, \mathscr{C}_{X}^{\bullet}(p,q)).
$$
Likewise, consider the inclusion
$\mu:\mathscr{L}_{Y}^{\bullet}(p,q)
\hookrightarrow\mathscr{C}_{Y}^{\bullet}(p,q)$ and then we get the isomorphism
$$
\mu:\mathbb{H}^{k-1}(Y, \mathscr{L}_{Y}^{\bullet}(p,q))
\stackrel{\simeq}\longrightarrow
\mathbb{H}^{k-1}(Y, \mathscr{C}_{Y}^{\bullet}(p,q)).
$$
Since $f$ is a proper modification and hence its degree is $1$,
so $\tau(\alpha)=f_{\star}\circ \mu \circ f^{\star}(\alpha)$
for each smooth differential form $\alpha$ on $X$,
see the proof of \cite[Theorem 12.9]{Dem12} or \cite[Lemma 2.2]{Wel74}.
In particular, we have a commutative diagram
\begin{equation}\label{form-current-comdiagram1}
\vcenter{
\xymatrix{
  \mathbb{H}^{k-1}(X, \mathscr{L}_{X}^{\bullet}(p,q)) \ar[d]_{f^{\star}} \ar[r]_{\tau}^{\simeq} & \mathbb{H}^{k-1}(X, \mathscr{C}_{X}^{\bullet}(p,q))   \\
  \mathbb{H}^{k-1}(Y, \mathscr{L}_{Y}^{\bullet}(p,q))
  \ar[r]_{\mu }^{\simeq}
  & \mathbb{H}^{k-1}(Y, \mathscr{C}_{Y}^{\bullet}(p,q)) \ar[u]^{f_{\star}} .}}
\end{equation}
The commutativity of \eqref{form-current-comdiagram1} implies that the morphism
$$
f^{\star}: \mathbb{H}^{k-1}(X,\mathscr{L}_{X}^{\bullet}(p,q))
\longrightarrow \mathbb{H}^{k-1}(Y,\mathscr{L}_{Y}^{\bullet}(p,q))
$$
is injective and so is the morphism \eqref{inj-X-Y} by Lemma \ref{equal-cohom-BCcomplex}.
\end{proof}

As a special case, the blow-up morphism of compact complex manifolds is one of the most important proper modifications.
Assume that $\imath:Z\hookrightarrow X$ is a closed complex submanifold with complex codimension $c\geq 2$.
Let $\pi:\tilde{X}\rightarrow X$ be the blow-up of $X$ along $Z$ and $E:=\pi^{-1}(Z)$ the exceptional divisor.
Then there is a commutative blow-up diagram
\begin{equation}\label{bld}
\vcenter{
\xymatrix@=1.5cm{
E \ar[d]_{\rho} \ar@{^{(}->}[r]^{\tilde{\iota}} & \tilde{X}\ar[d]^{\pi}\\
 Z \ar@{^{(}->}[r]^{\iota} & X,}}
\end{equation}
where $\tilde{\iota}: E \hookrightarrow \tilde{X}$ is the inclusion.

Consider the pair $(X,Z)$.
We define a relative version of the Bott--Chern complex.
The {\it relative Bott--Chern complex} with respect to $Z$,
denoted by $\mathscr{B}_{X,Z}^{\bullet}(p,q)$, is defined to be the sheaf complex:
\begin{equation*}\label{rel-BC-complex2}
\small{
\xymatrix@C=0.4cm{
0\ar[r]^{} & \mathbb{C}_{X,Z}\ar[r]^{} & \mathcal{K}_{X,Z}^{0}\oplus \bar{\mathcal{K}}_{X,Z}^{0} \ar[r]^{} &\mathcal{K}_{X,Z}^{1}\oplus \bar{\mathcal{K}}_{X,Z}^{1}
  \ar[r]^{}& \cdots  \ar[r]^{}& \mathcal{K}_{X,Z}^{p-1}\oplus \bar{\mathcal{K}}_{X,Z}^{p-1} \ar[r]^{} & \bar{\mathcal{K}}_{X,Z}^{p} \ar[r]^{} & \cdots \ar[r]^{} & \bar{\mathcal{K}}_{X,Z}^{q-1} \ar[r]^{} & 0,}}
\end{equation*}
where $\mathcal{K}_{X,Z}^{k}$ is the $k$-the relative Dolbeault sheaf (see Definition \ref{defn-relDol}).
From \eqref{rela-const-sheaf-exact} and \eqref{exact-purerelDolsh},
we get a short exact sequence of sheaf complexes
\begin{equation}\label{pre-split-blwp-BCcomplex}
\xymatrix@C=0.5cm{
0 \ar[r] & \mathscr{B}_{X,Z}^{\bullet}(p,q)
\ar[r] & \mathscr{B}_{X}^{\bullet}(p,q)
\ar[r]^{\imath^{\star}} &  \imath_{\ast}\mathscr{B}_{Z}^{\bullet}(p,q) \ar[r] & 0.}
\end{equation}
Likewise, we can define the relative version of $\mathscr{L}_{X}^{\bullet}(p,q)$ by replacing $\mathcal{A}_{X}^{s,t}$ with $\mathcal{K}_{X,Z}^{s,t}$, and denote it by $\mathscr{L}_{X,Z}^{\bullet}(p,q)$.
Without any essential changes in the proof of Lemma \ref{equal-cohom-BCcomplex}, we get a relative version of Lemma \ref{equal-cohom-BCcomplex}, i.e. there exists an isomorphism
\begin{equation}\label{equal-cohom-relBCcomplex}
\mathbb{H}^{l}(X,\mathscr{B}_{X,Z}^{\bullet}(p,q))\cong \mathbb{H}^{l-1}(X,\mathscr{L}_{X,Z}^{\bullet}(p,q)),
\end{equation}
for any $l\in \mathbb{Z}$.
Similar to \eqref{pre-split-blwp-BCcomplex} and \eqref{equal-cohom-relBCcomplex},
for the pair $(\tilde{X}, E)$,
there exist a short exact sequence of sheaf complexes
\begin{equation*}\label{E-pre-split-blwp-BCcomplex}
\xymatrix@C=0.5cm{
0 \ar[r] & \mathscr{B}_{\tilde{X},E}^{\bullet}(p,q)
\ar[r] & \mathscr{B}_{\tilde{X}}^{\bullet}(p,q)
\ar[r]^{\tilde{\imath}^{\star}} &  \tilde{\imath}_{\ast}\mathscr{B}_{E}^{\bullet}(p,q) \ar[r] & 0,}
\end{equation*}
and an isomorphism
\begin{equation*}\label{E-equal-cohom-relBCcomplex}
\mathbb{H}^{k}(\tilde{X},\mathscr{B}_{\tilde{X},E}^{\bullet}(p,q))
\cong
\mathbb{H}^{k-1}(\tilde{X},\mathscr{L}_{\tilde{X},E}^{\bullet}(p,q)),
\end{equation*}
for any $k\in \mathbb{Z}$.
The following lemma plays an important role in the proof of blow-up formulae.
\begin{lem}\label{Derived-dirim-kershcom}
There is an isomorphism
\begin{equation}\label{iso0}
\pi^{\star}:
\mathscr{B}_{X,Z}^{\bullet}(p,q)
\stackrel{\simeq}\longrightarrow
R\pi_{\ast} \mathscr{B}_{\tilde{X},E}^{\bullet}(p,q)
\end{equation}
in the derived category of sheaves of $\mathbb{C}$-modules on $X$.
In particular, for any $k\in \mathbb{Z}$,
the induced morphism
\begin{equation}\label{iso1}
\pi^{\star}:
\mathbb{H}^{k}(X, \mathscr{B}_{X,Z}^{\bullet}(p,q))
\stackrel{\simeq}\longrightarrow
\mathbb{H}^{k}(\tilde{X}, \mathscr{B}_{\tilde{X},E}^{\bullet}(p,q))
\end{equation}
is isomorphic.
\end{lem}

\begin{proof}
Consider a new sheaf complex $\mathscr{D}_{X,Z}^{\bullet}(p,q)$ for $(X,Z)$ (similarly for $(\tilde{X},E)$):
\begin{equation*}
\xymatrix@C=0.4cm{
0\ar[r]^{} & \mathcal{K}_{X,Z}\oplus \bar{\mathcal{K}}_{X,Z} \ar[r]^{} &\mathcal{K}_{X,Z}^{1}\oplus \bar{\mathcal{K}}_{X,Z}^{1}
  \ar[r]^{}& \cdots  \ar[r]^{}& \mathcal{K}_{X,Z}^{p-1}\oplus \bar{\mathcal{K}}_{X,Z}^{p-1} \ar[r]^{} & \bar{\mathcal{K}}_{X,Z}^{p} \ar[r]^{} & \cdots   \ar[r]^{} & \bar{\mathcal{K}}_{X,Z}^{q-1} \ar[r]^{} & 0.}
\end{equation*}
Hence there is a short exact sequence
\begin{equation}\label{exact-triangle-pre0}
\xymatrix@C=0.5cm{
0 \ar[r]^{} &
\mathscr{D}_{X,Z}^{\bullet}(p,q)[-1]
\ar[r]^{} &
 \mathscr{B}_{X,Z}^{\bullet}(p,q)
 \ar[r]^{} &
 \mathbb{C}_{X,Z}  \ar[r]^{}
  &  0.}
\end{equation}
Note that there is a canonical decomposition
$$
\mathscr{D}_{\tilde{X},E}^{\bullet}(p,q)
\cong
\mathcal{K}_{\tilde{X},E}^{\bullet <p} \oplus \bar{\mathcal{K}}_{\tilde{X},E}^{\bullet <q},
$$
and the derived direct image functor $R\pi_{\ast}$ commutes with direct sum.
On account of Lemma \ref{derived-im-relative-Dolsh},
we get an isomorphism
$$
\pi^{\star}:\mathscr{D}_{X,Z}^{\bullet}(p,q)\stackrel{\simeq}\longrightarrow R\pi_{\ast}\mathscr{D}_{\tilde{X},E}^{\bullet}(p,q).
$$
Due to Corollary \ref{derived-im-relative-constsh}, there exists an isomorphism
$$
\pi^{\star}:\mathbb{C}_{X,Z}
\stackrel{\simeq}\longrightarrow
R\pi_{\ast}\mathbb{C}_{\tilde{X},E}.
$$
Consider the commutative diagram:
\begin{equation}\label{exact-triangle-pre1}
\vcenter{
\xymatrix@C=0.4cm{
0 \ar[r]^{} &  \pi^{-1}\mathscr{D}_{X,Z}^{\bullet}(p,q)[-1] \ar[d]_{}^{ }  \ar[r]^{}
  &  \pi^{-1}\mathscr{B}_{X,Z}^{\bullet}(p,q)\ar[d]_{} \ar[r]^{}
  &  \pi^{-1}\mathbb{C}_{X,Z} \ar[d]_{}^{}  \ar[r]^{}
  &  0     \\
0 \ar[r]^{}  &  \mathscr{D}_{\tilde{X},E}^{\bullet}(p,q)[-1] \ar[r]^{}
  &  \mathscr{B}_{\tilde{X},E}^{\bullet}(p,q) \ar[r]^{}
  & \mathbb{C}_{\tilde{X},E}\ar[r]
  & 0.}}
\end{equation}
Taking the natural transformation $\mathrm{id} \rightarrow R\pi_{\ast}\pi^{-1}$ to \eqref{exact-triangle-pre0}, and the functor $R\pi_{\ast}$ to \eqref{exact-triangle-pre1},
we get a morphism of exact triangles
\begin{equation}\label{exact-triangle}
\vcenter{
\xymatrix@C=0.4cm{
  &  \mathscr{D}_{X,Z}^{\bullet}(p,q)[-1] \ar[d]_{}^{\simeq}  \ar[r]^{}
  &  \mathscr{B}_{X,Z}^{\bullet}(p,q)\ar[d]_{} \ar[r]^{}
  &  \mathbb{C}_{X,Z} \ar[d]_{}^{\simeq}  \ar[r]^{}
  &  \mathscr{D}_{X,Z}^{\bullet}(p,q)[-2] \ar[d]_{}^{\simeq}   \\
  & R\pi_{\ast}\mathscr{D}_{\tilde{X},E}^{\bullet}(p,q)[-1] \ar[r]^{}
  & R\pi_{\ast}\mathscr{B}_{\tilde{X},E}^{\bullet}(p,q) \ar[r]^{}
  & R\pi_{\ast}\mathbb{C}_{\tilde{X},E}\ar[r]
  & R\pi_{\ast}\mathscr{D}_{\tilde{X},E}^{\bullet}(p,q)[-2],}}
\end{equation}
in the derived category of sheaves of $\mathbb{C}_{X}$-modules.
Applying the standard Two of Three property in triangulated category theory (cf. \cite{Ive86}) to \eqref{exact-triangle}, we obtain that \eqref{iso0} is isomorphic.
Moreover, taking the hypercohomology, we are led to the conclusion that the morphism \eqref{iso1}
is an isomorphism.
\end{proof}

We are ready to present the main result of this section.

\begin{thm}\label{main-thm1}
With the same setting as in \eqref{bld},
we have a canonical isomorphism
$$
H_{BC}^{k}(\tilde{X},\mathbb{C}(p,q))
\stackrel{\Phi}\longrightarrow
H_{BC}^{k}(X,\mathbb{C}(p,q))
\oplus \bigg
[\bigoplus_{i=1}^{c-1}H_{BC}^{k-2i}(Z,\mathbb{C}(p-i,q-i))\bigg]
$$
for any integer $k\geq0$, where $\Phi$ is a linear map defined in \eqref{BC-explicit-map}.
In particular, when $k=p+q$ we get a canonical isomorphism of Bott--Chern cohomology
$$
H_{BC}^{p,q}(\tilde{X})
\stackrel{\Phi}\longrightarrow
H_{BC}^{p,q}(X)
\oplus \bigg[\bigoplus_{i=1}^{c-1} H_{BC}^{p-i,q-i}(Z)\bigg],
$$
and a canonical isomorphism of Appeli cohomology
$$
H_{A}^{p,q}(\tilde{X})
\stackrel{\Phi}\longrightarrow
H_{A}^{p,q}(X)
\oplus \bigg[\bigoplus_{i=1}^{c-1} H_{A}^{p-i,q-i}(Z)\bigg].
$$
\end{thm}
\begin{proof}
Our first goal is to construct a commutative diagram of long exact sequences of hypercohomology groups.
Since the pullback of differential forms commutes with the differential operators $\bar{\partial}$ and $\partial$, it follows from \eqref{exact-bi-degreerelDolsh} that there exists a commutative diagram of short exact sequences of vector spaces
\begin{equation*}
\xymatrix@C=0.5cm{
 0 \ar[r]^{} & \Gamma(X,\mathcal{K}_{X,Z}^{s,t}) \ar[d]_{\pi^{\star}} \ar[r]^{} & \Gamma(X,\mathcal{A}_{X}^{s,t}) \ar[d]_{\pi^{\star}} \ar[r]^{\imath^{\star}} &  \Gamma(Z,\mathcal{A}_{Z}^{s,t})\ar[d]_{\rho^{\star}} \ar[r]^{} & 0 \\
0 \ar[r] &  \Gamma(\tilde{X},\mathcal{K}_{\tilde{X},E}^{s,t}) \ar[r]^{} & \Gamma(\tilde{X},\mathcal{A}_{\tilde{X}}^{s,t}) \ar[r]^{\tilde{\imath}^{\star}} & \Gamma(E,\mathcal{A}_{E}^{s,t}) \ar[r] &0. }
\end{equation*}
Furthermore, we get a commutative diagram of short exact sequences of complexes
\begin{equation*}
\xymatrix@C=0.5cm{
 0 \ar[r]^{} & \Gamma(X,\mathscr{L}_{X,Z}^{\bullet}(p,q)) \ar[d]_{\pi^{\star}} \ar[r]^{} & \Gamma(X,\mathscr{L}_{X}^{\bullet}(p,q)) \ar[d]_{\pi^{\star}} \ar[r]^{\imath^{\star}} &  \Gamma(Z,\mathscr{L}_{Z}^{\bullet}(p,q))\ar[d]_{\rho^{\star}} \ar[r]^{} & 0 \\
0 \ar[r] &  \Gamma(\tilde{X},\mathscr{L}_{\tilde{X},E}^{\bullet}(p,q)) \ar[r]^{} & \Gamma(\tilde{X},\mathscr{L}_{\tilde{X}}^{\bullet}(p,q)) \ar[r]^{\tilde{\imath}^{\star}} & \Gamma(E,\mathscr{L}_{E}^{\bullet}(p,q)) \ar[r] &0,}
\end{equation*}
which induces a commutative diagram of long exact sequences of hypercohomology groups:
\begin{equation}\label{general-BC-comm-diag-0}
\vcenter{
\tiny{
\xymatrix@C=0.6cm{
  \cdots \ar[r]^{}
  & \mathbb{H}^{k-1}(X, \mathscr{L}_{X, Z}^{\bullet}(p,q))\ar[d]_{\pi^{\star}}^{}  \ar[r]^{}
  & \mathbb{H}^{k-1}(X, \mathscr{L}_{X}^{\bullet}(p,q))\ar[d]_{\pi^{\star}} \ar[r]^{\imath^{\star}}
  & \mathbb{H}^{k-1}(Z,\mathscr{L}_{Z}^{\bullet}(p,q)) \ar[d]_{\rho^{\star}}  \ar[r]^{}
  & \mathbb{H}^{k}(X, \mathscr{L}_{X, Z}^{\bullet}(p,q)) \ar[d]_{\pi^{\star}}^{}  \ar[r]^{} &  \cdots \\
  \cdots  \ar[r]
  & \mathbb{H}^{k-1}(\tilde{X}, \mathscr{L}_{\tilde{X},E}^{\bullet}(p,q))\ar[r]^{}
  & \mathbb{H}^{k-1}(\tilde{X}, \mathscr{L}_{\tilde{X}}^{\bullet}(p,q)) \ar[r]^{\tilde{\imath}^{\star}}
  & \mathbb{H}^{k-1}(E,\mathscr{L}_{E}^{\bullet}(p,q)) \ar[r]
  & \mathbb{H}^{k}(\tilde{X}, \mathscr{L}_{\tilde{X},E}^{\bullet}(p,q))\ar[r]&  \cdots}
  }}
\end{equation}
Due to Lemma \ref{equal-cohom-BCcomplex} and the isomorphism \eqref{equal-cohom-relBCcomplex},
the commutative diagram \eqref{general-BC-comm-diag-0} is equal to
\begin{equation}\label{general-BC-comm-diag-1}
\vcenter{
\tiny{
\xymatrix@C=0.6cm{
 \cdots  \ar[r]^{}
  & \mathbb{H}^{k}(X, \mathscr{B}_{X, Z}^{\bullet}(p,q))\ar[d]_{\pi^{\star}}^{}  \ar[r]^{}
  & H_{BC}^{k}(X, \mathbb{C}(p,q))\ar[d]_{\pi^{\star}} \ar[r]^{\imath^{\star}}
  & H_{BC}^{k}(Z,\mathbb{C}(p,q)) \ar[d]_{\rho^{\star}}  \ar[r]^{}
  & \mathbb{H}^{k+1}(X, \mathscr{B}_{X, Z}^{\bullet}(p,q)) \ar[d]_{\pi^{\star}}^{}  \ar[r]^{} &  \cdots\\
 \cdots  \ar[r]
  & \mathbb{H}^{k}(\tilde{X}, \mathscr{B}_{\tilde{X},E}^{\bullet}(p,q))\ar[r]^{}
  & H_{BC}^{k}(\tilde{X}, \mathbb{C}(p,q)) \ar[r]^{\tilde{\imath}^{\star}}
  & H_{BC}^{k}(E,\mathbb{C}(p,q)) \ar[r]
  & \mathbb{H}^{k+1}(\tilde{X}, \mathscr{B}_{\tilde{X},E}^{\bullet}(p,q))\ar[r]&  \cdots}
  }}
\end{equation}
On the one hand, from Lemma \ref{Derived-dirim-kershcom} (or \cite[Lemma 3.4]{YY20}), the first and the fourth verticals in \eqref{general-BC-comm-diag-1} are isomorphic.
On the other hand, owing to Proposition \ref{cohom-BCcomplex-injective} and the projective bundle formulae Theorem \ref{cohom-BCcomplex-projbundle}, the rest verticals in \eqref{general-BC-comm-diag-1} are injective.
Via a standard diagram-chasing in \eqref{general-BC-comm-diag-1}, we get an isomorphism of complex vector spaces
\begin{equation}\label{BC-ck-ck}
\small{
\mathrm{coker}
\bigl[H_{BC}^{k}(X,\mathbb{C}(p,q))
\stackrel{\pi^{\star}}\rightarrow H_{BC}^{k}(\tilde{X},\mathbb{C}(p,q))\bigr]
\cong
\mathrm{coker}
\bigl[H_{BC}^{k}(Z,\mathbb{C}(p,q))
\stackrel{\rho^{\star}}\rightarrow
H_{BC}^{k}(E,\mathbb{C}(p,q))\bigr].
}
\end{equation}
As a direct result of \eqref{BC-ck-ck}, we get the blow-up formula
\begin{eqnarray}\label{abstract-BCblowup-formu}
H_{BC}^{k}(\tilde{X}, \mathbb{C}(p,q))
&\cong&
H_{BC}^{k}(X, \mathbb{C}(p,q))
\oplus
\biggl[H_{BC}^{k}(E,\mathbb{C}_{E}(p,q)) \big/
\rho^{\star} H_{BC}^{k}(Z,\mathbb{C}(p,q))\biggr] \nonumber \\
& \cong & H_{BC}^{k}(X, \mathbb{C}(p,q))
\oplus
\biggl[\bigoplus_{i=1}^{c-1} H_{BC}^{k-2i}(Z, \mathbb{C}(p-i,q-i))\biggr],
\end{eqnarray}
where the last isomorphism follows from Corollary \ref{cohom-BCcomplex-projbundle}.

In what follows, we will identify
$H_{BC}^{k}(X,\mathbb{C}(p,q))$
and
$\mathbb{H}^{k-1}(X,\mathscr{L}_{X}^{\bullet}(p,q))$
as complex vector spaces, via the canonical isomorphism in Lemma \ref{equal-cohom-BCcomplex}.
The construction of $\Phi$ is the same as the one given in the proof of \cite[Theorem 1.2]{RYY20}.
Observe that the natural inclusion of complexes
$$
\tau:\mathscr{L}_{X}^{\bullet}(p,q)\hookrightarrow\mathscr{C}_{X}^{\bullet}(p,q)
$$
induces an isomorphism
$$
\tau:\mathbb{H}^{k-1}(X, \mathscr{L}_{X}^{\bullet}(p,q))
\stackrel{\simeq}\longrightarrow
\mathbb{H}^{k-1}(X, \mathscr{C}_{X}^{\bullet}(p,q)).
$$
Likewise, we have
$$
\tilde{\tau}:\mathbb{H}^{k-1}(\tilde{X}, \mathscr{L}_{\tilde{X}}^{\bullet}(p,q))
\stackrel{\simeq}\longrightarrow
\mathbb{H}^{k-1}(\tilde{X}, \mathscr{C}_{\tilde{X}}^{\bullet}(p,q))
$$
induced by the inclusion
$\tilde{\tau}:\mathscr{L}_{\tilde{X}}^{\bullet}(p,q)\hookrightarrow\mathscr{C}_{\tilde{X}}^{\bullet}(p,q)$.
Because the blow-up morphism is a proper modification,
similar to \eqref{form-current-comdiagram1}, we get a commutative diagram
\begin{equation}\label{form-current-comdiagram}
\vcenter{
\xymatrix{
  \mathbb{H}^{k-1}(X, \mathscr{L}_{X}^{\bullet}(p,q)) \ar[d]_{\pi^{\star}} \ar[r]_{\tau}^{\simeq} & \mathbb{H}^{k-1}(X, \mathscr{C}_{X}^{\bullet}(p,q))   \\
  \mathbb{H}^{k-1}(\tilde{X}, \mathscr{L}_{\tilde{X}}^{\bullet}(p,q))
  \ar[r]_{\tilde{\tau} }^{\simeq}
  & \mathbb{H}^{k-1}(\tilde{X}, \mathscr{C}_{\tilde{X}}^{\bullet}(p,q)) \ar[u]^{\pi_{\star}} ,}}
\end{equation}
namely, $\tau=\pi_{\star}\circ \tilde{\tau} \circ \pi^{\star}$.
This implies that the morphism $\pi_{\star}$ in \eqref{form-current-comdiagram} is surjective.
Thanks to Corollary \ref{cohom-BCcomplex-projbundle}, every class
$$
[\tilde{\alpha}]\in \mathbb{H}^{k-1}(E,\mathscr{L}_{E}^{\bullet}(p,q))
\cong\mathbb{H}^{k}(E,\mathscr{B}_{E}^{\bullet}(p,q))
$$
admits a unique decomposition
$$
[\tilde{\alpha}]=
\sum_{i=0}^{c-1}h^{i}\wedge
\rho^{\ast}([\alpha]_{(k-2i-1)}),
$$
where $[\alpha]_{(k-2i-1)}\in \mathbb{H}^{k-2i-1}(Z,\mathscr{L}_{Z}^{\bullet}(p-i,q-i))$ and
$h=c_{1}(\mathcal{O}_{E}(1))\in H^{2}(E, \mathbb{Z})$.
For each $i\in \{0, 1, \cdots, c-1\}$, we can define a linear map
\begin{eqnarray*}
  \Pi_{i}: \mathbb{H}^{k-1}(E,\mathscr{L}_{E}^{\bullet}(p,q))
  &\longrightarrow& \mathbb{H}^{k-2i-1}(Z,\mathscr{L}_{Z}^{\bullet}(p-i,q-i)) \\
  {[\tilde{\alpha}]}&\longmapsto& {[\alpha]}_{(k-2i-1)}.
\end{eqnarray*}
Moreover, we get a linear map
\small{
\begin{equation}\label{BC-explicit-map}
\Phi:\mathbb{H}^{k-1}(\tilde{X},\mathscr{L}_{\tilde{X}}^{\bullet}(p,q)
\longrightarrow
\mathbb{H}^{k-1}(X,\mathscr{L}_{X}^{\bullet}(p,q))
\oplus \bigg[\bigoplus_{i=1}^{c-1} \mathbb{H}^{k-2i-1}(Z,\mathscr{L}_{Z}^{\bullet}(p-i,q-i))\bigg],
\end{equation}
}
where
$\displaystyle \Phi= \tau^{-1}\circ \pi_{\star}\circ \tilde{\tau}+\sum_{i=1}^{c-1}\Pi_{i}\circ\tilde{\imath}^{\star}$.

Finally, we will verify that $\Phi$ is an isomorphism.
Note that $\Phi$ is a linear map of finite dimensional vector spaces over $\mathbb{C}$.
By the isomorphism \eqref{abstract-BCblowup-formu}, the map $\Phi$ is an isomorphism if and only if it is injective.
Combining \eqref{general-BC-comm-diag-0} with \eqref{form-current-comdiagram} derives a commutative diagram of short exact sequences:
\begin{equation}\label{current-BC-coker}
\vcenter{
\xymatrix@C=0.5cm{
0 & \ar[l]  \mathbb{H}^{k-1}(X, \mathscr{C}_{X}^{\bullet}(p,q)) &\ar[l]_{\pi_{\star}}
\mathbb{H}^{k-1}(\tilde{X}, \mathscr{C}_{\tilde{X}}^{\bullet}(p,q))
 & \ar[l]^{}
\mathrm{ker}\,(\pi_{\star}) & \ar[l]^{}  0\\
0 \ar[r]^{} & \mathbb{H}^{k-1}(X, \mathscr{L}_{X}^{\bullet}(p,q))
\ar[u]^{\tau}_{\cong} \ar[d]_{\imath^{\star}}
\ar[r]^{\pi^{\star}} &
\mathbb{H}^{k-1}(\tilde{X}, \mathscr{L}_{\tilde{X}}^{\bullet}(p,q))
\ar[d]_{\tilde{\imath}^{\star}}
\ar[u]^{\tilde{\tau}}_{\cong}
\ar[r]^{} &
\mathrm{coker}\,(\pi^{\star}) \ar[d]_{\overline{\tilde{\imath}^{\star}}}^{\cong} \ar[u]^{\overline{\tilde{\tau}}}_{}  \ar[r]^{} & 0 \\
0 \ar[r] &
\mathbb{H}^{k-1}(Z, \mathscr{L}_{Z}^{\bullet}(p,q))
\ar[r]^{\rho^{\star}} &
\mathbb{H}^{k-1}(E, \mathscr{L}_{E}^{\bullet}(p,q))
\ar[r]^{} &
\mathrm{coker}\,(\rho^{\star})\ar[r]^{} & 0.}}
\end{equation}
The commutativity of \eqref{current-BC-coker} implies that the map
$
\overline{\tilde{\tau}}:\mathrm{coker}\,(\pi^{\star})\longrightarrow \ker\,(\pi_{\star})
$
is isomorphic, and the map $\tilde{\imath}^{\star}$ is injective on
$\ker\,(\tau^{-1} \circ \pi_{\star}\circ \tilde{\tau})$.
Choose an element
$[\tilde{\beta}]\in \mathbb{H}^{k-1}(\tilde{X}, \mathscr{L}_{\tilde{X}}^{\bullet}(p,q))$.
Assume that $\Phi([\tilde{\beta}])=0$, then by the definition of $\Phi$ we have
$[\tilde{\beta}]\in\ker\,(\tau^{-1} \circ \pi_{\star}\circ \tilde{\tau})$
and
$\tilde{\imath}^{\star}([\tilde{\beta}])=0$ and therefore $[\tilde{\beta}]=0$.
This implies that \eqref{BC-explicit-map} is isomorphic.
Equivalently, we get an isomorphism
$$
\Phi:H_{BC}^{k}(\tilde{X},\mathbb{C}(p,q))
\stackrel{\simeq}\longrightarrow
H_{BC}^{k}(X,\mathbb{C}(p,q))
\oplus \bigg
[\bigoplus_{i=1}^{c-1}H_{BC}^{k-2i}(Z,\mathbb{C}(p-i,q-i))\bigg]
$$
In particular, if $k=p+q$ we obtain explicit presentation of the Bott--Chern blow-up formula
$$
\Phi:H_{BC}^{p,q}(\tilde{X})
\stackrel{\simeq}\longrightarrow
H_{BC}^{p,q}(X)
\oplus \bigg[\bigoplus_{i=1}^{c-1} H_{BC}^{p-i,q-i}(Z)\bigg].
$$
This completes the proof.
\end{proof}
%\begin{rem}
%As $\tilde{\imath}:E\subset \tilde{X}$ is a smooth hypersurface, presumably, one can establish a natural self-intersection formula
%$\tilde{\imath}^{\star}\tilde{\imath}_{\star}(-)=(-)\wedge h$
%for Bott--Chern and truncated holomorphic de Rham hypercohomologies
%by using the idea in the proof \cite[Proposition 1]{Gri10}.
%As a result, one may construct the inverse of the morphism $\Phi$ by setting:
%$$
%\Phi^{-1}(-)=\pi^{\ast}+\sum_{i=1}^{c-1} \tilde{\imath}_{\star} \big(h^{i-1}\wedge \rho^{\star}  (-)\big).
%$$
%\end{rem}

\begin{rem}
In fact, the Bott--Chern complex can be thought of as a $\mathbb{C}$-augmentation of the direct sum of truncated holomorphic and anti-holomorphic de Rham complexes.
One can also define the integral or real Bott--Chern complex:
\begin{equation}\label{r-z-bcs}
\xymatrix@C=0.4cm{
0\ar[r]^{} & \mathbb{G}_{X}\ar[r]^{} & \mathcal{O}_{X}\oplus \bar{\mathcal{O}}_{X} \ar[r]^{} &\Omega_{X}^{1}\oplus \bar{\Omega}_{X}^{1}
  \ar[r]^{}& \cdots  \ar[r]^{}& \Omega_{X}^{p-1}\oplus \bar{\Omega}_{X}^{p-1} \ar[r]^{} & \bar{\Omega}_{X}^{p} \ar[r]^{} & \cdots   \ar[r]^{} & \bar{\Omega}_{X}^{q-1} \ar[r]^{} & 0,}
\end{equation}
where $\mathbb{G}=\mathbb{Z}$ or $\mathbb{R}$ (cf. \cite{Sch07}).
It is worth noting that each integral (resp. real) Bott--Chern complex can split as the direct sum of a Deligne complex (resp. a real Deligne complex) and a truncated anti-holomorphic de Rham complex (cf. \cite[\S\,7.c]{Sch07}).
So the blow-up formulae for integral (resp. real) Bott--Chern cohomology can obtained from the Deligne blow-up formulae and the blow-up formulae for truncated anti-holomorphic de Rham cohomology.
The blow-up formulae for integral Bott--Chern cohomology have been proven in \cite{CY21} and Wu \cite{Wu20}, independently.
In contrast to the integral and real cases, in the case of $\mathbb{G}=\mathbb{C}$, the sheaf complex \eqref{r-z-bcs} has no natural splitting.
\end{rem}

Recall that the sheaf complex $\mathbb{C}_{X}(p)$ is quasi-isomorphic to the truncated holomorphic de Rham complex $\Omega^{\bullet\geq p}_{X}[-p]$.
As a result, we have a canonical isomorphism
\begin{equation*}
\mathbb{H}^{k}(X,\mathbb{C}_{X}(p))\cong\mathbb{H}^{k}(X,\Omega^{\bullet\geq p}_{X}[-p])
\cong\mathbb{H}^{k-p}(X,\Omega^{\bullet\geq p}_{X}).
\end{equation*}
We can now state the blow-up formulae for the hypercohomology of the sheaf complex $\mathbb{C}_{X}(p)$.
\begin{thm}\label{hol-bl-u-f}
For any integer $k\geq0$, there exists a canonical isomorphism
$$
\mathbb{H}^{k}(\tilde{X},\mathbb{C}_{\tilde{X}}(p))
\stackrel{\Phi}\longrightarrow
\mathbb{H}^{k}(X,\mathbb{C}_{X}(p))
\oplus \bigg[\bigoplus_{i=1}^{c-1} \mathbb{H}^{k-2i}(Z,\mathbb{C}_{Z}(p-i))\bigg],
$$
where $\Phi$ is the same as \eqref{BC-explicit-map}.
\end{thm}
\begin{proof}(Outline of proof)
First we consider the Dolbeault resolution of $\Omega^{\bullet\geq p}_{X}[-p]$, which is a truncation of the Dolbeault double complex.
Akin to Proposition \ref{dirim-truncated-deRham}, we can establish a projective bundle formula for the truncated holomorphic de Rham cohomology considered in Theorem \ref{hol-bl-u-f}.
Then the remainder of the proof go through by using the same arguments in the proof of Theorem \ref{main-thm1}; see also \cite[Theorem 1.4]{Men2}, \cite[Theorem 4.18]{Men4}, or \cite[Proposition 8]{CY21}.
\end{proof}

\begin{rem}
Let $H$ be a local system of complex vector spaces over $X$ having a holomorphic de Rham resolution (cf. \cite[II, Proposition 5.3]{Voi02}).
By a twisted Bott--Chern complex, we mean the sheaf complex
$$
\mathscr{B}_{X}^{\bullet}(p,q; H):
\xymatrix@C=0.5cm{
0\ar[r]^{} & H\ar[r]^{}  \ar[r]^{} &\Omega_{X}^{\bullet <p}(\mathscr{H}_{X})\oplus \bar{\Omega}_{X}^{\bullet<q}(\bar{\mathscr{H}}_{X} ) \ar[r]^{} & 0,}
$$
where $\mathscr{H}_{X}=H\otimes_{\mathbb{C}_{X}} \mathcal{O}_{X}$ is a locally free sheaf equipped with a flat connection, and $\Omega_{X}^{\bullet <p}(\mathscr{H}_{X})$ is the truncated holomorphic de Rham complex with coefficients in $\mathscr{H}_{X}$.
Then the {\it twisted Bott--Chern hypercohomology} is defined to be the hypercohomology of twisted Bott--Chern complex
$$
H_{BC}^{k}(X, H(p,q)):=\mathbb{H}^{k}(X, \mathscr{B}_{X}^{\bullet}(p,q; H)),\,\,\,
\forall\,k\in\mathbb{Z}.
$$
Similar to $\mathscr{L}_{X}^{\bullet}(p,q)$, we can construct a complex of fine sheaves, denoted by $\mathscr{L}_{X}^{\bullet}(p,q; H)$, which is quasi-isomorphic to the twisted Bott--Chern complex $\mathscr{B}_{X}^{\bullet}(p,q; H)[1]$.
Following the steps in the proof of Theorem \ref{main-thm1},
one can establish a canonical isomorphism
$$
H_{BC}^{k}(\tilde{X},\pi^{-1}H(p,q))
\stackrel{\Phi}\longrightarrow
H_{BC}^{k}(X,H(p,q))
\oplus \bigg
[\bigoplus_{i=1}^{c-1}H_{BC}^{k-2i}(Z,\imath^{-1}H(p-i,q-i))\bigg].
$$
Recently, Meng \cite{Men4} gave a systemical study of twisted cohomologies with supports on complex manifolds, and established blow-up formulae for twisted Dolbeault and Bott--Chern cohomologies.
\end{rem}

\paragraph{\textbf{Added in proof:}}
After this paper had been completed, Stelzig \cite{Ste22b} told us a different approach to the projective bundle formula and the blow-up formula for Bott--Chern hypercohomology based on the main result in his recent paper \cite{Ste22} and the structure theory of double complexes developed in \cite{Ste21} and \cite{Ste21b}.

%==============================================================================

\section{Construction of bimeromorphic invariants}\label{bimero-inv}

In this section, we apply our blow-up formulae (Theorem \ref{main-thm1}) to defining some new bimermorphic invariants for compact complex manifolds.
\begin{defn}
(1) A {\it meromorphic map} $f:X \dashrightarrow Y$ of compact complex spaces is a map $f$ from $X$ to the set of subsets of $Y$ such that the following conditions hold:
\begin{enumerate}
\item[(i)] The graph $\Gamma_{f}:=\{(x, y) \in X\times Y \mid y\in f(x)\}$ is an irreducible analytic subset of $X\times Y$;
\item[(ii)] The projection $p_{X}:\Gamma_f \rightarrow X$ is a proper modification.
\end{enumerate}

(2) A  meromorphic map $f:X \dashrightarrow Y$ is called {\it bimeromorphic map} if furthermore the projection $p_{Y}:\Gamma_{f} \rightarrow Y$ is also a modification.
If there is a bimeromorphic map $f:X \dashrightarrow Y$, then we say that $X$ and $Y$ are {\it bimeromorphically equivalent}.
\end{defn}

The {\it weak factorization theorem} of Abramovich--Karu--Matsuki--W{\l}odarczyk \cite[Theorem 0.3.1]{AKMW02} asserts that each bimeromorphic map between compact complex manifolds is a composition of a finite sequence of blow-ups and blow-downs of compact complex manifolds with smooth centers.
For this reason, to show that an invariant is stable under bimeromorphic transformations, it suffices to verify its invariance under blow-ups with smooth centers.

Let $X$ be a compact complex manifold.
Observe that the sheaf complex
$$
\mathbb{C}_{X}(1):
\xymatrix@C=0.5cm{
  0 \ar[r] & \mathbb{C}_{X} \ar[r]^{} & \mathcal{O}_{X} \ar[r] & 0 }
$$
is quasi-isomorphic to the truncated holomorphic de Rham complex $\Omega^{\bullet\geq1}_{X}[-1]$,
and the Dolbeault resolution gives rise to a fine resolution of $\Omega^{\bullet\geq1}_{X}[-1]$.
Therefore $\mathbb{C}_{X}(1)$ is quasi-isomorphic to the sheaf complex:
%$$
%\xymatrix@C=0.5cm{
%  0 \ar[r] &\mathcal{A}^{1,0}_{X} \ar[r]^{d\quad\,\,}
%  & \mathcal{A}^{2,0}_{X}\oplus\mathcal{A}^{1,1}_{X} \ar[r]^{d\quad\,\,\,}
%  & \mathcal{A}^{3,0}_{X}\oplus \mathcal{A}^{2,1}_{X}\oplus\mathcal{A}^{1,2}_{X}\ar[r]^{d\quad\,\,}
%  & \mathcal{A}^{4,0}_{X}\oplus\mathcal{A}^{3,1}_{X}\oplus\mathcal{A}^{2,2}_{X}
%  \oplus\mathcal{A}^{1,3}_{X} \ar[r]^{\quad\quad\quad\quad d} & \cdots.}
%$$
\begin{equation}\label{trct-dol}
\xymatrix@C=0.5cm{
  0 \ar[r] &\mathcal{A}^{1,0}_{X} \ar[r]^{d\quad\,\,}
  & \mathcal{A}^{2,0}_{X}\oplus\mathcal{A}^{1,1}_{X} \ar[r]^{d\quad\,\,\,}
  & \mathcal{A}^{3,0}_{X}\oplus \mathcal{A}^{2,1}_{X}\oplus\mathcal{A}^{1,2}_{X}\ar[r]^{d\quad\,\,}
  & \mathcal{A}^{4,0}_{X}\oplus\mathcal{A}^{3,1}_{X}\oplus\mathcal{A}^{2,2}_{X}
  \oplus\mathcal{A}^{1,3}_{X} \ar[r]^{\quad\quad\quad\quad d} & \cdots.}
\end{equation}
Recall that the sheaf of germs of pluriharmonic functions $\mathcal{H}_{X}$ has a fine resolution:
\begin{equation}\label{pluri-har-com}
\xymatrix@C=0.5cm{
  0 \ar[r] &\mathcal{A}^{0,0}_{X} \ar[r]^{\partial\bar{\partial}}
  & \mathcal{A}^{1,1}_{X} \ar[r]^{d\quad\,\,\,}
  &  \mathcal{A}^{1,2}_{X}\oplus\mathcal{A}^{2,1}_{X}\ar[r]^{d\quad\,\,}
  & \mathcal{A}^{1,3}_{X}\oplus\mathcal{A}^{2,2}_{X}\oplus\mathcal{A}^{3,1}_{X} \ar[r]^{\quad\quad \quad d} & \cdots.}
\end{equation}
In particular, we can define a morphism from \eqref{pluri-har-com} to \eqref{trct-dol}:
\begin{equation}\label{map-d-com}
\vcenter{
\small{
\xymatrix@C=0.5cm{
  0 \ar[r]^{} & \mathcal{A}^{0,0}_{X} \ar[d]_{\partial} \ar[r]^{\partial\bar{\partial}} & \mathcal{A}^{1,1}_{X}
   \ar[d]_{\jmath} \ar[r]^{d\quad\quad} & \mathcal{A}^{2,1}_{X}\oplus\mathcal{A}^{1,2}_{X}
    \ar[d]_{\jmath}\ar[r]^{d\quad\quad}&\mathcal{A}^{3,1}_{X}\oplus\mathcal{A}^{2,2}_{X}
  \oplus\mathcal{A}^{1,3}_{X}
   \ar[r]^{\quad\quad\quad d}\ar[d]_{\jmath}&\cdots\\
  0 \ar[r]^{} & \mathcal{A}^{1,0}_{X} \ar[r]^{d\quad\,\,}
  & \mathcal{A}^{2,0}_{X}\oplus\mathcal{A}^{1,1}_{X} \ar[r]^{d\quad\,\,\,}
  & \mathcal{A}^{3,0}_{X}\oplus \mathcal{A}^{2,1}_{X}\oplus\mathcal{A}^{1,2}_{X}\ar[r]^{d\quad\,\,}
  & \mathcal{A}^{4,0}_{X}\oplus\mathcal{A}^{3,1}_{X}\oplus\mathcal{A}^{2,2}_{X}
  \oplus\mathcal{A}^{1,3}_{X} \ar[r]^{\quad\quad\quad\quad d} & \cdots.  }}}
\end{equation}
Here $\jmath$ is the inclusion map.
The morphism \eqref{map-d-com} induces a map of cohomology groups
\begin{equation}\label{inv-1}
  \mathfrak{C}^{k}:H^{k-1}(X,\mathcal{H}_{X}) \longrightarrow H^{k-1}(X,\Omega_{X}^{\bullet \geq 1}),
\end{equation}
for any integer $k\geq1$.

We are now in a position to present the first result of this section.
\begin{thm}\label{birat-inv-thm1}
Both the kernel and the cokernel of the map \eqref{inv-1} are bimeromorphic invariants.
In particular, the integer
$$
\spadesuit^{k}(X)
=\dim_{\mathbb{C}}\, H^{k-1}(X,\Omega_{X}^{\bullet \geq 1})
-
\dim_{\mathbb{C}}\,H^{k-1}(X,\mathcal{H}_{X})
$$
is a bimeromorphic invariant of $X$.
\end{thm}

\begin{proof}
Let $\pi:\tilde{X}\rightarrow X$ be the blow-up of $X$ along a closed complex submanifold $Z$ with complex codimension $c\geq2$.
Observe that both $\mathscr{B}_{Z}(1-i,1-i)$ and $\mathbb{C}_{Z}(1-i)$ are identical to the locally constant sheaf $\mathbb{C}_{Z}$ for $i\geq 1$.
This implies
$$
H^{k-2i}_{BC}(Z,\mathbb{C}(1-i,1-i))
=H_{dR}^{k-2i}(Z; \mathbb{C})=
\mathbb{H}^{k-2i}(Z,\mathbb{C}_{Z}(1-i)),
$$
for $i\geq 1$.
According to the blow-up formulae in Theorems \ref{main-thm1} and \ref{hol-bl-u-f},
we have two canonical isomorphisms
\begin{eqnarray}\label{b-l-app}
H_{BC}^{k}(\tilde{X}, \mathbb{C}(1,1))
&\cong &
H_{BC}^{k}(X, \mathbb{C}(1,1))
\oplus
\bigg[\bigoplus\limits_{i=1}^{c-1}   H_{BC}^{k-2i}(Z, \mathbb{C}(1-i,1-i))\bigg]\nonumber\\
&=&
H_{BC}^{k}(X, \mathbb{C}(1,1))
\oplus
\bigg[\bigoplus\limits_{i=1}^{c-1}   H_{dR}^{k-2i}(Z; \mathbb{C})\bigg]
\end{eqnarray}
and
\begin{eqnarray}\label{c-l-app}
\mathbb{H}^{k}(\tilde{X}, \mathbb{C}_{\tilde{X}}(1))
&\cong &
\mathbb{H}^{k}(X, \mathbb{C}_{X}(1))
\oplus
\bigg[\bigoplus_{i=1}^{c-1}  \mathbb{H}^{k-2i}(Z, \mathbb{C}_{Z}(1-i))\bigg]\nonumber\\
&=&
\mathbb{H}^{k}(X, \mathbb{C}_{X}(1))
\oplus
\bigg[\bigoplus\limits_{i=1}^{c-1}   H_{dR}^{k-2i}(Z; \mathbb{C})\bigg],
\end{eqnarray}
for any $k\in \mathbb{N}$.
From Proposition \ref{pluri-har}, \eqref{map-d-com}, \eqref{b-l-app} and \eqref{c-l-app}
we get a commutative diagram of finite-dimensional complex vector spaces:
\begin{equation}\label{vai-1-0}
\vcenter{
\xymatrix{
  H^{k-1}(\tilde{X},\mathcal{H}_{\tilde{X}}) \ar[d]_{\tilde{\mathfrak{C}}^{k}}
  \ar[r]_{}^{\simeq\qquad\qquad} & H^{k-1}(X,\mathcal{H}_{X})\oplus
   \bigg[\bigoplus\limits_{i=1}^{c-1}   H_{dR}^{k-2i}(Z; \mathbb{C})\bigg] \ar[d]^{\mathfrak{C}^{k}\oplus\mathrm{id}} \\
  H^{k-1}(\tilde{X},\Omega_{\tilde{X}}^{\bullet\geq1})
  \ar[r]_{}^{\simeq\qquad\qquad\quad} & H^{k-1}(X,\Omega_{X}^{\bullet\geq1})\oplus
  \bigg[\bigoplus\limits_{i=1}^{c-1}  H_{dR}^{k-2i}(Z; \mathbb{C})\bigg]. }}
\end{equation}
It follows from the commutativity of \eqref{vai-1-0} that $\ker\,\tilde{\mathfrak{C}}^{k}$ (resp. $\mathrm{coker}\,\tilde{\mathfrak{C}}^{k}$) is isomorphic to $\ker\,\mathfrak{C}^{k}$ (resp. $\mathrm{coker}\,\mathfrak{C}^{k}$).
Particularly, due to Proposition \ref{pluri-har} and Lemma \ref{equal-cohom-BCcomplex} we obtain that, for each integer $k\geq1$, we have:
\begin{eqnarray*}
\spadesuit^{k}(X)
&=&
\mathrm{dim}_{\mathbb{C}}\,\mathbb{H}^{k}(X,\mathbb{C}_{X}(1))
-
\mathrm{dim}_{\mathbb{C}}\,H^{k}_{BC}(X,\mathbb{C}_{X}(1,1)) \\
&=&
\dim_{\mathbb{C}}\, H^{k-1}(X,\Omega_{X}^{\bullet \geq 1})
-
\dim_{\mathbb{C}}\,H^{k-1}(X,\mathcal{H}_{X})\\
&=&
\dim_{\mathbb{C}}\, H^{k-1}(\tilde{X},\Omega_{\tilde{X}}^{\bullet \geq 1})
-
\dim_{\mathbb{C}}\,H^{k-1}(\tilde{X},\mathcal{H}_{\tilde{X}})\\
&=&\spadesuit^{k}(\tilde{X}).
\end{eqnarray*}
Due to the weak factorization theorem \cite[Theorem 0.3.1]{AKMW02}, we get that $\spadesuit^{k}(X)$ is a bimeromorphic invariant.
\end{proof}

Likewise,
by comparing the blow-up formulae of Bott--Chern hypercohomology and de Rham cohomology,
we have

\begin{thm}\label{real-inv}
Let $X$ be a compact complex manifold of dimension $n\geq 2$.
Then the integer
$$
\clubsuit^{k}(X)
:= \dim_{\mathbb{C}} H^{k}_{dR}(X; \mathbb{C})
-\dim_{\mathbb{C}}\,H^{k-1}(X,\mathcal{H}_{X})
$$
is a bimeromorphic invariant of $X$.
\end{thm}

\begin{proof}
According to the blow-up formula of de Rham cohomology,
there exists an isomorphism
\begin{equation}\label{deR-app}
H_{dR}^{k}(\tilde{X}; \mathbb{C})
\cong   H_{dR}^{k}(X; \mathbb{C})
\oplus  \bigg[\bigoplus \limits_{i=1}^{c-1} H_{dR}^{k-2i}(Z; \mathbb{C})\bigg].
\end{equation}
By \eqref{b-l-app},
we have
\begin{equation}\label{bc-l-app}
H_{BC}^{k}(\tilde{X}, \mathbb{C}(1,1))
\cong
H_{BC}^{k}(X, \mathbb{C}(1,1))
\oplus
\bigg[\bigoplus \limits_{i=1}^{c-1} H_{dR}^{k-2i}(Z; \mathbb{C})\bigg].
\end{equation}
Comparing \eqref{deR-app} and \eqref{bc-l-app},
we get
$$
\clubsuit^{k}(\tilde{X})=\clubsuit^{k}(X)
$$
and therefore we conclude the proof by the weak factorization theorem (\cite[Theorem 0.3.1]{AKMW02}).
\end{proof}

\begin{rem}
When $k\geq n+2$, from definition, we have
$$
H^{k}_{BC}(X,\mathbb{C}_{X}(1,1))
\cong
\mathbb{H}^{k}(X,\mathbb{C}_{X}(1))
\cong
H_{dR}^{k}(X; \mathbb{C}),
$$
and therefore the bimeromorphic invariants $\spadesuit^{k}(X)$ and $\clubsuit^{k}(X)$ are equal to zero for $k\geq n+2$.
\end{rem}

Let $X$ be a compact complex manifold of complex dimension $n$.
By definition, the identify maps induce natural morphisms among the Bott--Chern, de Rham, Dolbeault, $\partial$ and Appeli cohomology as follows.
\begin{equation*}
\xymatrix@C=0.5cm{
 & H^{\bullet,\bullet}_{BC}(X) \ar[d]\ar[ld]\ar[rd] & \\
 H^{\bullet,\bullet}_{\partial}(X) \ar[rd] & H^{\bullet}_{dR}(X;\mathbb{C}) \ar[d] & H^{\bullet,\bullet}_{\bar{\partial}}(X) \ar[ld] \\
 & H^{\bullet,\bullet}_{A}(X) &
}
\end{equation*}
We denote the natural morphism by
\begin{eqnarray*}
\mathfrak{I}^{p,q}: H_{BC}^{p,q}(X) &\longrightarrow& H_{\bar{\partial}}^{p,q}(X)\\
{[\alpha]}_{BC} &\longmapsto& {[\alpha]}_{\bar{\partial}}
\end{eqnarray*}
for any $0\leq p,q\leq n$.
In general, the maps $\mathfrak{I}^{p,q}$ are neither surjective nor injective.
The morphisms $\mathfrak{I}^{p,q}$ are injective for all $0\leq p,q\leq n$ if and only if $X$ satisfies the $\partial\bar{\partial}$-lemma.
In three dimensional cases, as an application of Theorem \ref{main-thm1}, we have

\begin{thm}\label{bc-d-inv-1}
Suppose that $X$ is a compact complex threefold.
For any $0\leq p,q\leq 3$,
both the kernel and the cokernel of the natural map
$$
\mathfrak{I}^{p,q}:H_{BC}^{p,q}(X)\longrightarrow H_{\bar{\partial}}^{p,q}(X)
$$
are stable under bimeromorphisms.
\end{thm}

\begin{proof}
Let $Z$ be a closed complex submanifold of $X$ with complex codimension $c$ ($c=2, 3$) and $\pi:\tilde{X}\rightarrow X$ the blow-up of $X$ with the center $Z$.
Since $X$ is of $3$-dimensional, the center $Z$ is a point or a complex curve and therefore the natural map
$$
\mathfrak{I}_{Z}^{\bullet,\bullet}:H_{BC}^{\bullet,\bullet}(Z)
\longrightarrow H_{\bar{\partial}}^{\bullet,\bullet}(Z)
$$
is isomorphic.
On account of Theorem \ref{main-thm1} and the explicit Dolbeault blow-up formula \cite[Theorem 1.2]{RYY20}, we obtain the following commutative diagram:
\begin{equation}\label{bc-d}
\vcenter{
\xymatrix{
  H^{p,q}_{BC}(\tilde{X}) \ar[d]_{\tilde{\mathfrak{I}}^{p,q}}
  \ar[r]_{}^{\simeq\qquad\qquad\quad} & H^{p,q}_{BC}(X)
   \oplus \bigg[\bigoplus\limits_{i=0}^{c-1} H_{BC}^{p-i,q-i}(Z)\bigg] \ar[d]^{\mathfrak{I}^{p,q}\oplus\mathfrak{I}_{Z}^{\bullet,\bullet}} \\
  H^{p,q}_{\bar{\partial}}(\tilde{X})
  \ar[r]_{}^{\simeq\qquad\qquad\quad} &
  H^{p,q}_{\bar{\partial}}(X)
  \oplus \bigg[\bigoplus\limits_{i=0}^{c-1} H_{\bar{\partial}}^{p-i,q-i}(Z)\bigg].}}
\end{equation}
As $\mathfrak{I}_{Z}^{\bullet,\bullet}$ is an isomorphism, the commutativity of \eqref{bc-d} implies that $\ker\,\mathfrak{I}$ (resp. $\mathrm{coker}\,\mathfrak{I}$) is isomorphic to $\ker\,\tilde{\mathfrak{I}}$ (resp. $\mathrm{coker}\,\tilde{\mathfrak{I}}$), and combining with the weak factorization theorem concludes the proof.
\end{proof}

It is noteworthy that $H^{p,q}_{BC}(X)$ and $H^{p,q}_{\bar{\partial}}(X)$ are not bimeromorphic invariants unless $p=0$ or $q=0$.
As a corollary of Theorem \ref{bc-d-inv-1},
for compact complex threefolds, the integer
$$
\Delta_{BC,\bar{\partial}}^{p,q}(X)=
h_{BC}^{p,q}(X)-h_{\bar{\partial}}^{p,q}(X), \; \forall\; 0\leq p,q\leq 3
$$
is stable under blow-ups and hence is a bimeromorphic invariant by the weak factorization theorem, see also \cite[Corollary 1.5]{RYY17}.
Likewise, we can construct the bimeromorphic invariants of threefolds via the natural morphism
\begin{equation}\label{bjmp}
\mathfrak{I}^{\bullet,\bullet}:
H_{BC}^{\bullet,\bullet}(X)
\longrightarrow
H_{\bigstar}^{\bullet,\bullet}(X),
\end{equation}
where  $\bigstar\in \{\partial, A\}$.
In general, for a compact complex manifold $X$ with complex dimension $\geq4$, the kernel and the cokernel of the natural map \eqref{bjmp} are not bimeromorphic invariants.
The reason lies in the fact that for different cohomologies the contributions of the center $Z$ in the blow-up formulae are not equivalent to each other.

According to a result of Stelzig \cite{Ste21c}, we know that if a bimeromorphic invariant is a universal $\mathbb{Z}$-linear combination or congruence of Hodge and Chern numbers, then it is a linear combination or congruence of the $(p,0)$ or $(0,p)$-Hodge numbers only.
In view of the important role of the Bott--Chern (hyper) cohomology in non-K\"{a}hler complex geometry, it is natural to consider the following:
\begin{prob}
Which linear combinations or congruences of Bott--Chern and Hodge numbers are bimeromorphic invariants of compact complex manifolds without the $\partial\bar{\partial}$-lemma?
\end{prob}

For a compact complex manifold $X$ with $\mathrm{dim}_{\mathbb{C}}(X)\leq4$, Stelzig \cite[Theorem F]{Ste21c} determined all linear combinations of cohomological invariants which admit the bimeromorphic invariance.
Comparing Theorems \ref{thm1}-\ref{bc-d-inv-0} with \cite[Theorem F]{Ste21c}, an interesting problem is:\footnote{This problem is suggested to us by Stelzig \cite{Ste22b}.}
\begin{prob}\label{prob4.2}
Let $X$ be a compact complex threefold or fourfold, then which zigzags the invariants $\spadesuit^{\ast}(X)$ and $\clubsuit^{\ast}(X)$ count exactly?
\end{prob}

%====================================================================
%====================================================================
\section{Examples}\label{example}

In this section, we first compute the Bott--Chern hypercohomology groups for some compact complex surfaces and threefolds, and then we consider their invariants defined in Theorem \ref{thm1}.
In general, it is difficult to compute the Bott--Chern hypercohomology of a higher dimensional compact complex manifold explicitly.
However, thanks to Lemma \ref{equal-cohom-BCcomplex},
we can write down the Bott--Chern hypercohomology groups of surfaces and threefolds
via the fine sheaf complex $\mathscr{L}^{\bullet}(p,q)$.
To be more specific, we have the following tables which record the Bott--Chern hypercohomology groups:
$$
\tiny{
\left.\begin{array}{|c|c|c|c|}
\hline \textrm{surface}  & \mathscr{B}_{S}^{\bullet}(1,1) & \mathscr{B}_{S}^{\bullet}(1,2) & \mathscr{B}_{S}^{\bullet}(2,2) \\
\hline H_{BC}^{1} & \mathbb{C} & \mathbb{C} & \mathbb{C} \\
\hline H_{BC}^{2} & H_{BC}^{1,1}(S) & H_{A}^{0,1}(S) & \frac{\ker(d:A^{1,0}\oplus A^{0,1}\rightarrow A^{1,1})}{d(A^{0,0})}  \\
\hline H_{BC}^{3} & \frac{\ker(d: A^{2,1}\oplus A^{1,2}\rightarrow A^{2,2})}{d(A^{1,1})} & H_{BC}^{1,2}(S) & H_{A}^{1,1}(X)
 \\
\hline H_{BC}^{4} & \mathbb{C} & \mathbb{C} & \mathbb{C} \\
\hline \end{array}\right.
}
$$
$$
\tiny{
\left.\begin{array}{|c|c|c|c|c|c|c|}
\hline \textrm{threefold}  & \mathscr{B}_{X}^{\bullet}(1,1) & \mathscr{B}_{X}^{\bullet}(1,2) & \mathscr{B}_{X}^{\bullet}(1,3) & \mathscr{B}_{X}^{\bullet}(2,2) & \mathscr{B}_{X}^{\bullet}(2,3) & \mathscr{B}_{X}^{\bullet}(3,3) \\
\hline H_{BC}^{1} & \mathbb{C} & \mathbb{C} & \mathbb{C} & \mathbb{C} & \mathbb{C} & \mathbb{C} \\
\hline H_{BC}^{2} & H_{BC}^{1,1}(X) & H_{A}^{0,1}(X) & H_{\bar{\partial}}^{0,1}(X)& H_{dR}^{1}(X; \mathbb{C}) & \frac{\ker(d:A^{1}\rightarrow A^{1,1})}{dA^{0,0}} & H_{dR}^{1}(X; \mathbb{C}) \\
\hline H_{BC}^{3} & \frac{\ker d|_{A^{2,1}\oplus A^{1,2}}}{d(A^{1,1})} & H_{BC}^{1,2}(X) & H_{A}^{0,2}(X) & H_{A}^{1,1}(X) & \frac{\ker(d:A^{1,1}\oplus A^{0,2}\rightarrow A^{1,2})}{\bar{\partial}A^{1,0}+dA^{0,1}} & \frac{\ker(d:A^{2}\rightarrow A^{2,1}\oplus A^{1,2})}{dA^{1}} \\
\hline H_{BC}^{4} & \frac{\ker d|_{A^{4}}}{d(A^{2,1}\oplus A^{1,2})} & \frac{\ker d|_{A^{2,2}\oplus A^{1,3}}}{d(A^{1,2})} & H_{BC}^{1,3}(X) & H_{BC}^{2,2}(X) & H_{A}^{1,2}(X) & \frac{\ker(d:A^{2,1}\oplus A^{1,2}\rightarrow A^{2,2})}{\bar{\partial}A^{2,0}+dA^{1,1}+\partial A^{0,2})} \\
\hline H_{BC}^{5} & H_{dR}^{5}(X; \mathbb{C}) & \frac{\ker d|_{A^{5}}}{d(A^{2,2}\oplus A^{1,3})} & H_{\partial}^{2,3}(X) & \frac{\ker d|_{A^{5}}}{d(A^{2,2})} & H_{BC}^{2,3}(X) & H_{A}^{2,2}(X)\\
\hline H_{BC}^{6} & H_{dR}^{6}(X; \mathbb{C}) & H_{dR}^{6}(X; \mathbb{C}) & H_{dR}^{6}(X; \mathbb{C}) & H_{dR}^{6}(X; \mathbb{C}) & H_{dR}^{6}(X; \mathbb{C}) &H_{dR}^{6}(X; \mathbb{C}) \\
\hline \end{array}\right.
}
$$
where $A^{k}:=\Gamma(X, \mathcal{A}_{X}^{k})$ and $A^{s,t}:=\Gamma(X, \mathcal{A}_{X}^{s,t})$.

Let $\mathcal{E}^{\bullet}$ be a sheaf complex of $\mathbb{C}$-modules on a compact complex manifold $X$ which has finite-dimensional hypercohomology groups.
By the Euler characteristics of $\mathcal{E}^{\bullet}$,
we mean the alternative sum of the dimensions of its hypercohomology groups, namely,
$$
\chi(\mathcal{E}^{\bullet})=
\sum_{i\in\mathbb{Z}}(-1)^{i}\mathrm{dim}_{\mathbb{C}}\,\mathbb{H}^{i}(X,\mathcal{E}^{\bullet}).
$$
Consider the short exact sequence sheaf complexes
\begin{equation}\label{example-ex-seq}
\xymatrix@C=0.5cm{
  0 \ar[r] & \bar{\mathcal{O}}_{X}[-1] \ar[r]^{} & \mathscr{B}^{\bullet}_{X}(1,1) \ar[r]^{} & \mathbb{C}_{X}(1) \ar[r] & 0. }
\end{equation}
The exactness of \eqref{example-ex-seq} implies
\begin{eqnarray}\label{BCcoHnumber-chi-1}
  \chi(\mathscr{B}_{X}^{\bullet}(1,1))
   &=& \chi(\mathbb{C}_{X}(1))+\chi(\bar{\mathcal{O}}_{X}[-1])
= \chi(\mathbb{C}_{X}(1))-\chi(X,\bar{\mathcal{O}}_{X})\nonumber \\
   &=& \chi(\mathbb{C}_{X}(1))-\chi(\mathcal{O}_{X}),
\end{eqnarray}
Likewise, consider the short exact sequence
\begin{equation*}\label{example-ex-seq-0}
\xymatrix@C=0.5cm{
  0 \ar[r] & (\mathcal{O}_{X}\oplus\bar{\mathcal{O}}_{X})[-1] \ar[r]^{} & \mathscr{B}^{\bullet}_{X}(1,1) \ar[r]^{} & \mathbb{C}_{X}\ar[r] & 0. }
\end{equation*}
We have
\begin{eqnarray}\label{BCcoHnumber-chi-2}
  \chi(\mathscr{B}_{X}^{\bullet}(1,1))
   &=& \chi(\mathbb{C}_{X})+\chi((\mathcal{O}_{X}\oplus\bar{\mathcal{O}}_{X}[-1]))
= \chi(\mathbb{C}_{X})-\chi(\mathcal{O}_{X}\oplus\bar{\mathcal{O}}_{X})\nonumber \\
   &=& \chi(\mathbb{C}_{X})-2\chi(\mathcal{O}_{X}).
\end{eqnarray}
In particular, if $X$ is a threefold, from \eqref{BCcoHnumber-chi-1} and \eqref{BCcoHnumber-chi-2},
we get
\begin{equation}\label{chi-1}
h_{BC}^{4}(X, \mathbb{C}(1,1))-h_{BC}^{3}(X, \mathbb{C}(1,1))
=\chi(\mathbb{C}_{X}(1))-\chi(\mathcal{O}_{X})-h_{BC}^{1,1}(X)+b_{5}(X)
\end{equation}
and
\begin{equation*}\label{chi-2}
h_{BC}^{4}(X, \mathbb{C}(1,1))-h_{BC}^{3}(X, \mathbb{C}(1,1))
=\chi(X)-2\chi(\mathcal{O}_{X})-h_{BC}^{1,1}(X)+b_{5}(X),
\end{equation*}
where $\chi(X)$ is the Euler characteristic of $X$ and $b_{5}(X)$ is the 5-th Betti number of $X$.
Recall that the sheaf complex
$$
\mathbb{C}_{X}(1):
\xymatrix@C=0.5cm{
  0 \ar[r] & \mathbb{C}_{X} \ar[r]^{} & \mathcal{O}_{X} \ar[r] & 0 }
$$
is quasi-isomorphic to the truncated holomorphic de Rham complex $\Omega^{\bullet\geq1}_{X}[-1]$.
Put $n=\dim_{\mathbb{C}}\,X$.
Due to the Serre duality for truncated holomorphic de Rham complexes \cite[Theorem 1.3]{Men1},
for any $0\leq l\leq n$, we have
\begin{equation*}\label{coHnumber-C(1)}
  \mathbb{H}^{l}(X,\mathbb{C}_{X}(1))
  \cong \mathbb{H}^{l}(X,\Omega^{\bullet\geq1}_{X}[-1])= \mathbb{H}^{l-1}(X,\Omega^{\bullet\geq1}_{X})\cong\mathbb{H}^{2n+1-l}(X,\Omega^{\bullet<n}_{X}).
\end{equation*}
If $n=3$, taking the Dolbeault resolution of $\Omega^{\bullet<3}_{X}$, we can compute $\mathbb{H}^{l}(X,\mathbb{C}_{X}(1))$ via the truncated Dolbeault double complex.
Moreover, we get
$$
\mathbb{H}^{5}(X,\mathbb{C}_{X}(1))
\cong
H_{dR}^{5}(X; \mathbb{C})
\cong
H_{BC}^{5}(X,\mathbb{C}(1,1))
$$
and
$$
\mathbb{H}^{6}(X,\mathbb{C}_{X}(1))
\cong
H_{dR}^{6}(X; \mathbb{C})
\cong
H_{BC}^{6}(X,\mathbb{C}(1,1)).
$$

%--------------------------------------------------------------------------------------------------------------------

\subsection{Compact complex surfaces}

Let $S$ be a compact complex surface.
Because of the $E_{1}$-degeneracy of the Fr\"{o}licher spectral sequence of $S$, the hypercohomology of the sheaf complex $\mathbb{C}_{S}(1)$ can be read off from the Dolbeault cohomology and hence we have:
$$
\dim_{\mathbb{C}}\mathbb{H}^{1}(S, \mathbb{C}_{S}(1))=h_{\bar{\partial}}^{1,0}(S),\,\,\,\;
\dim_{\mathbb{C}}\mathbb{H}^{2}(S, \mathbb{C}_{S}(1))=h_{\bar{\partial}}^{2,0}(S)+h_{\bar{\partial}}^{1,1}(S),
$$
$$
\dim_{\mathbb{C}}\mathbb{H}^{3}(S, \mathbb{C}_{S}(1))=b_{1}(S),\,\,\,\;
\dim_{\mathbb{C}}\mathbb{H}^{4}(S, \mathbb{C}_{S}(1))=h_{\bar{\partial}}^{2,2}(S).
$$
From \eqref{BCcoHnumber-chi-2}, we get
$\chi(\mathscr{B}_{S}^{\bullet}(1,1))=\chi(S)-2\chi(\mathcal{O}_{S})$
and this implies
$$
\dim_{\mathbb{C}}H_{BC}^{3}(S, \mathbb{C}(1,1))
=h_{BC}^{1,1}+2\chi(\mathcal{O}_{S})-\chi(S).
$$
As a result, the bimeromorphic invariants $\spadesuit^{\bullet}(S)$ are
\begin{eqnarray*}
\spadesuit^{1}(S)&=&h_{\bar{\partial}}^{1,0}(S)-1, \\
\spadesuit^{2}(S)&=&h_{\bar{\partial}}^{1,1}(S)-h_{BC}^{1,1}(S)+h_{\bar{\partial}}^{2,0}(S),\\
\spadesuit^{3}(S)&=&b_{1}(S)-h_{BC}^{1,1}-2\chi(\mathcal{O}_{S})+\chi(S).
\end{eqnarray*}
To be more specific, we have the following table of invariants for some classical compact complex surfaces.
$$
\left.\begin{array}{|c|c|c|c|c|}
\hline  S & \spadesuit^{1} & \spadesuit^{2} & \spadesuit^{3} & \spadesuit^{4} \\
\hline \PV^{2} & -1 & 0 & 0 & 0 \\
\hline \textrm{K3 surfaces} & -1 & 1 & 0 & 0 \\
\hline \textrm{Torus} & 1 & 1 & 0 & 0 \\
\hline \textrm{Primary Kodaria} & 0 & 0 & 0 & 0 \\
\hline \textrm{Second Kodaria} & -1 & -1 & 0 & 0 \\
\hline
\end{array}\right.
$$

%=====================================
\subsection{Complex nilmanifolds}
In this subsection, we will focus on some three-dimensional complex nilmanifolds.
First we review some basics on complex nilmanifolds.
Let $G$ be a simply-connected nilpotent Lie group with Lie algebra $\mathfrak{g}$ and $\Gamma\subset G$ a lattice with maximal rank.
The quotient space $X:=\Gamma\setminus G$ is called a \emph{real nilmanifold}.
Owing to a result by Nomizu \cite{No54}, there exists an isomorphism from the
Lie-algebra de Rham cohomology of $\mathfrak{g}$ to the de Rham cohomology of $X$, which was induced by the natural inclusion of left-invariant differential forms in the de Rham complex of $X$.
Moreover, if $\mathfrak{g}$ is endowed with an invariant complex structure $J$ and set $\mathfrak{g}_{\mathbb{C}}:=\mathfrak{g}\otimes_{\mathbb{R}} \mathbb{C}$,
then $(X,J)$ becomes a compact complex manifold called the \emph{complex nilmanifold}.
Moreover, it is conjectured that the natural inclusion of complexes
$$\iota: (\wedge^{p,\bullet}(\mathfrak{g}_{\mathbb{C}}^{\ast}), \bar{\partial})
\hookrightarrow (\Gamma(X, \mathcal{A}_{X}^{p,\bullet}),  \bar{\partial})$$
is a quasi-isomorphism,
namely, the natural inclusion  $\iota$ induces an isomorphism
\begin{equation}\label{nil-Lie-Dol-cohom}
\iota:H^{p,q}_{\bar{\partial}}(\mathfrak{g})
\stackrel{\simeq}\longrightarrow
H^{p,q}_{\bar{\partial}}(X)
\end{equation}
for any $p,q\in \mathbb{N}$, see \cite{CFGU00,CF01,Rol09} etc.
Based on this isomorphism,
the various cohomology of complex nilmanifolds can be computed in terms of correspondent Lie-algebra tcohomology, we refer to \cite{Sak76,CFGU00,CF01,Rol09,Ang13,COUV16,FRR19} and the references therein.
For our purpose, if \eqref{nil-Lie-Dol-cohom} holds, by the standard spectral sequence theory,
the truncated de Rham cohomology $\mathbb{H}^{\bullet}(X, \mathbb{C}_{X}(1))$ can be computed by using the Lie-algebra truncated de Rham cohomology.
Put $\mathscr{L}_{\mathfrak{g}_{\mathbb{C}}^{\ast}}^{\bullet}(1,1)$ the complex of vector spaces analogous to $\mathscr{L}_{X}^{\bullet}(1,1)$, and denote by
$H_{BC}^{l}(\mathfrak{g}_{\mathbb{C}}^{\ast};\mathbb{C}(1,1))
:=H^{l-1}(\mathscr{L}_{\mathfrak{g}_{\mathbb{C}}^{\ast}}(1,1))$
the associated Lie algebra Bott--Chern hypercohomolgy.
\begin{lem}\label{nil-Lie-BC-hypercoh}
If the isomorphism \eqref{nil-Lie-Dol-cohom} holds, then the natural inclusion $\iota$ induces an isomorphism of Bott--Chern hypercohomolgy
\begin{equation}\label{nil-Lie-BC-hypercoh-mor}
\iota:
H_{BC}^{l}(\mathfrak{g}_{\mathbb{C}}^{\ast};\mathbb{C}(1,1))
\stackrel{\simeq}\longrightarrow
H_{BC}^{l}(X, \mathbb{C}(1,1))
\end{equation}
for any $l\in \mathbb{N}$.
\end{lem}

\begin{proof}
According to \cite[Theorem 3.7]{Ang13}, the assertion is true for $l=2$ which is the isomorpohism between the Lie-algebra Bott--Chern cohomology $H_{BC}^{1,1}(\mathfrak{g}_{\mathbb{C}}^{\ast})$ and the Bott--Chern cohomology $H_{BC}^{1,1}(X)$.
Consider the case of $l\geq 3$.
We first verify the injectivity of \eqref{nil-Lie-BC-hypercoh-mor}.
The strategy of the proof is the same as the one used in the proofs of \cite[Lemma 9]{CF01} and \cite[Lemma 3.6]{Ang13}.
Suppose $g$ is a $G$-left-invariant Hermitian metric on $X$.
Observe that both $d$ and $d^{\ast}$ preserve the $G$-left-invariant forms, and so is the operator $\Delta_{l}=\delta^{\ast}_{l}\delta_{l}+\delta_{l-1}\delta^{\ast}_{l-1}$.
Therefore, by Hodge's theorem, there exists a decomposition of $G$-left-invariant forms
$$
\mathscr{L}^{l}_{\mathfrak{g}_{\mathbb{C}}^{\ast}}(1,1)=\bigoplus\limits_{s+t=l+1\atop{s\geq p,t\geq q}} \bigwedge^{s,t} \mathfrak{g}_{\mathbb{C}}^{\ast}
=\ker\, \Delta_{l}\, \oplus \, \mathrm{im}\,\delta_{l-1}\oplus \, \mathrm{im}\,\delta^{\ast}_{l}.
$$
Assume that $[\alpha]$ is a class in $H_{BC}^{l}(\mathfrak{g}_{\mathbb{C}}^{\ast};\mathbb{C}(1,1))$ satisfying
 $\iota([\alpha])=0\in H_{BC}^{l}(X, \mathbb{C}(1,1))$.
This implies that there exists a form
$$\beta\in \mathscr{L}^{l-1}_{X}(1,1)=\bigoplus\limits_{s+t=l \atop{s\geq p,t\geq q}}\mathcal{A}^{s,t}(X)$$
such that $\alpha=\delta_{l-1}(\beta)$.
Up to a zero term in $H_{BC}^{l}(\mathfrak{g}_{\mathbb{C}}^{\ast};\mathbb{C}(1,1))$,
we may suppose that $\beta\in \big[\iota(\mathscr{L}^{l-1}_{\mathfrak{g}_{\mathbb{C}}^{\ast}}(1,1))\big]^{\perp}$,
where the orthogonality is meant with respect to the inner product on $\mathscr{L}^{l}_{X}(1,1)$ induced by the $G$-left invariant metric $g$ on $X$.
Since $\alpha=\delta_{l-1}(\beta)$ is a $G$-left invariant form and so is $\delta_{l}^{\ast}\delta_{l}(\beta)$.
As a result, we have
$$
||\delta_{l-1}(\beta)||^{2}=\langle \delta_{l-1}^{\ast}\delta_{l-1}(\beta), \beta \rangle=0,
$$
which means $\alpha=\delta_{l-1}(\beta)=0$ and therefore the injectivity is proved.

It remains to show that \eqref{nil-Lie-BC-hypercoh-mor} is surjective.
Let $\alpha$ be a $\delta_{l}$-closed differential form which represents a nonzero class in $H_{BC}^{l}(X,\mathbb{C}(1,1))$.
Then $\alpha$ also represents a de Rham class in $H_{dR}^{l+1}(X; \mathbb{C})$.
Note that $\alpha$ has a unique expression
$$
\alpha=\alpha^{l,1}+\alpha^{l-1,2}+\cdots+\alpha^{2,l-1}+\alpha^{1,l}.
$$
Because of the isomorphism
$$
\iota: H_{dR}^{l+1}(\mathfrak{g}_{\mathbb{C}}^{\ast}; \mathbb{C})
\stackrel{\simeq}\longrightarrow
H_{dR}^{l+1}(X; \mathbb{C}),
$$
there is a class $[\beta]\in H_{dR}^{l+1}(\mathfrak{g}_{\mathbb{C}}^{\ast}; \mathbb{C})$ such that
\begin{equation}\label{nil-d-exact}
\alpha=\beta+d\gamma
\end{equation}
for some differntial $l$-form $\gamma$.
By comparing the types of the forms in \eqref{nil-d-exact}, we obtain
$$
\beta^{l+1,0}=-\partial \gamma^{l,0}
\;\textrm{ and }\;
\beta^{0,l+1}=-\bar{\partial}\gamma^{0,l}.
$$
Note that for a $G$-left-invariant $\bar{\partial}$-closed form $\phi$ (resp. $\partial$-closed form $\phi$),
every solution of the equation $\bar{\partial}\psi=\phi$ (resp. $\partial\psi=\phi$) is $G$-left-invariant up to a $\bar{\partial}$-exact  (resp. $\partial$-exact) term, see Step 2 in the proof of \cite[Theorem 3.7]{Ang13}.
Applying this property to our case,
we get two $G$-left-invariant forms $\gamma_{1}^{l,0}$ and $\gamma_{1}^{0,l}$ such that
$$
\gamma^{l,0}=\gamma_{1}^{l,0}+\partial u^{l-1,0}
\;\textrm{ and }\;
\gamma^{0,l}=\gamma_{1}^{0,l}+\bar{\partial} v^{0,l-1}
$$
for some $(l-1,0)$-form $u^{k-1,0}$ and $(0,l-1)$-form $v^{l-1,0}$.
Define a $G$-left-invariant form
$$
\tilde{\beta}
:=\bar{\partial}\gamma_{1}^{l,0}+\beta^{l,1}+\beta^{l-1,2}+\cdots+\beta^{2,l-1}+
\beta^{1,l}+\partial\gamma_{1}^{0,l}.
$$
From \eqref{nil-d-exact}, we have
\begin{eqnarray*}
\alpha
&=&
\tilde{\beta}+d(\gamma^{l-1,1}+\cdots+\gamma^{1,l-1})+\bar{\partial}\partial u^{l-1,0}+\partial \bar{\partial}v^{0,l-1} \\
&=& \tilde{\beta}+d(\gamma^{l-1,1}+\cdots+\gamma^{1,l-1}-\bar{\partial}u^{l-1,0}-\partial v^{0,l-1}) \\
&=& \tilde{\beta}+d\big[(\gamma^{l-1,1}-\bar{\partial}u^{l-1,0})+\gamma^{l-2,2}+
\cdots+\gamma^{2,l-2}+(\gamma^{1,l-1}-\partial v^{0,l-1})\big].
\end{eqnarray*}
It follows $d\tilde{\beta}=0$ and hence $[\alpha]=[\tilde{\beta}]$ in $H_{BC}^{l}(X, \mathbb{C}(1,1))$, and this completes the proof.
\end{proof}

We are ready to compute the bimeromorphic invariants for some complex nilmanifolds.
\begin{ex}[Iwasawa manifolds]\label{Iwasawa}
Consider the Heisenberg Lie group
$$
\mathrm{H}(3; \mathbb{C})
:=
\Bigg\{ {\tiny{\left(\begin{array}{ccc}
1 & z_{1} & z_{3}\\
0 & 1  & z_{2} \\
0 & 0 & 1
\end{array}
\right)} }\mid z_{1}, z_{2}, z_{3} \in \mathbb{C} \Bigg\}
\subset \mathrm{GL}(3; \mathbb{C}),
$$
and its discrete subgroup
$$
\mathrm{H}(3; \mathbb{Z}[\sqrt{-1}])
:=
\mathrm{GL}(3; \mathbb{Z}[\sqrt{-1}])\cap \mathrm{H}(3; \mathbb{C})
\subset \mathrm{H}(3; \mathbb{C})
$$
where $\mathbb{Z}[\sqrt{-1}]=\{a+b\sqrt{-1}\mid a,b\in \mathbb{Z}\}$ is the Gaussian integers.
Let $\mathfrak{g}$ be the Lie algebra of $\mathrm{H}(3; \mathbb{Z}[\sqrt{-1}])$.
By definition, $\mathrm{H}(3; \mathbb{C})$ is isomorphic to $\mathbb{C}^{3}$ as complex manifolds, and $\mathrm{H}(3; \mathbb{Z}[\sqrt{-1}])$ acts on $\mathrm{H}(3; \mathbb{C})$ via the left multiplication.
Such a $\mathrm{H}(3; \mathbb{Z}[\sqrt{-1}])$-action is free and properly discontinuous.
As a result, the quotient space
$$
\mathbb{I}_{3}:=\mathrm{H}(3; \mathbb{C})/ \mathrm{H}(3; \mathbb{Z}[\sqrt{-1}])
$$
is a compact 6-dimensional smooth manifold.
Moreover, there exists a canonical $\mathbb{H}(3;\mathbb{C})$-invariant complex structure $J_{0}$ on $\mathbb{I}_{3}$, such that $(\mathbb{I}_{3},J_{0})$ becomes a \emph{non-K\"{a}hler, no-formal, and holomorphically parallelizable} complex threefold, called the {\it Iwasawa manifold}.
In particular, all the de Rham, Dolbeault, Bott--Chern, and Aeppli cohomologies of $\mathbb{I}_{3}$ and its deformed objects can be computed via their Lie-algebra Dolbeault complexes.
Note that the space of $\mathrm{H}(3; \mathbb{C})$-invariant $(1,0)$-forms on $\mathrm{H}(3; \mathbb{C})$ has a basis $\{\omega^{1},\omega^{2},\omega^{3}\}$ satisfying the structure equations
$$
\begin{cases}
d\omega^{1}=0, \\
d\omega^{2}=0,\\
d\omega^{3}=-\omega^{1}\wedge\omega^{2},
\end{cases}
$$
see \cite[Section 4]{Ang13}.

On the one hand, via a straightforward computation we get the following table recording the complex dimensions of $\mathbb{H}^{l}(X,\mathbb{C}_{X}(1))$.
$$
\left.\begin{array}{|c|c|c|c|c|c|c|c|}
\hline \mathbb{H}^{\bullet}(X,\mathbb{C}(1)) &\mathbb{H}^{1} & \mathbb{H}^{2}& \mathbb{H}^{3} & \mathbb{H}^{4} & \mathbb{H}^{5} & \mathbb{H}^{6}\\
\hline \mathbb{I}_{3}  &  2 & 6 & 9 & 8 & 4 & 1 \\
\hline
\end{array}\right.
$$
This implies $\chi(\mathbb{C}_{X}(1))=0$.
On the other hand, by the computation of de Rham, Dolbeault, and Bott--Chern cohomologies of Iwasawa manifold (cf. \cite[Appendix]{Ang13}), we obtain $\chi(\mathcal{O}_{X})=0$ and therefore $\chi(\mathscr{B}_{X}(1,1))=0$ by \eqref{BCcoHnumber-chi-1}.
Due to Lemma \ref{nil-Lie-BC-hypercoh}, the Bott--Chern hypercohomology of $\mathbb{I}_{3}$ can be computed by its left-invariant forms.
Using the equality \eqref{chi-1}, we only need to compute $H_{BC}^{3}(\mathbb{I}_{3},\mathbb{C}(1,1))$, which is isomorphic to second hypercohomology group of $\mathscr{L}^{\bullet}_{\mathbb{I}_{3}}(1,1)$.
%A direct computation deduces the following table.
%$$
%\left.\begin{array}{|c|c|c|c|c|c|c|c|}
%\hline H_{BC}^{\bullet}(X, \mathbb{C}(1,1)) & H_{BC}^{1} & H_{BC}^{2}& H_{BC}^{3} & H_{BC}^{4} & H_{BC}^{5} & H_{BC}^{6}\\
%\hline \mathbb{I}_{3}  & 1 & 4 & 8 & 8 & 4 & 1 \\
%\hline
%\end{array}\right.
%$$
By definition, we obtain the associated bimeromorphic invariants of $\mathbb{I}_{3}$.
$$
\left.\begin{array}{|c|c|c|c|c|c|c|c|}
\hline  &l=1 &l=2&l=3&l=4& l=5 & l=6\\
\hline \dim\, H^{l}_{dR}(\mathbb{I}_{3},\mathbb{C}) & 4 & 8 & 10 & 8 & 4 & 1 \\
\hline \dim\, \mathbb{H}^{l}(\mathbb{I}_{3},\mathbb{C}(1))& 2 & 6 & 9 & 8 & 4 & 1 \\
\hline \dim\, H^{l}_{BC}(\mathbb{I}_{3},\mathbb{C}(1,1)) & 1 & 4 & 8 & 8 & 4 & 1 \\
\hline \spadesuit^{l} & 1 & 2 & 1 & 0 & 0 & 0 \\
\hline \clubsuit^{l} & 3 & 4 & 2 & 0 & 0 & 0 \\
\hline
\end{array}\right.
$$
Set $\omega^{r\bar{s}}=\omega^{r}\wedge\omega^{\bar{s}}$ for any $r,s\geq1$.
More precisely, we have
\begin{eqnarray*}
H^{1,1}_{\bar{\partial}}(\mathbb{I}_{3})
&=&
\langle
[\omega^{1\bar{1}}],
[\omega^{1\bar{2}}],
[\omega^{2\bar{1}}],
[\omega^{2\bar{2}}],
[\omega^{3\bar{1}}],
[\omega^{3\bar{2}}]
\rangle, \\
H^{1,2}_{\bar{\partial}}(\mathbb{I}_{3})
&=&
\langle
[\omega^{1\bar{1}\bar{3}}],
[\omega^{1\bar{2}\bar{3}}],
[\omega^{2\bar{1}\bar{3}}],
[\omega^{2\bar{2}\bar{3}}],
[\omega^{3\bar{1}\bar{3}}],
[\omega^{3\bar{2}\bar{3}}]
\rangle, \\
H^{2,2}_{\bar{\partial}}(\mathbb{I}_{3})
&=&
\langle
[\omega^{12\bar{1}\bar{3}}],
[\omega^{12\bar{2}\bar{3}}],
[\omega^{13\bar{1}\bar{3}}],
[\omega^{13\bar{2}\bar{3}}],
[\omega^{23\bar{1}\bar{3}}],
[\omega^{23\bar{2}\bar{3}}]
\rangle, \\
H^{1,1}_{BC}(\mathbb{I}_{3})
&=&
\langle
[\omega^{1\bar{1}}],
[\omega^{1\bar{2}}],
[\omega^{2\bar{1}}],
[\omega^{2\bar{2}}]
\rangle, \\
H^{1,2}_{BC}(\mathbb{I}_{3})
&=&
\langle
[\omega^{1\bar{1}\bar{2}}],
[\omega^{1\bar{1}\bar{3}}],
[\omega^{1\bar{2}\bar{3}}],
[\omega^{2\bar{1}\bar{2}}],
[\omega^{2\bar{1}\bar{3}}]
[\omega^{2\bar{2}\bar{3}}]
\rangle, \\
H^{2,2}_{BC}(\mathbb{I}_{3})
&=&
\langle
[\omega^{12\bar{1}\bar{3}}],
[\omega^{12\bar{2}\bar{3}}],
[\omega^{13\bar{1}\bar{2}}],
[\omega^{13\bar{1}\bar{3}}],
[\omega^{13\bar{2}\bar{3}}],
[\omega^{23\bar{1}\bar{2}}],
[\omega^{23\bar{1}\bar{3}}],
[\omega^{23\bar{2}\bar{3}}]
\rangle.
\end{eqnarray*}
Consequently, we obtain the kernel and cokernel of $\mathfrak{I}^{1,1}$, $\mathfrak{I}^{2,1}$ and $\mathfrak{I}^{2,2}$ as follows:
\begin{eqnarray*}
\ker(\mathfrak{I}^{1,1})
&=& 0,\quad
\ker(\mathfrak{I}^{1,2})
=
\langle
[\omega^{1\bar{1}\bar{2}}],
[\omega^{2\bar{1}\bar{2}}]
\rangle, \quad
\ker(\mathfrak{I}^{2,2})
=\langle
[\omega^{13\bar{1}\bar{2}}],
[\omega^{23\bar{1}\bar{2}}]
 \rangle; \\
\coker(\mathfrak{I}^{1,1})
&=&
\langle
[\omega^{3\bar{1}}],
[\omega^{3\bar{2}}]
\rangle,\quad
\coker(\mathfrak{I}^{1,2})
=
\langle
[\omega^{3\bar{1}\bar{3}}],
[\omega^{3\bar{2}\bar{3}}]
\rangle,\quad
\coker(\mathfrak{I}^{2,2})
=0.
\end{eqnarray*}
It is noteworthy that the numerical bimeromorphic invariant
$$
\Delta_{BC,\bar{\partial}}^{1,2}(\mathbb{I}_{3})
=h_{BC}^{1,2}(\mathbb{I}_{3})-h_{\bar{\partial}}^{1,2}(\mathbb{I}_{3})
$$
in Example \ref{Iwasawa} is zero.
However, the kernel and cokernel of $\mathfrak{I}^{1,2}$ are non-trivial.
\end{ex}

%============================================

\begin{ex}[Nilmanifold with Lie algebra $\mathfrak{h}_{6}$]
Let $M=\Gamma\setminus G$ be a complex 3-dimensional nilmanifold endowed with an invariant complex structure $J$ such that the underlying Lie algebra is isomorphic to $\mathfrak{h}_{6}$ in the classification of \cite[Theorem 2.1]{COUV16}.
A result of Ceballos--Otal--Ugarte--Villacampa \cite[Proposition 4.3]{COUV16} shows that the Fr\"{o}licher spectral sequence of $M$ degenerates at $E_{1}$-page and the Hodge symmetry holds;
however, it does not satisfy the $\partial\bar{\partial}$-lemma.
In particular, the Hodge and Bott--Chern diamonds of $M$ are:
\begin{equation*}
\begin{array}{cccccc}
\begin{matrix}
&&&&&    1  \\
&&&&  2  &&  2 \\
&&& 2 && 5 && 2\\
&&1 && 5 &&  5 && 1\\
&&& 2 && 5 && 2\\
&&&&  2  &&  2 \\
&&&&&    1
\end{matrix} &
\begin{matrix}
&&&&&    1  \\
&&&&  2  &&  2 \\
&&& 2 && 5 && 2\\
&&1 && 6 &&  6 && 1\\
&&& 2 && 6 && 2\\
&&&&  3  &&  3 \\
&&&&&    1
\end{matrix} &\\
\;\;\;\; (\textrm{Hodge}) & \;\;\;\;(\textrm{Bott--Chern})  \\[1ex]
\end{array}
\end{equation*}
Since $E_{1}\cong E_{\infty}$, from \eqref{c-dol-0}-\eqref{c-dol-1}, the hypercohomology groups of $\mathbb{C}_{M}(1)$ can be read off from the Hodge diamond via the isomorphism
\begin{equation}\label{hper-c-1}
\mathbb{H}^{l}(M,\mathbb{C}_{M}(1))\cong\mathbb{H}^{l-1}(M,\Omega^{\bullet\geq1}_{M})
\cong
\bigoplus_{s\geq 1\atop{s+t=l-1}}H^{s,t}_{\bar{\partial}}(M).
\end{equation}

Recall that $(\mathfrak{h}_{6}\otimes\mathbb{C})^{1,0}$ has a basis $\{\omega^{1},\omega^{2},\omega^{3}\}$ satisfying
$$
\begin{cases}
 d\omega^{1}=0, \\
d\omega^{2}=0,\\
d\omega^{3}=\omega^{12}+\omega^{1\bar{1}}+\omega^{1\bar{2}}.
\end{cases}
$$
Then we get a table recording the dimensions of $\mathbb{H}^{l}(M, \mathbb{C}_{M}(1))$ and the Bott--Chern hypercohomology groups.
$$
\left.\begin{array}{|c|c|c|c|c|c|c|c|}
\hline  &l=1 &l=2&l=3&l=4& l=5 & l=6\\
\hline \dim\, H^{l}_{dR}(M) &  4 & 9  & 11  & 9  & 4  & 1  \\
\hline \dim\,\mathbb{H}^{l}(M,\mathbb{C}_{M}(1))& 2 & 7 & 11 & 9 & 4 & 1 \\
\hline \dim\, H^{l}_{BC}(M,\mathbb{C}(1,1)) & 1 & 5 & 10 & 9 & 4 & 1 \\
\hline \spadesuit^{l} & 1 & 2 & 1 & 0 & 0 & 0 \\
\hline \clubsuit^{l} & 3 & 4 & 1 & 0 & 0 & 0 \\
\hline
\end{array}\right.
$$
More precisely, the generators of the cohomology groups are:
\begin{eqnarray*}
\mathbb{H}^{1}(M, \mathbb{C}_{M}(1))&=&\langle[\omega^{1}],[\omega^{2}]\rangle\\
\mathbb{H}^{2}(M, \mathbb{C}_{M}(1))
&=&
\langle
[\omega^{13}],
[\omega^{1\bar{1}}], [\omega^{1\bar{2}}], [\omega^{2\bar{1}}], [\omega^{2\bar{2}}],
[\omega^{1\bar{3}}+\omega^{3\bar{1}}], [\omega^{23}-\omega^{3\bar{1}}-\omega^{3\bar{2}}]
\rangle\\
\mathbb{H}^{3}(M, \mathbb{C}_{M}(1))
&=& \langle  [\omega^{123}], [\omega^{12\bar{1}}],  [\omega^{13\bar{1}}], [\omega^{13\bar{2}}],  [\omega^{1\bar{1}\bar{3}}],  [\omega^{2\bar{1}\bar{3}}],
[\omega^{12\bar{3}}+\omega^{23\bar{1}}], [\omega^{12\bar{3}}-\omega^{23\bar{2}}],   \\
&& [\omega^{12\bar{3}}-\omega^{2\bar{2}\bar{3}}], [\omega^{12\bar{3}}-\omega^{3\bar{1}\bar{2}}],
[\omega^{13\bar{3}}+\omega^{23\bar{3}}-\omega^{31\bar{3}}-\omega^{3\bar{2}\bar{3}}]
\rangle\\
\mathbb{H}^{4}(M, \mathbb{C}_{M}(1))
&=& \langle  [\omega^{123\bar{1}}],  [\omega^{123\bar{2}}], [\omega^{123\bar{3}}],  [\omega^{2\bar{1}\bar{2}\bar{3}}],  [\omega^{12\bar{1}\bar{3}}],
[\omega^{13\bar{1}\bar{3}}], [\omega^{23\bar{1}\bar{2}}],
 [\omega^{3\bar{1}\bar{2}\bar{3}}+\omega^{23\bar{1}\bar{3}}],  \\
&& [\omega^{13\bar{2}\bar{3}}+\omega^{23\bar{1}\bar{3}}+\omega^{23\bar{2}\bar{3}}]
\rangle \\
H_{BC}^{3}(M, \mathbb{C}_{M}(1,1))
&=& \langle  [\omega^{12\bar{1}}],  [\omega^{13\bar{1}}], [\omega^{13\bar{2}}],  [\omega^{1\bar{1}\bar{3}}],  [\omega^{2\bar{1}\bar{3}}],
[\omega^{12\bar{3}}+\omega^{23\bar{1}}], [\omega^{12\bar{3}}-\omega^{23\bar{2}}],   \\
&& [\omega^{12\bar{3}}-\omega^{2\bar{2}\bar{3}}], [\omega^{12\bar{3}}-\omega^{3\bar{1}\bar{2}}],
[\omega^{13\bar{3}}+\omega^{23\bar{3}}-\omega^{31\bar{3}}-\omega^{3\bar{2}\bar{3}}]
\rangle \\
H_{BC}^{3}(M, \mathbb{C}_{M}(1,1))
&=&\mathbb{H}^{4}(M, \mathbb{C}_{M}(1)) .
\end{eqnarray*}
Furthermore, the Dolbeault and Bott--Chern cohomologies are:
\begin{eqnarray*}
H^{1,1}_{\bar{\partial}}(M)
&=&
\langle
[\omega^{1\bar{2}}],
[\omega^{2\bar{1}}],
[\omega^{2\bar{2}}],
[\omega^{1\bar{3}}-\omega^{3\bar{2}}],
[\omega^{3\bar{1}}+\omega^{3\bar{2}}]
\rangle \\
H^{2,1}_{\bar{\partial}}(M)
&=&
\langle
[\omega^{12\bar{1}}],
[\omega^{1\bar{3}\bar{1}}],
[\omega^{12\bar{3}}+\omega^{23\bar{1}}],
[\omega^{12\bar{3}}-\omega^{23\bar{2}}],
[\omega^{13\bar{2}}]
\rangle \\
H^{2,2}_{\bar{\partial}}(M)
&=&
\langle
[\omega^{12\bar{1}\bar{3}}],
[\omega^{12\bar{2}\bar{3}}],
[\omega^{13\bar{1}\bar{3}}],
[\omega^{13\bar{2}\bar{3}}],
[\omega^{23\bar{1}\bar{3}}+\omega^{23\bar{2}\bar{3}}]
\rangle\\
H^{1,1}_{BC}(M)
&=&
\langle
[\omega^{1\bar{1}}],
[\omega^{1\bar{2}}],
[\omega^{2\bar{1}}],
[\omega^{2\bar{2}}],
[\omega^{1\bar{3}}+\omega^{3\bar{1}}]
\rangle \\
H^{2,1}_{BC}(M)
&=&
\langle
[\omega^{12\bar{1}}],
[\omega^{12\bar{2}}],
[\omega^{13\bar{1}}],
[\omega^{12\bar{3}}+\omega^{23\bar{1}}],
[\omega^{12\bar{3}}-\omega^{23\bar{2}}],
[\omega^{13\bar{2}}]
\rangle \\
H^{2,2}_{BC}(M)
&=&
\langle
[\omega^{12\bar{1}\bar{3}}],
[\omega^{12\bar{2}\bar{3}}],
[\omega^{13\bar{1}\bar{2}}],
[\omega^{13\bar{1}\bar{3}}],
[\omega^{23\bar{1}\bar{2}}],
[\omega^{13\bar{2}\bar{3}}+\omega^{23\bar{1}\bar{3}}+\omega^{23\bar{2}\bar{3}}]
\rangle.
\end{eqnarray*}
Combining with \cite[Proposition 4.3]{COUV16} derives:
\begin{eqnarray*}
\ker(\mathfrak{I}^{1,1})
&=&
\langle  [\omega^{1\bar{1}}+\omega^{1\bar{2}}]  \rangle,\quad
\ker(\mathfrak{I}^{2,1})
=\langle  [\omega^{12\bar{1}}+\omega^{12\bar{2}} \rangle \\
\ker(\mathfrak{I}^{2,2})
&=&\langle
[\omega^{13\bar{1}\bar{2}}] ,
[\omega^{12\bar{1}\bar{3}}+\omega^{12\bar{2}\bar{3}}+\omega^{23\bar{1}\bar{2}}]
\rangle \\
\coker(\mathfrak{I}^{1,1})
&=&
\langle
[\omega^{1\bar{3}}+\omega^{3\bar{2}}]
\rangle
=\langle
[\omega^{3\bar{1}}+\omega^{3\bar{2}}]
\rangle,\quad
\coker(\mathfrak{I}^{2,1})
= 0\\
\coker(\mathfrak{I}^{2,2})
&=&
\langle
[-\omega^{13\bar{2}\bar{3}}]
\rangle
=
\langle
[\omega^{23\bar{1}\bar{3}}+\omega^{23\bar{2}\bar{3}}]
\rangle.
\end{eqnarray*}
\end{ex}

\begin{ex}[Nilmanifold with Lie algebra $\mathfrak{h}_{7}$]
Let $\mathfrak{h}_{7}$ be the 6-dimensional Lie algebra in the classification of \cite[Theorem 2.1]{COUV16}.
Assume that $\mathfrak{h}_{7}$ is the Lie algebra of the universal cover $G$ of the complex 3-dimensional nilmanifold $M=\Gamma\setminus G$.
According to the work \cite{FRR19}, the Lemma \ref{nil-Lie-BC-hypercoh} holds for $M$.
Note that the space  $G$-invariant  $(1,0)$-forms admits a basis $\{\omega^{1},\omega^{2},\omega^{3}\}$ satisfying the structure equation
$$
\begin{cases}
d\omega^{1}=0, \\
d\omega^{2}=\omega^{1\bar{1}},\\
d\omega^{3}=\omega^{12}+\omega^{1\bar{2}}.
\end{cases}
$$
Following the steps in Example \ref{Iwasawa}, we obtain the following:
\begin{eqnarray*}
\mathbb{H}^{1}(M, \mathbb{C}_{M}(1))
&=&\langle[\omega^{1}]\rangle\\
\mathbb{H}^{2}(M, \mathbb{C}_{M}(1))
&=&
\langle
[\omega^{13}], [\omega^{1\bar{2}}], [\omega^{2\bar{1}}],
[\omega^{3\bar{1}}+\omega^{2\bar{2}}],
[\omega^{1\bar{3}}+\omega^{2\bar{2}}],
[\omega^{23}-\omega^{3\bar{2}}]
\rangle\\
\mathbb{H}^{3}(M, \mathbb{C}_{M}(1))
&=&
\langle
[\omega^{123}], [\omega^{12\bar{1}}], [\omega^{12\bar{3}}],
[\omega^{13\bar{1}}], [\omega^{13\bar{2}}],
[\omega^{2\bar{1}\bar{2}}], [\omega^{2\bar{1}\bar{3}}],
[\omega^{13\bar{3}}+\omega^{23\bar{2}}],\\
&&
[\omega^{23\bar{1}}+\omega^{1\bar{2}\bar{3}}],
[\omega^{23\bar{3}}+\omega^{3\bar{2}\bar{3}}],
[\omega^{3\bar{1}\bar{2}}+\omega^{1\bar{2}\bar{3}}],
[\omega^{3\bar{1}\bar{3}}+\omega^{13\bar{3}}]
\rangle\\
\mathbb{H}^{4}(M, \mathbb{C}_{M}(1))
&=&
\langle
[\omega^{2\bar{1}\bar{2}\bar{3}}],
[\omega^{12\bar{3}\bar{1}}],
[\omega^{123\bar{2}}],
[\omega^{12\bar{2}\bar{3}}],
[\omega^{12\bar{1}\bar{3}}],
[\omega^{13\bar{1}\bar{3}}],\\
&&
[\omega^{3\bar{1}\bar{2}\bar{3}}+\omega^{23\bar{1}\bar{3}}],
[\omega^{13\bar{2}\bar{3}}+\omega^{123\bar{3}}]
\rangle \\
H_{BC}^{2}(M, \mathbb{C}(1,1))
&=&
\langle
[\omega^{1\bar{1}}], [\omega^{1\bar{2}}], [\omega^{2\bar{1}}],
[\omega^{3\bar{1}}+\omega^{2\bar{2}}],[\omega^{1\bar{3}}+\omega^{2\bar{2}}]
\rangle =H_{BC}^{1,1}(M)  \\
H_{BC}^{3}(M, \mathbb{C}(1,1))
&=&
\langle
[\omega^{12\bar{1}}], [\omega^{12\bar{3}}],
[\omega^{13\bar{1}}], [\omega^{13\bar{2}}],
[\omega^{2\bar{1}\bar{2}}], [\omega^{2\bar{1}\bar{3}}],
[\omega^{13\bar{3}}+\omega^{23\bar{2}}],\\
&&
[\omega^{23\bar{1}}+\omega^{1\bar{2}\bar{3}}],
[\omega^{23\bar{3}}+\omega^{3\bar{2}\bar{3}}],
[\omega^{3\bar{1}\bar{2}}+\omega^{1\bar{2}\bar{3}}],
[\omega^{3\bar{1}\bar{3}}+\omega^{13\bar{3}}]
\rangle  \\
H_{BC}^{4}(M, \mathbb{C}(1,1))&=&\mathbb{H}^{4}(M, \mathbb{C}_{X}(1)).
\end{eqnarray*}
As a result, we get the following table of invariants.
$$
\left.\begin{array}{|c|c|c|c|c|c|c|c|}
\hline  &l=1 &l=2&l=3&l=4& l=5 & l=6\\
\hline \dim\,H_{dR}^{l}(M) & 3  & 8  & 12  &  8  &  3 &  1 \\
\hline \dim\,\mathbb{H}^{l}(M,\mathbb{C}_{M}(1))& 1 & 6 & 12 & 8 & 3 & 1 \\
\hline \dim\, H^{l}_{BC}(M,\mathbb{C}(1,1)) & 1 & 5 & 11 & 8 & 3 & 1 \\
\hline \spadesuit^{l} & 0 & 1 & 1 & 0 & 0 & 0 \\
\hline \clubsuit^{l} & 2  &  3 & 1 & 0 &  0 & 0 \\
\hline
\end{array}\right.
$$
Using the same arguments, we can compute the generators of the Dolbeault and Bott--Chern cohomologies as follows (see also \cite{AFR15}).
\begin{eqnarray*}
H_{\bar{\partial}}^{1,1}(M)
&=&
\langle
[\omega^{2\bar{1}}], [\omega^{3\bar{2}}],
[\omega^{3\bar{1}}+\omega^{2\bar{2}}],[\omega^{1\bar{3}}+\omega^{2\bar{2}}]
\rangle, \\
H_{\bar{\partial}}^{2,1}(M)
&=&
\langle
[\omega^{12\bar{1}}], [\omega^{12\bar{2}}],
[\omega^{12\bar{3}}],
[\omega^{13\bar{2}}],[\omega^{13\bar{3}}+\omega^{23\bar{2}}]
\rangle, \\
H_{\bar{\partial}}^{1,2}(M)
&=&
\langle
[\omega^{2\bar{1}\bar{2}}], [\omega^{2\bar{1}\bar{3}}],
[\omega^{1\bar{2}\bar{3}}],
[\omega^{3\bar{2}\bar{3}}],[\omega^{3\bar{1}\bar{3}}+\omega^{2\bar{2}\bar{3}}]
\rangle, \\
H_{\bar{\partial}}^{2,2}(M)
&=&
\langle
[\omega^{12\bar{1}\bar{3}}], [\omega^{12\bar{2}\bar{3}}],
[\omega^{13\bar{2}\bar{3}}],[\omega^{23\bar{1}\bar{2}}]
\rangle,\\
H_{BC}^{1,1}(M)
&=&
\langle
[\omega^{1\bar{1}}], [\omega^{1\bar{2}}], [\omega^{2\bar{1}}],
[\omega^{3\bar{1}}+\omega^{2\bar{2}}],[\omega^{1\bar{3}}+\omega^{2\bar{2}}]
\rangle, \\
H_{BC}^{2,1}(M)
&=&
\langle
[\omega^{12\bar{1}}], [\omega^{12\bar{2}}],
[\omega^{12\bar{3}}], [\omega^{13\bar{1}}],
[\omega^{13\bar{2}}],[\omega^{13\bar{3}}+\omega^{23\bar{2}}]
\rangle, \\
H_{BC}^{1,2}(M)
&=&
\langle
[\omega^{1\bar{1}\bar{2}}],
[\omega^{2\bar{1}\bar{2}}],
[\omega^{2\bar{1}\bar{3}}],
[\omega^{1\bar{2}\bar{3}}],
[\omega^{1\bar{1}\bar{3}}],
[\omega^{3\bar{1}\bar{3}}+\omega^{2\bar{2}\bar{3}}]
\rangle, \\
H_{BC}^{2,2}(M)
&=&
\langle
[\omega^{12\bar{1}\bar{3}}], [\omega^{12\bar{2}\bar{3}}],
[\omega^{13\bar{1}\bar{3}}],
[\omega^{13\bar{2}\bar{3}}],[\omega^{23\bar{1}\bar{2}}]
\rangle.
\end{eqnarray*}
To be more specific, we can compute the kernels and cokernels of $\mathfrak{C}$ as follows.
\begin{eqnarray*}
\ker(\mathfrak{C}^{2})
&=&
\langle
[\omega^{1\bar{1}}]
\rangle,\quad
\ker(\mathfrak{C}^{3})
= 0; \\
\coker(\mathfrak{C}^{2})
& = &\langle
[\omega^{13}], [\omega^{23}-\omega^{3\bar{2}}]
\rangle,\quad
\coker(\mathfrak{C}^{3})=
\langle
[\omega^{123}]
\rangle.
\end{eqnarray*}
Similarly, we have
\begin{eqnarray*}
\ker(\mathfrak{I}^{1,1})
&=&
\langle
[\omega^{1\bar{1}}], [\omega^{1\bar{2}}]
\rangle,\quad
\ker(\mathfrak{I}^{1,2})
=
\langle
[\omega^{1\bar{1}\bar{2}}],
 [\omega^{1\bar{1}\bar{3}}-\omega^{2\bar{1}\bar{2}}],
\rangle \\
\ker(\mathfrak{I}^{2,1})
&=&
\langle
[\omega^{12\bar{2}}-\omega^{13\bar{1}}]
\rangle,\quad
\ker(\mathfrak{I}^{2,2})
=
\langle
[\omega^{13\bar{1}\bar{2}}],
[\omega^{13\bar{2}\bar{3}}+\omega^{23\bar{1}\bar{2}}-\omega^{13\bar{1}\bar{3}}]
\rangle; \\
\coker(\mathfrak{I}^{1,1})
&=&
\langle
[\omega^{3\bar{2}}]
\rangle,\quad
\coker(\mathfrak{I}^{1,2})
=
\langle
[\omega^{3\bar{2}\bar{3}}]
\rangle \\
\coker(\mathfrak{I}^{2,1})
&=&  0,\quad
\coker(\mathfrak{I}^{2,2}) = 0.
\end{eqnarray*}
\begin{equation*}
\begin{array}{cccccc}
\begin{matrix}
&&&&&    1  \\
&&&&  1  &&  2 \\
&&& 2 && 4 && 2\\
&&1 && 5 && 5 && 1\\
&&& 3 && 4 && 2\\
&&&&  2  &&  2 \\
&&&&&    1
\end{matrix} &
\begin{matrix}
&&&&&    1  \\
&&&&  1  &&  1 \\
&&& 3 && 5 && 3\\
&&1 && 6 &&  6 && 1\\
&&& 2 && 5 && 2\\
&&&&  3  &&  3 \\
&&&&&    1
\end{matrix} &\\
\;\;\;\; (\textrm{Hodge}) & \;\;\;\;(\textrm{Bott--Chern})  \\[1ex]
\end{array}
\end{equation*}
\end{ex}
%================================================
\subsection{K\"{a}hler threefolds}
In this subsection, we consider two classical examples of K\"{a}hler threefolds:
smooth quintic threefolds and smooth cubic threefolds.
\begin{ex}[Quintic threefolds]
Recall that a smooth quintic threefold $X$ is a smooth hypersurface of degree $5$ in $\mathbb{CP}^{4}$,
which is a Calabi--Yau manifold since
$$\omega_{X}\cong \mathcal{O}_{X}(5-4-1)=\mathcal{O}_{X}.$$
Because $X$ is K\"{a}hler and its Hodge numbers are known (see the diagram below), the the hypercohomology groups of $\mathbb{C}_{X}(1)$ can be computed via the isomorphism \eqref{hper-c-1}.
Using the short exact sequence \eqref{example-ex-seq} and the argument above we can compute the Bott--Chern hypercohomology of $X$.
Finally, we get $\spadesuit^{k}(X)=0$ for $k\in \{2,4,5,6\}$ and $\spadesuit^{1}(X)=-1$ and $\spadesuit^{3}(X)=1$;
$\clubsuit^{k}(X)=0$ for $k\in \{2,4,5,6\}$ and $\clubsuit^{1}(X)=-1$ and $\clubsuit^{3}(X)=2$.
\end{ex}

\begin{ex}[Cubic threefolds]
Let $X$ be a smooth cubic threefold, i.e., a smooth hypersurface of degree $3$ in $\mathbb{CP}^{4}$.
This is a famous non-rational Fano threefold.
Using the same arguments, we obtain $\spadesuit^{k}(X)=\clubsuit^{k}(X)=0$ except for $\spadesuit^{1}(X)=\clubsuit^{1}(X)=-1$.
\end{ex}
\begin{equation*}
\begin{array}{cccccc}
\begin{matrix}
&&&&&    1  \\
&&&&  0  &&  0 \\
&&& 0 && 1 && 0\\
&&1 && 101 &&  101 && 1\\
&&& 0 && 1 && 0\\
&&&&  0  &&  0 \\
&&&&&    1
\end{matrix} &
\begin{matrix}
&&&&&    1  \\
&&&&  0  &&  0 \\
&&& 0 && 1 && 0\\
&&0 && 5 &&  5 && 0\\
&&& 0 && 1 && 0\\
&&&&  0  &&  0 \\
&&&&&    1
\end{matrix} &\\
\;\;\;\; (\textrm{ Hodge diamond of quintic threefold}) & \;\;\;\;(\textrm{Hodge diamond of cubic threefold})  \\[1ex]
\end{array}
\end{equation*}

%============================================================================

\appendix

\section{Relative sheaves}\label{appA}

The purpose of this appendix is to present a brief review of basic properties of relative sheaves of a compact complex manifold $X$ with respect to a closed complex submanifold.
Assume that $\imath:Z\hookrightarrow X$ is a closed complex submanifold.
Let $U:=X\setminus Z$ the complement open subset and $\jmath: U\hookrightarrow X$ the inclusion map.
Set $\mathbb{G}=\mathbb{R}$ or $\mathbb{C}$.
Then there is a short exact sequence of sheaves
\begin{equation}\label{rela-const-sheaf-exact}
\xymatrix@C=0.5cm{
0 \ar[r] & \jmath_{!}\jmath^{-1} \mathbb{G}_{X} \ar[r]^{} & \mathbb{G}_{X} \ar[r]^{} & \imath_{\ast} \imath^{-1}\mathbb{G}_{X} \ar[r] & 0.}
\end{equation}
Note that $\imath^{-1}\mathbb{G}_{X}=\mathbb{G}_{Z}$.
We call the sheaf $\mathbb{G}_{X,Z}:=\jmath_{!}\jmath^{-1} \mathbb{G}_{X}$ the {\it relative constant sheaf} with respect to the pair $(X,Z)$.
The sheaf cohomology of relative constant sheaves, compactly supported de Rham cohomologies, and relative de Rham cohomologies are isomorphic, i.e. we have
$$
H^{k}(X, \mathbb{R}_{X,Z})\cong H_{dR,c}^{k}(U; \mathbb{R})\cong H_{dR}^{k}(X,Z;\mathbb{R})
$$
and
$$
H^{k}(X, \mathbb{C}_{X,Z})\cong H_{dR,c}^{k}(U; \mathbb{C})\cong H_{dR}^{k}(X,Z;\mathbb{C}).
$$

Since $\iota: Z\hookrightarrow X$ is a closed complex submanifold of $X$,
there exist two natural surjective morphisms of sheaves on $X$
$$
\iota^{\star}: \Omega_{X}^{s}\longrightarrow \iota_{\ast}\Omega^{s}_{Z},\quad \forall\, s\in \mathbb{N},
$$
and
$$
\iota^{\star}: \mathcal{A}_{X}^{s,t}\longrightarrow \iota_{\ast}\mathcal{A}_{Z}^{s,t}, \quad \forall\, s, t\in \mathbb{N},
$$
which are induced by the pullback of differential forms.

\begin{defn}[\cite{RYY20,YY20}]\label{defn-relDol}
For any $s\in \mathbb{N}$, the kernel sheaf
$$
\mathcal{K}_{X,Z}^{s}
:=
\ker\big(\Omega_{X}^{s} \stackrel{\imath^{\star}} \longrightarrow \imath_{\ast}\Omega_{Z}^{s}\big)
$$
is called the {\it $s$-th relative Dolbeault sheaf} with respect to the pair $(X,Z)$.
For any $s, t\in \mathbb{N}$, the kernel sheaf
$$
\mathcal{K}_{X,Z}^{s,t}
:=
\ker\big(\mathcal{A}_{X}^{s,t} \stackrel{\imath^{\star}} \longrightarrow  \imath_{\ast}\mathcal{A}_{Z}^{s,t}\big)
$$
is called the {\it $(s, t)$-th relative Dolbeault sheaf} with respect to the pair $(X,Z)$.
\end{defn}

The sheaf $\mathcal{K}_{X,Z}^{s}$ is a $\mathcal{O}_{X}$-module and $\mathcal{K}_{X,Z}^{s,t}$ is a $\mathcal{C}^{\infty}_{X}$-module.
From definition, there exist two short exact sequences of sheaves on $X$:
\begin{equation}\label{exact-purerelDolsh}
\xymatrix@C=0.5cm{
  0 \ar[r] & \mathcal{K}_{X,Z}^{s} \ar[r]^{} & \Omega_{X}^{s} \ar[r]^{\imath^{\ast}} &\imath_{\ast}\Omega_{Z}^{s} \ar[r] & 0,}
\end{equation}
and
\begin{equation}\label{exact-bi-degreerelDolsh}
\xymatrix@C=0.5cm{
  0 \ar[r] & \mathcal{K}_{X,Z}^{s,t} \ar[r]^{} & \mathcal{A}_{X}^{s,t} \ar[r]^{\imath^{\ast}} & \imath_{\ast}\mathcal{A}_{Z}^{s,t} \ar[r] & 0.}
\end{equation}
Moreover, it is important to notice that the sheaf complex $\mathcal{K}_{X,Z}^{s,\bullet}$ is a fine resolution of $ \mathcal{K}_{X,Z}^{s}$ and the relative holomorphic de Rham complex $\mathcal{K}_{X,Z}^{\bullet}$ is a resolution of $\mathbb{C}_{X,Z}$
and thus we get the following relative Dolbeault theorem immediately.

\begin{lem}
There exist isomorphisms
$$
H^{k}(X, \mathcal{K}_{X,Z}^{s})
\cong
\mathbb{H}^{k}(X, \mathcal{K}_{X,Z}^{s,\bullet})
$$
and
$$
H^{k}(X, \mathbb{C}_{X,Z})
\cong
\mathbb{H}^{k}(X, \mathcal{K}_{X,Z}^{\bullet})
$$
for each integer $k\in \mathbb{Z}$.
\end{lem}

Finally, we will review the behavior of relative sheaves under blow-ups.
Assume that $\imath:Z\hookrightarrow X$ is a closed complex submanifold with complex codimension $c\geq 2$.
Let $\pi:\tilde{X}\rightarrow X$ be the blow-up of $X$ long $Z$ and $E:=\pi^{-1}(Z)$ the exceptional divisor.
Then there is a commutative diagram
\begin{equation*}\label{blowing-up-diagram}
\vcenter{
\xymatrix@=1cm{
E \ar[d]_{\rho} \ar@{^{(}->}[r]^{\tilde{\iota}} & \tilde{X}\ar[d]^{\pi}\\
 Z \ar@{^{(}->}[r]^{\iota} & X,}}
\end{equation*}
where $\tilde{\iota}: E \hookrightarrow \tilde{X}$ is the inclusion.
We set $U:=X-Z$ and $\tilde{U}:=\tilde{X}-E$ the two complement open subsets.
Consider the higher direct images of relative constant sheaves along blow-ups.
Then we have the following result.

\begin{lem}\label{direct-im-relative-constsh}
For any $t\in \mathbb{N}$, we have the following statement:
$$
R^{t}\pi_{\ast} \mathbb{G}_{\tilde{X},E}
\cong
\begin{cases}
\mathbb{G}_{X,Z},      &  t=0; \\
0,      & \text{otherwise}.
\end{cases}
$$
\end{lem}

\begin{proof}
Consider the commutative diagram
\begin{equation*}
\xymatrix@C=1cm{
\tilde{U} \ar[d]_{\pi|_{\tilde{U} }} \ar@{^{(}->}[r]^{\tilde{\jmath}} & \tilde{X} \ar[d]^{\pi}\\
U \ar@{^{(}->}[r]^{\jmath} & X.}
\end{equation*}
Since $\pi|_{\tilde{U}}$ is a biholomorphic map, it follows from the definition that the following isomorphisms hold
\begin{equation}\label{isos}
\jmath_{!}(\pi|_{\tilde{U}})_{\ast} \tilde{\jmath}^{-1} \mathbb{G}_{\tilde{X}}
\cong
\jmath_{!}(\pi|_{\tilde{U}})_{\ast} \mathbb{G}_{\tilde{U}}\cong\jmath_{!}\mathbb{G}_{U}
\cong
\jmath_{!}\jmath^{-1} \mathbb{G}_{X}
=
\mathbb{G}_{X,Z}.
\end{equation}
Observe that the higher direct images of $\mathbb{G}_{\tilde{X},E}$ along $\pi$ satisfies
\begin{equation}\label{dir-re}
R^{t}\pi_{\ast} \mathbb{G}_{\tilde{X},E}
=R^{t}\pi_{\ast} \tilde{\jmath}_{!} \tilde{\jmath}^{-1} \mathbb{G}_{\tilde{X}}
\cong \jmath_{!} R^{t}(\pi|_{\tilde{U}})_{\ast} \tilde{\jmath}^{-1} \mathbb{G}_{\tilde{X}}.
\end{equation}
Combining \eqref{isos} with \eqref{dir-re} concludes the proof.
\end{proof}

As a direct consequence of Lemma \ref{direct-im-relative-constsh}, we get

\begin{lem}\label{derived-im-relative-constsh}
There is a quasi-isomorphism
$$
\pi^{\star}:
\mathbb{G}_{X,Z}
\stackrel{\sim}\longrightarrow
R\pi_{\ast}\mathbb{G}_{\tilde{X},E}.
$$
In particular, for any $k\in \mathbb{Z}$,
there is a natural isomorphism
$$
H^{k}(X, \mathbb{G}_{X,Z})
\stackrel{\simeq}\longrightarrow
H^{k}(\tilde{X}, \mathbb{G}_{\tilde{X},E}).
$$
\end{lem}

\begin{proof}
Consider the Godement resolution $ \mathcal{G}^{\bullet}$ of the relative constant sheaf $\mathbb{G}_{\tilde{X},E}$ (cf. \cite[\S 4.3.1]{Voi02}).
Hence, $R\pi_{\ast}\mathbb{G}_{\tilde{X},E} \cong \pi_{\ast}\mathcal{G}^{\bullet}$.
By Lemma \ref{direct-im-relative-constsh},
$R^{t}\pi_{\ast} \mathbb{G}_{\tilde{X},E}=0$ for $t\geq 1$.
Since $\pi_{\ast}$ is a left exact functor,
$\pi_{\ast}\mathcal{G}^{\bullet}$ is a resolution of $\pi_{\ast}\mathbb{G}_{\tilde{X},E}$.
Observe that $\pi_{\ast}\mathbb{G}_{\tilde{X},E}\cong\mathbb{G}_{X,Z}$.
It follows that $\pi_{\ast}\mathcal{G}^{\bullet}$ gives rise to a resolution of $\mathbb{G}_{X,Z}$.
So we arrive at the conclusion that the lemma holds.
\end{proof}

\begin{rem}
Based on Lemma \ref{direct-im-relative-constsh},
one can also get the isomorphism of sheaf cohomology in Lemma \ref{derived-im-relative-constsh} by applying the Leray spectral sequence to $\mathbb{G}_{\tilde{X},E}$.
\end{rem}

Similarly, for higher direct images of relative Dolbeault sheaves,
we have the following result, see \cite[Lemma 4.4]{RYY20} or \cite[Lemma 4.2]{Men1}.

\begin{lem}\label{direct-im-relative-Dolsh}
For any $s\in \mathbb{N}$, we have
$$
\pi^{\star}:
\mathcal{K}_{X,Z}^{s} \stackrel{\simeq}\longrightarrow
\pi_{\ast} \mathcal{K}_{\tilde{X},E}^{s}
\;
\textrm{and}
\;\;
R^{t}\pi_{\ast} \mathcal{K}_{\tilde{X},E}^{s}=0\; \textrm{for}\; t\geq 1.
$$
\end{lem}

For each $p\in \mathbb{N}_{+}$, we consider the truncated relative Dolbeault sheaf complex $\mathcal{K}_{X,Z}^{\bullet <p}$:
$$
\xymatrix@C=0.5cm{
0\ar[r]^{} & \mathcal{K}_{X,Z}^{0} \ar[r]^{\partial} & \mathcal{K}_{X,Z}^{1} \ar[r]^{\partial}
\ar[r]^{} & \mathcal{K}_{X,Z}^{2} \ar[r]^{}& \cdots  \ar[r]^{\partial} & \mathcal{K}_{X,Z}^{p-1}  \ar[r]^{}& 0.}
$$

As an application of Lemma \ref{direct-im-relative-Dolsh},
we obtain

\begin{lem}\label{derived-im-relative-Dolsh}
There is a quasi-isomorphism
$$
\pi^{\star}:
\mathcal{K}_{X,Z}^{\bullet <p}
\stackrel{\sim}\longrightarrow
R\pi_{\ast}\mathcal{K}_{\tilde{X},E}^{\bullet <p}.
$$
In particular,
there exists a natural isomorphism
\begin{equation}\label{iso-hd}
\pi^{\star}:
\mathbb{H}^{k}(X,\mathcal{K}_{X,Z}^{\bullet <p})
\stackrel{\simeq}\longrightarrow
\mathbb{H}^{k}(\tilde{X},\mathcal{K}_{\tilde{X},E}^{\bullet <p}),
\end{equation}
for any $k\in \mathbb{Z}$.
\end{lem}

\begin{proof}
Notice that the sheaf complex $\mathcal{K}_{\tilde{X},E}^{s, \bullet}$ gives rise to a fine resolution of the sheaf $\mathcal{K}_{\tilde{X},E}^{s}$.
Moreover, for $0\leq s<p$, the fine sheaves $\mathcal{K}_{\tilde{X},E}^{s,t}$ forms a double complex which defines a resolution of the sheaf complex $\mathcal{K}_{\tilde{X},E}^{\bullet <p}$.
Therefore we deduce that, as sheaf complexes, the derived direct image $R\pi_{\ast}\mathcal{K}_{\tilde{X},E}^{\bullet <p}$ is isomorphic to the simple complex of the double complex $\pi_{\ast}\mathcal{K}_{\tilde{X},E}^{\bullet<p,\bullet}$.
Observe that each $\pi_{\ast}\mathcal{K}_{\tilde{X},E}^{s,t}$ is a fine sheaf on $X$ since it is a $\mathcal{C}^{\infty}_{X}$-module.
Because of the left exactness of the direct image functor $\pi_{\ast}$,
from Lemma \ref{direct-im-relative-Dolsh},
the sheaf complex $\pi_{\ast}\mathcal{K}_{\tilde{X},E}^{s, \bullet}$ gives rise to a fine resolution of the sheaf $\pi_{\ast}\mathcal{K}_{\tilde{X},E}^{s}$.
By Lemma \ref{direct-im-relative-Dolsh} again, we have a canonical isomorphism
$\mathcal{K}_{X,Z}^{s}\cong\pi_{\ast}\mathcal{K}_{\tilde{X},E}^{s}$.
So that $\pi_{\ast}\mathcal{K}_{\tilde{X},E}^{s, \bullet}$ is a fine resolution of $\mathcal{K}_{X,Z}^{s}$, and therefore the simple complex of the double complex $\pi_{\ast}\mathcal{K}_{\tilde{X},E}^{\bullet<p,\bullet}$ becomes a resolution of the sheaf complex $\mathcal{K}_{X,Z}^{\bullet<p}$.
Consequently, as sheaf complexes, $\mathcal{K}_{X,Z}^{\bullet <p}$ is quasi-isomorphic to the derived direct image $R\pi_{\ast}\mathcal{K}_{\tilde{X},E}^{\bullet <p}$.
Taking hypercohomologies of $\mathcal{K}_{X,Z}^{\bullet <p}$ and $\mathcal{K}_{X,Z}^{\bullet <p}$, the isomorphism \eqref{iso-hd} follows.
\end{proof}

\section{Derived direct images}\label{appB}
This appendix devotes to the derived direct images of the Bott--Chern complex and truncated holomorphic de Rham complex under the projective bundle morphism.
Let $Z$ be a compact complex manifold.
Assume that $\mathcal{V}$ is a holomorphic vector bundle of rank $c$ over $Z$.
Let $\rho:\PV=\PV(\mathcal{V})\rightarrow Z$ be the associated projective bundle and
$h:=c_{1}(\mathcal{O}_{\PV}(1))\in H^{2}(\PV, \mathbb{Z})$ the first Chern class of the relative tautological line bundle $\mathcal{O}_{\PV}(1)$.

We first compute the derived direct images of the Bott--Chern complex $\mathscr{B}_{\PV}^{\bullet}(p,p)$ along $\rho$.
Since $\mathscr{B}_{\PV}^{\bullet}(p,p)$ is quasi-isomorphic to the fine sheaf complex $\mathscr{L}_{\PV}^{\bullet}(p,p)[-1]$,
its derived direct image is
$$
R\rho_{\ast}\mathscr{B}_{\PV}^{\bullet}(p,p)
\cong
R\rho_{\ast}\mathscr{L}_{\PV}^{\bullet}(p,p)[-1]=\rho_{\ast}\mathscr{L}_{\PV}^{\bullet}(p,p)[-1].
$$

\begin{lem}\label{(p,p)-projbundle-formula}
For any $p\geq2$, there is a canonical quasi-isomorphism of sheaf complexes
$$
\varphi=\sum_{i=0}^{c-1} h^{i}\wedge \rho^{\star}:
\bigoplus_{i=0}^{c-1} \mathscr{L}_{Z}^{\bullet}(p-i,p-i)[-2i]
\stackrel{\sim}\longrightarrow
R\rho_{\ast}\mathscr{L}_{\PV}^{\bullet}(p,p)
$$
on $Z$.
\end{lem}
\begin{proof}
For the sake of simplicity, we set
$$
\mathscr{F}^{\bullet}_{Z}(p)=\bigoplus_{i=0}^{c-1} \mathscr{L}_{Z}^{\bullet}(p-i,p-i)[-2i].
$$
From definition, we have
$$
\mathscr{H}^{j}(\mathscr{F}^{\bullet}_{Z}(p))
= \bigoplus_{i=0}^{c-1} \mathscr{H}^{j-2i}(\mathscr{L}_{Z}^{\bullet}(p-i,p-i))
= \bigoplus_{i=0}^{c-1} \mathscr{H}^{j-2i} (\mathscr{S}_{Z}^{\bullet}(p-i,p-i)).
$$
Because of $\mathscr{H}^{j-2i} (\mathscr{S}_{Z}^{\bullet}(p-i,p-i))=0$ for all $0\leq i\leq c-1$ when $j>p+c-2$, we get $\mathscr{H}^{j}(\mathscr{F}^{\bullet}_{Z}(p))=0$ for any $j>p+c-2$.
By definition,
%we need to show that the cohomology sheaves of $\mathscr{F}^{\bullet}_{Z}(p)$ and $R\rho_{\ast}\mathscr{L}_{\PV}^{\bullet}(p,p)$ are isomorphic under the morphisms induced by $\varphi$.
it suffices to verify that the morphism of stalks $$
\mathscr{H}^{j}(\varphi_{x}):[\mathscr{H}^{j}(\mathscr{F}^{\bullet}_{Z}(p))]_{x}\longrightarrow
[\mathscr{H}^{j}(\mathscr{L}^{\bullet}_{\PV}(p,p))]_{x}
$$
is an isomorphism, for any $x\in Z$ and $j\in\mathbb{N}$.
Since the inclusion $\mathscr{S}^{\bullet}_{\PV}(p,p)\hookrightarrow\mathscr{L}^{\bullet}_{\PV}(p,p)$
is a quasi-isomorphism, there exists a canonical isomorphism of the stalks of the cohomology sheaves.
Henceforth, we identify $[\mathscr{H}^{j}(\mathscr{S}^{\bullet}_{\PV}(p,p))]_{x}$ with $[\mathscr{H}^{j}(\mathscr{L}^{\bullet}_{\PV}(p,p))]_{x}$.
For any $k\geq2$, the $j$-th cohomology sheaf of $\mathscr{L}^{\bullet}_{Z}(k,k)$ is:
\begin{equation}\label{hdi1}
\mathscr{H}^{j}(\mathscr{L}^{\bullet}_{Z}(k,k))
=
\begin{cases}
 \mathbb{C}_{Z},    &  j=0; \\
\frac{\Omega_{Z}^{k-1}}{\partial\Omega_{Z}^{k-2}}\oplus \frac{\bar{\Omega}_{Z}^{k-1}}{\overline{\partial\Omega}_{Z}^{k-2}},
& j=k-1; \\
0, & \textrm{otherwise}.
\end{cases}
\end{equation}

We divide the proof into two cases.
\paragraph{\textbf{Case 1}}
Assume that $p>c$, then $p-(c-1)\geq2$.
From \eqref{hdi1}, we get
\begin{equation}\label{hdi2}
\mathscr{H}^{j}(\mathscr{F}^{\bullet}_{Z}(p))
=
\begin{cases}
 \mathbb{C}_{Z},    &  j<p-1,\, j \,\,\textrm{is even}; \\
\mathbb{C}_{Z}\oplus\frac{\Omega_{Z}^{2p-j-2}}{\partial\Omega_{Z}^{2p-j-3}}\oplus \frac{\bar{\Omega}_{Z}^{2p-j-2}}{\overline{\partial\Omega}_{Z}^{2p-j-3}},
& p-1\leq j\leq p+c-2,\,j \,\,\textrm{is even}; \\
\frac{\Omega_{Z}^{2p-j-2}}{\partial\Omega_{Z}^{2p-j-3}}\oplus \frac{\bar{\Omega}_{Z}^{2p-j-2}}{\overline{\partial\Omega}_{Z}^{2p-j-3}},
& p-1\leq j\leq p+c-2,\,j \,\,\textrm{is odd}; \\
0, & \textrm{otherwise}.
\end{cases}
\end{equation}
The proof is carried out by induction.
Let $x$ be an arbitrary point in $Z$.
Assume that the Lemma \ref{(p,p)-projbundle-formula} holds for $p=m$, then we have an isomorphism:
\begin{equation}\label{equm}
\mathscr{H}^{j}(\varphi_{x}):
[\mathscr{H}^{j}(\mathscr{F}^{\bullet}_{Z}(m))]_{x}
\stackrel{\simeq}\longrightarrow
[\mathscr{H}^{j}(R\rho_{\ast}\mathscr{L}_{\PV}^{\bullet}(m,m))]_{x}
=
[R^{j}\rho_{\ast}\mathscr{L}_{\PV}^{\bullet}(m,m)]_{x},
\end{equation}
for any $j\in\mathbb{N}$.
Set $p=m+1$ and consider the short exact sequence:
\begin{equation}\label{projbundle-exact-seq(p,p)}
\xymatrix@C=0.5cm{
0\ar[r]^{} &
\bigl(\Omega_{\PV}^{m}\oplus \bar{\Omega}_{\PV}^{m}\bigr)[-m]
\ar[r]^{} &
\mathscr{S}^{\bullet}_{\PV}(m+1,m+1)
\ar[r]^{} &
\mathscr{S}^{\bullet}_{\PV}(m,m)
\ar[r]^{} & 0}
\end{equation}
in the category of sheave complexes of $\mathbb{C}_{\PV}$-modules.
A direct local computation shows that the following morphism of sheaf complexes is a quasi-isomorphism:
$$
%\bigoplus_{i=0}^{c-1} \Omega_{Z}^{m-i}[-i]
%\stackrel{\sim}\hookrightarrow
\varphi=\sum_{i=0}^{c-1} h^{i}\wedge \rho^{\star}:
\bigoplus_{i=0}^{c-1} \mathcal{A}_{Z}^{m-i,\bullet}[-i]
\longrightarrow
R\rho_{\ast} \mathcal{A}_{\mathbb{P}}^{m,\bullet} \cong R\rho_{\ast}\Omega_{\mathbb{P}}^{m}
$$
To be more specific, we have:
\begin{equation}\label{omghdi}
[R^{j}\rho_{\ast}(\Omega_{\PV}^{m}\oplus \bar{\Omega}_{\PV}^{m}[-m])]_{x}
%=R^{j-m}\rho_{\ast}(\Omega_{\PV}^{m}\oplus \bar{\Omega}_{\PV}^{m})
=
\begin{cases}
 0,    &  j\leq m-1; \\
\mathscr{H}^{j}(\varphi_{x})\bigl([\Omega_{Z}^{2m-j}\oplus \bar{\Omega}_{Z}^{2m-j}]_{x}\bigr),     & m\leq j\leq m+c-1; \\
0, & \textrm{otherwise},
\end{cases}
\end{equation}
see for example \cite[Remark 9]{CY19} or \cite[Lemma 11]{CY21}.
From the long exact sequence of the higher direct images of \eqref{projbundle-exact-seq(p,p)},
we obtain
\begin{equation*}\label{(p+1,p+1)-dir-im-1}
[R^{j}\rho_{\ast}\mathscr{S}^{\bullet}_{\PV}(m+1,m+1)]_{x}
\cong
[R^{j}\rho_{\ast} \mathscr{S}^{\bullet}_{\PV}(m,m)]_{x}\; \; ( j\leq m-2)
\end{equation*}
and
\begin{equation*}\label{(p+1,p+1)-dir-im-2}
[R^{j}\rho_{\ast}\mathscr{S}^{\bullet}_{\PV}(m+1,m+1)]_{x}
\cong
[R^{j}\rho_{\ast} \mathscr{S}^{\bullet}_{\PV}(m,m)]_{x}\; \; ( j\geq m+c).
\end{equation*}
Moreover, there exists a long exact sequence:
\begin{equation}\label{(p+1,p+1)-long-dir-im}
\vcenter{
\tiny{
\xymatrix@C=0.5cm{
  0 \ar[r] & [R^{m-1}\rho_{\ast}\mathscr{S}^\bullet_{\PV}(m+1,m+1)]_{x}
   \ar[r]^{} & [R^{m-1}\rho_{\ast}\mathscr{S}^\bullet_{\PV}(m,m)]_{x}
    \ar[r]^{\delta_{m-1}} &
    %\rho^{\star}(\Omega_{Z}^{m}\oplus \bar{\Omega}_{Z}^{m}) \\
    [R^{m}\rho_{\ast}(\Omega_{\PV}^{m}\oplus \bar{\Omega}_{\PV}^{m}[-m])]_{x}\\
  \ar[r]^{} & [R^{m}\rho_{\ast}\mathscr{S}^\bullet_{\PV}(m+1,m+1)]_{x}
   \ar[r]^{} & [R^{m}\rho_{\ast}\mathscr{S}^\bullet_{\PV}(m,m)]_{x}
    \ar[r]^{\delta_{m}\quad\quad} &
    %h\wedge\rho^{\star}(\Omega_{Z}^{m-1}\oplus \bar{\Omega}_{Z}^{m-1}) \\
    [R^{m+1}\rho_{\ast}(\Omega_{\PV}^{m}\oplus \bar{\Omega}_{\PV}^{m}[-m])]_{x}\\
    \ar[r] & [R^{m+1}\rho_{\ast}\mathscr{S}^\bullet_{\PV}(m+1,m+1)]_{x}
   \ar[r]^{} & [R^{m+1}\rho_{\ast}\mathscr{S}^\bullet_{\PV}(m,m)]_{x}
    \ar[r]^{\delta_{m+1}\quad\quad} &
    %h^{2}\wedge\rho^{\star}(\Omega_{Z}^{m-2}\oplus \bar{\Omega}_{Z}^{m-2})\\
    [R^{m+2}\rho_{\ast}(\Omega_{\PV}^{m}\oplus \bar{\Omega}_{\PV}^{m}[-m])]_{x}\\
    \cdots \ar[r]&[R^{m+c-2}\rho_{\ast}\mathscr{S}^\bullet_{\PV}(m+1,m+1)]_{x}
   \ar[r]^{} & [R^{m+c-2}\rho_{\ast}\mathscr{S}^\bullet_{\PV}(m,m)]_{x}
    \ar[r]^{\delta_{m+c-2}\quad\quad} &
    %h^{c-1}\wedge\rho^{\star}(\Omega_{Z}^{m-c+1}\oplus \bar{\Omega}_{Z}^{m-c+1})\\
    [R^{m+c-1}\rho_{\ast}(\Omega_{\PV}^{m}\oplus \bar{\Omega}_{\PV}^{m}[-m])]_{x}\\
    \ar[r]&[R^{m+c-1}\rho_{\ast}\mathscr{S}^\bullet_{\PV}(m+1,m+1)]_{x}
   \ar[r]^{} & [R^{m+c-1}\rho_{\ast}\mathscr{S}^\bullet_{\PV}(m,m)]_{x}
    \ar[r]^{} & 0. }}}
\end{equation}
Note that each coboundary operator in \eqref{(p+1,p+1)-long-dir-im} is induced by $\partial\oplus\bar{\partial}$.
Suppose $m$ is \emph{odd}.
On account of \eqref{hdi2} and \eqref{equm}, for all $0\leq i\leq c$, the operator $\delta_{m-1+i}$ is \emph{injective} when $i$ is odd,
and the kernel of $\delta_{m-1+i}$ is isomorphic to $[\mathbb{C}_{Z}]_{x}$ when $i$ is even.
From \eqref{omghdi} and the exactness of \eqref{(p+1,p+1)-long-dir-im}, we obtain:
\begin{equation}\label{hdi3}
\small{
[R^{m-1+i}\rho_{\ast}\mathscr{L}_{\PV}^{\bullet}(m+1,m+1)]_{x}
\cong
\begin{cases}
 [\mathbb{C}_{Z}]_{x},    &  i=0; \\
 [\frac{\Omega_{Z}^{m-i+1}}{\partial\Omega_{Z}^{m-i}}\oplus \frac{\bar{\Omega}_{Z}^{m-i+1}}{\overline{\partial\Omega}_{Z}^{m-i}}]_{x},
& i>0,\,i \,\,\textrm{is odd}; \\
[\mathbb{C}_{Z}\oplus \frac{\Omega_{Z}^{m-i+1}}{\partial\Omega_{Z}^{m-i}}\oplus \frac{\bar{\Omega}_{Z}^{m-i+1}}{\overline{\partial\Omega}_{Z}^{m-i}}]_{x},
& i>0,\,i \,\,\textrm{is even}.
\end{cases}}
\end{equation}
If $m$ is \emph{even}, by the same argument as above, we get:
\begin{equation}\label{hdi4}
\small{
[R^{m-1+i}\rho_{\ast}\mathscr{L}_{\PV}^{\bullet}(m+1,m+1)]_{x}
\cong
\begin{cases}
 0,    &  i=0; \\
 [\mathbb{C}_{Z}\oplus \frac{\Omega_{Z}^{m-i+1}}{\partial\Omega_{Z}^{m-i}}\oplus \frac{\bar{\Omega}_{Z}^{m-i+1}}{\overline{\partial\Omega}_{Z}^{m-i}}]_{x},
& i>0,\,i \,\,\textrm{is odd}; \\
[\frac{\Omega_{Z}^{m-i+1}}{\partial\Omega_{Z}^{m-i}}\oplus \frac{\bar{\Omega}_{Z}^{m-i+1}}{\overline{\partial\Omega}_{Z}^{m-i}}]_{x},
& i>0,\,i \,\,\textrm{is even}.
\end{cases}}
\end{equation}
Combining \eqref{hdi2} with \eqref{hdi3}-\eqref{hdi4} concludes an isomorphism of stalks:
\begin{equation*}
\mathscr{H}^{j}(\varphi_{x}):
[\mathscr{H}^{j}(\mathscr{F}^{\bullet}_{Z}(m+1))]_{x}
\stackrel{\simeq}\longrightarrow
[\mathscr{H}^{j}(R\rho_{\ast}\mathscr{L}_{\PV}^{\bullet}(m+1,m+1))]_{x},
%=R^{j}\rho_{\ast}\mathscr{L}_{\PV}^{\bullet}(m+1,m+1),
\end{equation*}
for any $j\in\mathbb{N}$.
This implies that Lemma \ref{(p,p)-projbundle-formula} holds for all $p>c$.

\paragraph{\textbf{Case 2}}
Assume that $p\leq c$, then $p-(c-1)\leq1$.
Note that
\begin{equation}\label{hdi5}
[\mathscr{H}^{j}(\mathscr{F}^{\bullet}_{Z}(p))]_{x}
=
\begin{cases}
 [\mathbb{C}_{Z}]_{x},    &  j<p-1,\, j \,\,\textrm{is even}; \\
[\mathbb{C}_{Z}\oplus\frac{\Omega_{Z}^{2p-j-2}}{\partial\Omega_{Z}^{2p-j-3}}\oplus \frac{\bar{\Omega}_{Z}^{2p-j-2}}{\overline{\partial\Omega}_{Z}^{2p-j-3}}]_{x},
& p-1\leq j\leq2p-3,\,j \,\,\textrm{is even}; \\
[\frac{\Omega_{Z}^{2p-j-2}}{\partial\Omega_{Z}^{2p-j-3}}\oplus \frac{\bar{\Omega}_{Z}^{2p-j-2}}{\overline{\partial\Omega}_{Z}^{2p-j-3}}]_{x},
& p-1\leq j\leq2p-3,\,j \,\,\textrm{is odd}; \\
[\mathcal{O}_{Z}+\bar{\mathcal{O}}_{Z}]_{x}, & j=2p-2;\\
[\mathbb{C}_{Z}]_{x}, & 2p-1\leq j\leq p+c-2,\,j \,\,\textrm{is odd}; \\
0, & \textrm{otherwise}.
\end{cases}
\end{equation}
Suppose the assertion in Lemma \ref{(p,p)-projbundle-formula} is true when $p=m$.
Then we have an isomorphism:
\begin{equation}\label{as-m}
\mathscr{H}^{j}(\varphi_{x}):
[\mathscr{H}^{j}(\mathscr{F}^{\bullet}_{Z}(m))]_{x}
\stackrel{\simeq}\longrightarrow
[\mathscr{H}^{j}(R\rho_{\ast}\mathscr{L}_{\PV}^{\bullet}(m,m))]_{x}
=
[R^{j}\rho_{\ast}\mathscr{L}_{\PV}^{\bullet}(m,m)]_{x},
\end{equation}
If $m+1>c$, then it follows from the result in \textbf{Case 1} that Lemma \ref{(p,p)-projbundle-formula} holds for $m+1$.
So we only need to consider the case of $m+1\leq c$.
Since $R^{j}\rho_{\ast}(\Omega_{\PV}^{m}\oplus \bar{\Omega}_{\PV}^{m}[-m])_{x}=0$ when $j\leq m-2$ or $j\geq2m+1$, we obtain
\begin{equation*}\label{(p+1,p+1)-dir-im-2.1}
[R^{j}\rho_{\ast}\mathscr{S}^{\bullet}_{\PV}(m+1,m+1)]_{x}
\cong
[R^{j}\rho_{\ast} \mathscr{S}^{\bullet}_{\PV}(m,m)]_{x}\; \;
\end{equation*}
for any $j\leq m-2$ or $j\geq 2m+1$.
As a result, there exists a long exact sequence induced by \eqref{projbundle-exact-seq(p,p)}:
\begin{equation}\label{(p+1,p+1)-long-dir-im-2}
\vcenter{
\small{
\xymatrix@C=0.49cm{
0 \ar[r] & [R^{m-1}\rho_{\ast}\mathscr{S}^\bullet_{\PV}(m+1,m+1)]_{x}
   \ar[r]^{} & [R^{m-1}\rho_{\ast}\mathscr{S}^\bullet_{\PV}(m,m)]_{x}
    \ar[r]^{\delta_{m-1}\quad} &
    %\rho^{\star}(\Omega_{Z}^{m}\oplus \bar{\Omega}_{Z}^{m}) \\
    [R^{m}\rho_{\ast}(\Omega_{\PV}^{m}\oplus \bar{\Omega}_{\PV}^{m}[-m])]_{x}\\
  \cdots \ar[r] & [R^{2m-3}\rho_{\ast}\mathscr{S}^\bullet_{\PV}(m+1,m+1)]_{x}
   \ar[r]^{} & [R^{2m-3}\rho_{\ast}\mathscr{S}^\bullet_{\PV}(m,m)]_{x}
    \ar[r]^{\delta_{2m-3}\,\,\quad} & [R^{2m-2}\rho_{\ast}(\Omega_{\PV}^{m}\oplus \bar{\Omega}_{\PV}^{m}[-m])]_{x} \\
  \ar[r]^{} & [R^{2m-2}\rho_{\ast}\mathscr{S}^\bullet_{\PV}(m+1,m+1)]_{x}
   \ar[r]^{} & [R^{2m-2}\rho_{\ast}\mathscr{S}^\bullet_{\PV}(m,m)]_{x}
    \ar[r]^{\delta_{2m-2}\,\,\quad} & [R^{2m-1}\rho_{\ast}(\Omega_{\PV}^{m}\oplus \bar{\Omega}_{\PV}^{m}[-m])]_{x}\\
    \ar[r] & [R^{2m-1}\rho_{\ast}\mathscr{S}^\bullet_{\PV}(m+1,m+1)]_{x}
   \ar[r]^{} & [R^{2m-1}\rho_{\ast}\mathscr{S}^\bullet_{\PV}(m,m)]_{x}
    \ar[r]^{\delta_{2m-1}\quad} & [R^{2m}\rho_{\ast}(\Omega_{\PV}^{m}\oplus \bar{\Omega}_{\PV}^{m}[-m])]_{x}\\
    \ar[r]&[R^{2m}\rho_{\ast}\mathscr{S}^\bullet_{\PV}(m+1,m+1)]_{x}
   \ar[r]^{} & [R^{2m}\rho_{\ast}\mathscr{S}^\bullet_{\PV}(m,m)]_{x}
    \ar[r]^{} & 0.}}}
\end{equation}

We claim that
$[\mathscr{H}^{2m}(\mathscr{F}^{\bullet}_{Z}(m+1))]_{x}=[\mathcal{O}_{Z}+\bar{\mathcal{O}}_{Z}]_{x}$
is isomorphic to
$[R^{2m}\rho_{\ast}\mathscr{S}^\bullet_{\PV}(m+1,m+1)]_{x}$
under the morphism $\mathscr{H}^{2m}(\varphi_{x})$.
Let $\star\in\{\partial,\bar{\partial},\partial\bar{\partial}\}$, and $\mathscr{Z}^{m,m}_{\PV,\star}$ the sheaf of $\star$-closed forms of bidegree $(m,m)$.
According to \cite[Lemma 3.2.1]{Ko11}, the sheaf
$\mathscr{Z}^{m,m}_{\PV,\partial}+\mathscr{Z}^{m,m}_{\PV,\bar{\partial}}$
is isomorphic to
$\mathscr{Z}^{m,m}_{\PV,\partial\bar{\partial}}$
under the inclusion morphism
\begin{equation}\label{m-closhf}
\mathscr{Z}^{m,m}_{\PV,\partial}+\mathscr{Z}^{m,m}_{\PV,\bar{\partial}}
\hookrightarrow
\mathscr{Z}^{m,m}_{\PV,\partial\bar{\partial}}
\end{equation}
and hence, for any $x\in Z$, the stalk of the higher direct image at $x$ is:
\begin{eqnarray*}
  [R^{2m}\rho_{\ast}\mathscr{S}^\bullet_{\PV}(m+1,m+1)]_{x}
  %&=& \mathscr{H}^{2m}(\rho_{\ast}\mathscr{L}^{\bullet}_{\PV}(m+1,m+1))\\
  &=& \biggl[\rho_{\ast}\bigl(\frac{\mathscr{Z}^{m,m}_{\PV,\partial\bar{\partial}}}
{\partial\mathcal{A}^{m-1,m}_{\PV}+
\bar{\partial}\mathcal{A}^{m,m-1}_{\PV}}\bigr)\biggr]_{x} \\
  &=& \biggl[\rho_{\ast}\bigl( \frac{\mathscr{Z}^{m,m}_{\PV,\partial}+
  \mathscr{Z}^{m,m}_{\PV,\bar{\partial}}}
{\partial\mathcal{A}^{m-1,m}_{\PV}+
\bar{\partial}\mathcal{A}^{m,m-1}_{\PV}}\bigr)\biggr]_{x}.
\end{eqnarray*}
Observe that there is a natural sheaf morphism
$$
h^{m}\wedge\rho^{\star}:
\mathcal{O}_{Z}+ \bar{\mathcal{O}}_{Z}
\longrightarrow\rho_{\ast}\mathscr{Z}^{m,m}_{\PV,\partial\bar{\partial}}
$$
which induces the morphism $\mathscr{H}^{2m}(\varphi)$.
Let $W$ be a small polydisc neighbourhood around $x$ such that $\PV|_{W}$ admits a canonical trivialization $\rho^{-1}(W)\cong W\times\mathbb{CP}^{c-1}$, and $W$ is $\partial$ and $\bar{\partial}$-acyclic,
i.e., has the property
$$
H^{s,t}_{\partial}(W)=0,\,\,\, H^{s,t}_{\bar{\partial}}(W)=0 \,\,\,\,\textrm{for } t>0.
$$
Then we have
\begin{eqnarray}\label{r2m}
  R^{2m}\rho_{\ast}\mathscr{S}^\bullet_{\PV}(m+1,m+1)(W)
  %&=& \mathscr{H}^{2m}(\rho_{\ast}\mathscr{L}^{\bullet}_{\PV}(m+1,m+1))(W) \nonumber\\
  %&=& \frac{\mathscr{Z}^{m,m}_{\PV,\partial\bar{\partial}}(\rho^{-1}(W))}
%{\partial\mathcal{A}^{m-1,m}_{\PV}(\rho^{-1}(W))+
%\bar{\partial}\mathcal{A}^{m,m-1}_{\PV}(\rho^{-1}(W))} \\
  &=&  \frac{\mathscr{Z}^{m,m}_{\PV,\partial}(\rho^{-1}(W))+
  \mathscr{Z}^{m,m}_{\PV,\bar{\partial}}(\rho^{-1}(W))}
{\partial\mathcal{A}^{m-1,m}_{\PV}(\rho^{-1}(W))+
\bar{\partial}\mathcal{A}^{m,m-1}_{\PV}(\rho^{-1}(W))}.\nonumber
\end{eqnarray}
Notice that
$$
H^{m,m}_{\bar{\partial}}(W\times\mathbb{CP}^{c-1})\cong
H^{m,m}_{\bar{\partial}}(\rho^{-1}(W))=
\frac{\mathscr{Z}^{m,m}_{\PV,\bar{\partial}}(\rho^{-1}(W))}
{\bar{\partial}\mathcal{A}^{m,m-1}_{\PV}(\rho^{-1}(W))}.
$$
As $W$ is acyclic for Dolbeault cohomology, it follows from the Dolbeault K\"{u}nneth formula there exists an isomorphism:
$$
\frac{\mathscr{Z}^{m,m}_{\PV,\bar{\partial}}(\rho^{-1}(W))}
{\bar{\partial}\mathcal{A}^{m,m-1}_{\PV}(\rho^{-1}(W))}\cong
h^{m}\wedge\rho^{\star}H^{0,0}_{\bar{\partial}}(W)
$$
and therefore we get
\begin{equation}\label{bar-clo-shf}
\mathscr{Z}^{m,m}_{\PV,\bar{\partial}}(\rho^{-1}(W))
=
h^{m}\wedge\rho^{\star}H^{0,0}_{\bar{\partial}}(W)+
\bar{\partial}\mathcal{A}^{m,m-1}_{\PV}(\rho^{-1}(W)).
\end{equation}
Likewise, using the K\"{u}nneth formula for $\partial$-cohomology, we can prove the following:
\begin{equation}\label{par-clo-shf}
\mathscr{Z}^{m,m}_{\PV,\partial}(\rho^{-1}(W))
=
h^{m}\wedge\rho^{\star}H^{0,0}_{\partial}(W)+
\partial\mathcal{A}^{m,m-1}_{\PV}(\rho^{-1}(W)).
\end{equation}
Combining \eqref{bar-clo-shf} with \eqref{par-clo-shf} derives:
\begin{eqnarray*}
  R^{2m}\rho_{\ast}\mathscr{S}^\bullet_{\PV}(m+1,m+1)(W)
  &=& h^{m}\wedge\rho^{\star}(H^{0,0}_{\partial}(W)+H^{0,0}_{\bar{\partial}}(W)) \\
  &\cong& (\mathcal{O}_{Z}+\bar{\mathcal{O}}_{Z})(W).
  %&\cong& \mathcal{H}_{Z}(W).
\end{eqnarray*}
This implies that the morphism of the stalks
$$
\mathscr{H}^{2m}(\varphi_{x}):[\mathcal{O}_{Z}+\bar{\mathcal{O}}_{Z}]_{x}
\longrightarrow
[R^{2m}\rho_{\ast}\mathscr{S}^\bullet_{\PV}(m+1,m+1)]_{x}
$$
is an isomorphism.

Due to \eqref{omghdi}, \eqref{hdi5}, and\eqref{as-m},  the long exact sequence \eqref{(p+1,p+1)-long-dir-im-2} equals the following one:
\begin{equation}\label{(p+1,p+1)-long-dir-im-3}
\vcenter{
\small{
\xymatrix@=0.31cm{
0 \ar[r] & [R^{m-1}\rho_{\ast}\mathscr{S}^\bullet_{\PV}(m+1,m+1)]_{x}
   \ar[r]^{} & [R^{m-1}\rho_{\ast}\mathscr{S}^\bullet_{\PV}(m,m)]_{x}
    \ar[r]^{\delta_{m-1}\quad} &
    %\rho^{\star}(\Omega_{Z}^{m}\oplus \bar{\Omega}_{Z}^{m}) \\
    [R^{m}\rho_{\ast}(\Omega_{\PV}^{m}\oplus \bar{\Omega}_{\PV}^{m}[-m])]_{x}\\
  \cdots \ar[r] & [R^{2m-3}\rho_{\ast}\mathscr{S}^\bullet_{\PV}(m+1,m+1)]_{x}
   \ar[r]^{} & \mathscr{H}^{2m-3}(\varphi_{x})\bigl([\frac{\Omega^{1}_{Z}}{\partial\mathcal{O}_{Z}}\oplus
    \frac{\bar{\Omega}^{1}_{Z}}{\overline{\partial\mathcal{O}}_{Z}}]_{x}\bigr)
    \ar[r]^{\delta_{2m-3}} &
    \mathscr{H}^{2m-2}(\varphi_{x})\bigl([\Omega_{Z}^{2}\oplus \bar{\Omega}_{Z}^{2}]_{x}\bigr) \\
  \ar[r]^{} & [R^{2m-2}\rho_{\ast}\mathscr{S}^\bullet_{\PV}(m+1,m+1)]_{x}
   \ar[r]^{} &
   \mathscr{H}^{2m-2}(\varphi_{x})\bigl([\mathcal{O}_{Z}+\bar{\mathcal{O}}_{Z}]_{x}\bigr)
    \ar[r]^{\delta_{2m-2}} &
    \mathscr{H}^{2m-1}(\varphi_{x})\bigl([\Omega_{Z}^{1}\oplus \bar{\Omega}_{Z}^{1}]_{x}\bigr) \\
    \ar[r] & [R^{2m-1}\rho_{\ast}\mathscr{S}^\bullet_{\PV}(m+1,m+1)]_{x}
   \ar[r]^{} &
   \mathscr{H}^{2m-1}(\varphi_{x})\bigl([\mathbb{C}_{Z}]_{x}\bigr)
    \ar[r]^{\delta_{2m-1}} &
    \mathscr{H}^{2m}(\varphi_{x})\bigl([\mathcal{O}_{Z}\oplus \bar{\mathcal{O}}_{Z}]_{x}\bigr)\\
    \ar[r]&
    \mathscr{H}^{2m}(\varphi_{x})\bigl([\mathcal{O}_{Z}+\bar{\mathcal{O}}_{Z}]_{x}\bigr) %R^{2m}\rho_{\ast}\mathscr{S}^\bullet_{\PV}(m+1,m+1)
   \ar[r]^{} & 0
    \ar[r]^{} & 0.}}}
\end{equation}
Note that the coboundary operator $\delta_{2m-1}=(+,-)$ is injective and for each $m-1\leq l\leq 2m-2$ the coboundary operator $\delta_{l}$ is induced by $\partial\oplus\bar{\partial}$.
Especially, the kernel of $\delta_{l}$ is isomorphic to $[\mathbb{C}_{Z}]_{x}$ when $l$ is even, and $\delta_{l}$ is injective when $l$ is odd.
By the exactness of \eqref{(p+1,p+1)-long-dir-im-3}, we get:
\begin{equation}\label{hdi99}
[R^{j}\rho_{\ast}\mathscr{S}_{\mathbb{P}}^{\bullet}(m+1,m+1)]_{x}
\cong
\begin{cases}
0,    &  j=m-1,\, j \,\,\textrm{is odd}; \\
 [\mathbb{C}_{Z}]_{x},    &  j=m-1,\, j \,\,\textrm{is even}; \\
 [\frac{\Omega_{Z}^{2m-j}}{\partial\Omega_{Z}^{2m-j-1}}\oplus \frac{\bar{\Omega}_{Z}^{2m-j}}{\overline{\partial\Omega}_{Z}^{2m-j-1}}]_{x},
& m\leq j\leq2m-1,\,j \,\,\textrm{is odd}; \\
[\mathbb{C}_{Z}\oplus\frac{\Omega_{Z}^{2m-j}}{\partial\Omega_{Z}^{2m-j-1}}\oplus \frac{\bar{\Omega}_{Z}^{2m-j}}{\overline{\partial\Omega}_{Z}^{2m-j-1}}]_{x},
& m\leq j<2m-1,\,j \,\,\textrm{is even}; \\
[\mathcal{O}_{Z}+\bar{\mathcal{O}}_{Z}]_{x}, & j=2m.
\end{cases}
\end{equation}
Comparing \eqref{hdi5} with \eqref{hdi99} derives the isomorphism:
\begin{equation*}
\mathscr{H}^{j}(\varphi_{x}):
[\mathscr{H}^{j}(\mathscr{F}^{\bullet}_{Z}(m+1))]_{x}
\stackrel{\simeq}\longrightarrow
[\mathscr{H}^{j}(R\rho_{\ast}\mathscr{L}_{\PV}^{\bullet}(m+1,m+1))]_{x},
%=R^{j}\rho_{\ast}\mathscr{L}_{\PV}^{\bullet}(m+1,m+1),
\end{equation*}
for any $j\in\mathbb{N}$, and this completes the proof.
\end{proof}

Now we are in the position to prove Proposition \ref{(p,q)-projbundle-formula}.

\begin{proof}[Proof of Proposition \ref{(p,q)-projbundle-formula}]
The idea of the proof is same as the proof of Lemma \ref{(p,p)-projbundle-formula}.
Based on Lemma \ref{(p,p)-projbundle-formula},
we fix the first degree $p$ and apply the inductive method to $q$ with $q\geq p$.
We assume that the proposition holds for $\mathscr{L}_{\PV}^{\bullet}(p,m)$ such that $m\geq p$.
The remaining thing in the proof is to show that the proposition holds for $\mathscr{L}_{\PV}^{\bullet}(p,m+1)$.
For this, we consider the short exact sequence
\begin{equation*}
\xymatrix@C=0.5cm{
0\ar[r]^{}&
\bar{\Omega}_{\PV}^{m}[-m]
\ar[r]^{} &
\mathscr{S}^{\bullet}_{\PV}(p,m+1)
\ar[r]^{} &
\mathscr{S}^{\bullet}_{\PV}(p,m)
\ar[r]^{} & 0}
\end{equation*}
in the category of complexes of $\mathbb{C}_{\PV}$-modules.
For any $x\in Z$, similar to \eqref{omghdi}, we have
$$
[R^{j}\rho_{\ast}(\bar{\Omega}_{\PV}^{m}[-m])]_{x}
%=R^{j-m}\rho_{\ast} \bar{\Omega}_{\PV}^{m}
=
\begin{cases}
 0,    &  j\leq p-1, \\
\mathscr{H}^{j}(\varphi_{x})\bigl([\bar{\Omega}_{Z}^{2m-j}]_{x}\bigr), & m\leq j\leq m+c-1;    \\
0, & \text{otherwise}.
\end{cases}
$$
Then the rest of the proof go through by using the same argument in the proof of Lemma  \ref{(p,p)-projbundle-formula}.
\end{proof}

To conclude this appendix, we establish the projective bundle formula for the truncated holomorphic de Rham complex.
\begin{prop}\label{dirim-truncated-deRham}
There exists a canonical isomorphism:
$$
\bigoplus_{i=0}^{c-1} \mathbb{H}^{k-2i}(Z,\Omega_{Z}^{\bullet <p-i})
\cong
\mathbb{H}^{k}(Z,R\rho_{\ast}\Omega_{\PV}^{\bullet <p}),
$$
for any $k\geq0$.
\end{prop}

\begin{proof}
According to the Dolbeault resolution, the sheaf
$\Omega_{Z}^{s}$ (resp. $\Omega_{\PV}^{s}$) is quasi-isomorphic to the sheaf complex $(\mathcal{A}_{Z}^{s, \bullet},\bar{\partial})$ (resp. $(\mathcal{A}_{\PV}^{s, \bullet},\bar{\partial})$ ) and therefore we get two canonical quasi-isomorphisms:
$$
\bigoplus_{i=0}^{c-1} \Omega_{Z}^{\bullet <p-i}[-2i]\simeq
\mathrm{Tot}\bigl(\bigoplus_{i=0}^{c-1}\mathcal{A}_{Z}^{\bullet <p-i, \bullet}[-i, -i], \partial, \bar{\partial}\bigr)
$$
and $R\rho_{\ast}\Omega_{\PV}^{\bullet <p}\simeq\mathrm{Tot}(\rho_{\ast}\mathcal{A}_{\PV}^{\bullet <p, \bullet},\partial, \bar{\partial})$.
%Consider the double complexes of sheaves
%$$
%\biggl(\bigoplus_{i=0}^{c-1}\mathcal{A}_{Z}^{\bullet <p-i, \bullet}[-i, -i], \partial, \bar{\partial}\biggr)
%\;\;
%\textrm{and}
%\;\;
%\biggl(\rho_{\ast}\mathcal{A}_{\PV}^{\bullet <p, \bullet},\partial, \bar{\partial}\biggr).
%$$
There is a natural morphism of double complexes
\begin{equation}\label{truncated-morph-double}
\varphi=\sum_{i=0}^{c-1} h^{i}\wedge\rho^{\star}:
\bigl(\bigoplus_{i=0}^{c-1}\mathcal{A}_{Z}^{\bullet <p-i, \bullet}[-i, -i], \partial, \bar{\partial}\bigr)
\longrightarrow
\bigl(\rho_{\ast}\mathcal{A}_{\PV}^{\bullet <p, \bullet},\partial, \bar{\partial}\bigr).
\end{equation}
Moreover, by \cite[Lemma 11]{CY21}, for any $r\geq0$ the map
$$
\varphi=\sum_{i=0}^{c-1} h^{i}\wedge\rho^{\star}:
\bigl(\bigoplus_{i=0}^{c-1}\mathcal{A}_{Z}^{r, \bullet}[-i, -i], \bar{\partial}\bigr)
\longrightarrow
\bigl(\rho_{\ast}\mathcal{A}_{\PV}^{r, \bullet}, \bar{\partial}\bigr),
$$
is a quasi-isomorphism.
This means that the morphism \eqref{truncated-morph-double} induces an $E_{1}$-isomorphism of the associated spectral sequences and hence the $E_{\infty}$-isomorphism which shows that
$$
\varphi:\mathrm{Tot}\bigl(\bigoplus_{i=0}^{c-1}\mathcal{A}_{Z}^{\bullet <p-i, \bullet}[-i, -i], \partial, \bar{\partial}\bigr)\longrightarrow
\mathrm{Tot}(\rho_{\ast}\mathcal{A}_{\PV}^{\bullet <p, \bullet},\partial, \bar{\partial})
$$
is a quasi-isomorphism.
Consequently, we are led to the conclusion that the assertion holds.
\end{proof}

\begin{rem}\label{pro-D}
We denote by $\mathscr{D}_{\PV}^{\bullet}(p,q):=\Omega_{\PV}^{\bullet<p}\oplus \bar{\Omega}_{\PV}^{\bullet<q}$.
Then we have
$$
R\rho_{\ast}\mathscr{D}_{\PV}^{\bullet}(p,q)
\cong
R\rho_{\ast}\Omega_{\PV}^{\bullet<p}\oplus
R\rho_{\ast}\bar{\Omega}_{\PV}^{\bullet<q}.
$$
Using the same argument in the proof of Proposition \ref{dirim-truncated-deRham}, we can show that there exists a canonical isomorphism:
$$
\bigoplus_{i=0}^{c-1} \mathbb{H}^{k-2i}(Z,\mathscr{D}_{Z}^{\bullet}(p-i,q-i))
\cong
\mathbb{H}^{k-2i}(Z,R\rho_{\ast}\mathscr{D}_{\PV}^{\bullet}(p,q))
$$
for all $k\geq0$.
\end{rem}

%============================================================================

\end{document}